\documentclass[10pt,reqno]{amsproc}

\usepackage[margin = 1in]{geometry}

\usepackage{comment}
\usepackage[small,bf]{caption}
\usepackage{nicefrac}
\usepackage[usenames, dvipsnames]{color}
\usepackage{hyperref}
\hypersetup
{
    colorlinks,	%
    citecolor=blue,%
    filecolor=black,%
    linkcolor=blue,%
    urlcolor=blue
}
\usepackage{amsmath}
\usepackage{amsthm}
\usepackage{amssymb}
\usepackage{amsfonts}
\usepackage{mathabx}
\usepackage{eulervm, xfrac}
\usepackage[sans]{dsfont}
\usepackage{mathdots}
\usepackage{mathrsfs}
\usepackage{rotating, setspace}
\usepackage[utf8]{inputenc}
\usepackage[english]{babel}
\DeclareMathAlphabet{\mathbfsf}{\encodingdefault}{\sfdefault}{bx}{n}
\usepackage[all]{xy}
\usepackage{tikz-cd}

\usepackage{enumerate}
\usepackage[capitalise]{cleveref}
\usepackage{subcaption}
\usepackage{algorithm}

\usepackage{algpseudocode}

\usepackage[sorting = nyt,backend=bibtex,style=numeric,maxbibnames=6]{biblatex}
\usepackage{thmtools,thm-restate}

\newcommand{\define}[1]{\textbf{#1}}

\newcommand{\calB}{\mathcal{B}}
\newcommand{\cB}{\mathcal{B}}

\newcommand{\cM}{\mathcal{M}}

\newcommand{\cI}{\mathcal{I}}
\newcommand{\calI}{\mathcal{I}}
\newcommand{\calF}{\mathcal{F}}
\newcommand{\cF}{\mathcal{F}}
\newcommand{\calG}{\mathcal{G}}

\newcommand{\Ccat}{\mathbfsf{C}}
\newcommand{\cat}{\mathbfsf{C}}
\newcommand{\pcat}{\mathbfsf{pC}}

\newcommand{\Vect}{\mathbfsf{Vect}}
\newcommand{\vect}{\mathbfsf{vect}}

\newcommand{\pVect}{\mathbfsf{pVect}}

\newcommand{\Fun}{\mathbfsf{Fun}}
\newcommand{\Top}{\mathbfsf{Top}}
\newcommand{\pTop}{\mathbfsf{pTop}}
\newcommand{\Set}{\mathbfsf{Set}}
\newcommand{\set}{\mathbfsf{set}}

\newcommand{\BCS}{\mathbfsf{Barcodes}}

\newcommand{\con}{\mathbfsf{c}}
\newcommand{\lc}{\mathbfsf{lc}}

\newcommand{\NN}{\mathbb{N}}

\newcommand{\R}{\mathbb{R}}

\newcommand{\id}{\text{id}}
\newcommand{\BC}{\text{BC}}
\newcommand{\dom}{\mathbfsf{dom}}
\newcommand{\cop}{\mathbfsf{cop}}
\newcommand{\pbc}{\mathbfsf{pbc}}

\newcommand{\e}{\epsilon}
\DeclareMathOperator{\Ima}{Im}

\newcommand{\squigrightarrow}{\rightsquigarrow}

\theoremstyle{definition}

\newtheorem{thm}{Theorem}[section]
\newtheorem{theorem}[thm]{Theorem}
\newtheorem{proposition}[thm]{Proposition}
\newtheorem{remark}[thm]{Remark}

\newtheorem{corollary}[thm]{Corollary}
\newtheorem{lemma}[thm]{Lemma}
\newtheorem{example}[thm]{Example}

\newtheorem{definition}[thm]{Definition}
\newtheorem{Algo}[thm]{Algorithm}

\definecolor{darkblue}{rgb}{0.0, 0.0, 0.8}
\definecolor{darkred}{rgb}{0.8, 0.0, 0.0}
\definecolor{darkgreen}{rgb}{0.3, 0.7, 0.5}
\definecolor{ncolor}{rgb}{0.8, 0.8, 0.0}
\definecolor{reblue}{rgb}{0.8, 0.0, 0.8}

\newcommand{\dint}{d_\mathrm{I}}

\newcommand{\Hom}{\mathrm{Hom}}
\newcommand{\cle}{\preccurlyeq}

\begin{document}

\title{Decorated Merge Trees for Persistent Topology}

\author[J. Curry]{Justin Curry}
\author[H. Hang]{Haibin Hang}
\author[W. Mio]{Washington Mio}
\author[T. Needham]{Tom Needham}
\author[O. B. Okutan]{Osman Berat Okutan}

\maketitle

\begin{abstract}
This paper introduces decorated merge trees (DMTs) as a novel invariant for persistent spaces. DMTs combine both $\pi_0$ and $H_n$ information into a single data structure that distinguishes filtrations that merge trees and persistent homology cannot distinguish alone.
Three variants on DMTs, which emphasize category theory, representation theory and persistence barcodes, respectively, offer different advantages in terms of theory and computation.
Two notions of distance---an interleaving distance and bottleneck distance---for DMTs are defined and a hierarchy of stability results that both refine and generalize existing stability results is proved here.
To overcome some of the computational complexity inherent in these distances, we provide a novel use of Gromov-Wasserstein couplings to compute optimal merge tree alignments for a combinatorial version of our interleaving distance which can be tractably estimated. We introduce computational frameworks for generating, visualizing and comparing decorated merge trees derived from synthetic and real data. Example applications include comparison of point clouds, interpretation of persistent homology of sliding window embeddings of time series, visualization of topological features in segmented brain tumor images and topology-driven graph alignment.
\end{abstract}

\setcounter{tocdepth}{1}
\tableofcontents

\section{Introduction}

In this paper we introduce a new set of tools for Topological Data Analysis (TDA) called Decorated Merge Trees (DMTs).
Not only do these new tools have a rich underlying theory that spans category theory and metric geometry, they also provide topological signatures for datasets such as point clouds, time series, grayscale images and networks which are more informative and interpretable than standard persistent homology barcodes. Figure~\ref{fig:DMT_first_example} illustrates the main construction of the paper with a simple example. 
In this figure, two point clouds with different coarse topological structure are depicted. 
Their traditional TDA signatures---degree-0 and degree-1 Vietoris-Rips persistence diagrams---do not distinguish these point clouds. 
Our DMT construction illustrates the multiscale topology of each point cloud by overlaying a merge tree (capturing multiscale connectivity) with a degree-1 persistent homology barcode. 
This depicts not only the multiscale homological ($H_1$) data of each point cloud, but also the (topological) location of each degree-1 feature in the dataset. 
This paper formalizes the DMT construction from several perspectives and extends classical lines of inquiry in the TDA literature---metric stability, decomposability and practical computational aspects---to this novel setting. 

Although the construction of the decorated merge tree presented above is intuitive, it turns out that there are multiple ways of tracing births and deaths of homological features along an evolving set of connected components. 
To this end we provide in Section \ref{sec:DMTs-3-ways} three different definitions of a decorated merge tree:
\begin{enumerate}
    \item The \define{categorical decorated merge tree} relies on the definition of a category of parameterized vector spaces $\pVect$. This definition fits squarely within the framework of generalized persistence modules \cite{bubenik2015metrics,bubenik2014categorification} as it is defined in terms of a functor from $(\R,\leq)$ to $\pVect$. This definition allows us to leverage existing results to define an interleaving distance and establish its stability.

    \item The \textbf{concrete decorated merge tree} takes the perspective that the underlying merge tree, along with its poset structure $(\cM_F,\cle)$, should define the domain of a functor to the category of vector spaces $\Vect$, where the homology of each component at each time is recorded. This definition is equivalent to the categorical one, but is more intuitive and suggests computational approaches.
    \item The \define{barcode decorated merge tree} takes a perspective similar to other TDA techniques~\cite{landi2018rank,lesnick2015interactive,turner2014persistent} that reduce the study of complicated persistent spaces to ensembles of 1-dimensional persistence modules. The barcode DMT associates to each point in the merge tree the barcode gotten by restricting the filtration to the line that starts at that point and stretches to infinity.
\end{enumerate}

\begin{figure}
    \centering
    \includegraphics[width = 0.8\textwidth]{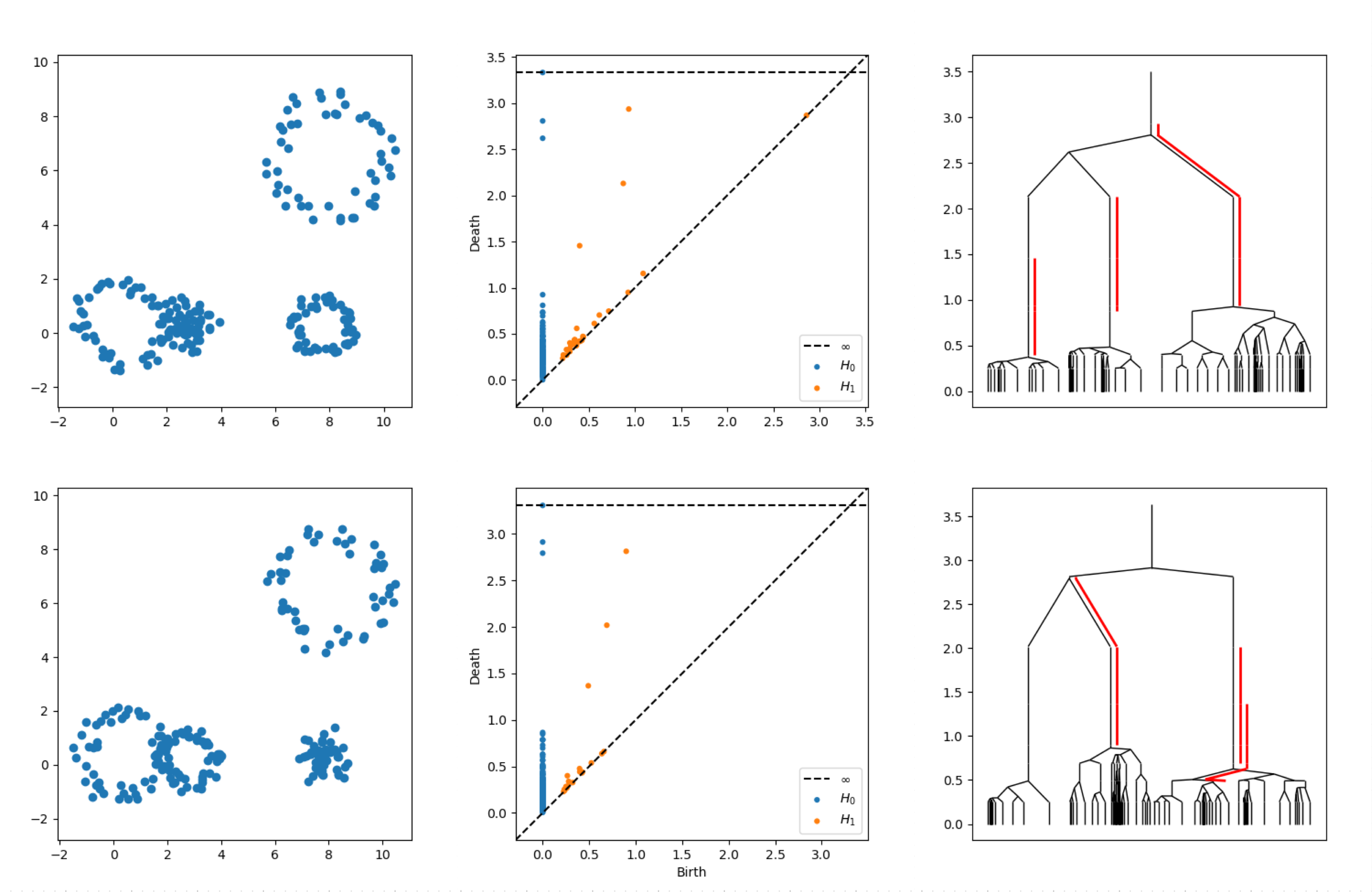}
    \caption{Decorated Merge Trees. The left column shows two point clouds. Their degree-0 and degree-1 Vietoris-Rips persistence diagrams in the middle column are essentially the same, despite clear topological differences in the point clouds. The decorated merge trees (DMTs) in the third column clearly distinguish the datasets topologically by fusing degree-0 and degree-1 information to track the topological \emph{location} of the degree-1 features. Each red bar corresponds to a degree-1 persistent feature and its placement in the merge tree indicates the connected components over which the feature persists.}
    \label{fig:DMT_first_example}
\end{figure}

In Section \ref{sec:stability-of-DMTs}, we introduce  an interleaving distance between decorated merge trees and establish several bounds that set decorated merge trees apart from existing TDA methods.  In particular, let $f,g:X\to \R$ be functions whose sublevel sets are locally connected. For every homological degree $n$, one obtains categorical decorated merge trees $\tilde{F}_n$ and $\tilde{G}_n$, as well as classical merge trees and persistence modules, associated to the sublevel set filtrations of $f$ and $g$. Theorems~\ref{thm:hierarchy-stability} and~\ref{thm:cat-DMT-more-sensitive-than-PH} allow one to extract the following statement:

\smallskip
\noindent {\bf Theorem.} 
The interleaving distance between $\tilde{F}_n$ and $\tilde{G}_n$  is stable with respect to $L^{\infty}$-distance between $f$ and $g$, and is more sensitive than (i.e., lower bounded by) both the interleaving distance between the associated merge trees and the interleaving distance between the associated persistence modules.
\smallskip

After these stability results are proven, we focus on barcode decorated merge trees. 
We show in Theorem~\ref{thm:continuity_barcode_DMT} that a barcode decorated merge tree can be understood as a Lipschitz map from the merge tree to the space of barcodes.
Barcode decorated merge trees are amenable to a theory of matchings (Definition~\ref{defn:matchings_DMTs}) and thus a new decorated bottleneck distance (Definition~\ref{defn:decorated-bottleneck-distance}). 
This offers a tractable and approximable metric for comparing these enriched invariants. 
Theorem~\ref{thm:hierarchy-stability} shows that this matching distance is stable with respect to the interleaving distance between categorical decorated merge trees.

In Section \ref{sec:representations}, representation-theoretic aspects of DMTs take center stage. 
In general, one cannot hope for simple indecomposables such as the barcode decompositions that appear in standard persistent homology---see Example \ref{ex:indecomposable-macaroni} for an illustration. 
We say that a DMT is \emph{real interval decomposable} if it decomposes as a direct sum of DMTs with totally ordered support.
Theorem \ref{thm:untwisted} provides a condition which is equivalent to real interval decomposability. 
The class of real interval decomposable DMTs is of particular interest. 
As Theorem~\ref{thm:injective-barcode-transform} shows, on this class the map taking a DMT to a barcode DMT is injective, thus providing a positive solution to a \emph{topological inverse problem}~\cite{oudot2020inverse}. 
We also give methods for generating real interval decomposable DMTs directly from data, with theoretical guarantees of their correctness (Proposition \ref{prop:disjointness}).

Section \ref{sec:computational_aspects} shifts the focus to computational aspects of interleaving distance between decorated merge trees. Building on work of Gasparovic, et al.~\cite{gasparovic2019intrinsic}, we reformulate computation of the matching distance between barcode decorated merge trees as the search for an alignment between nodes of the trees which is optimal with respect to a certain cost function (Proposition \ref{prop:labeled_distance}). 
This reformulation allows us to introduce a method for estimating the metric via a continuous relaxation which can be solved within the Gromov-Wasserstein framework from optimal transport theory  \cite{memoli2011gromov}. 
Our algorithm is novel even when estimating the interleaving distance between (undecorated) merge trees, but has close connections to other recent advances in the literature \cite{li2021sketching,memoli2021ultrametric,memoli2019gromov}. 
The computational focus is continued in Section \ref{sec:examples}, where algorithms for computing and visualizing decorated merge trees from synthentic and real datasets are described. 
The Python code used to produce the figures and experiments for the paper are publicly available under an open source license at \url{https://github.com/trneedham/Decorated-Merge-Trees}.

The main paper concludes with a discussion of future directions of research in Section~\ref{sec:discussion}. 
In particular, we note that the DMT concept has natural generalizations such as Reeb or MAPPER graphs decorated with zig-zag modules~\cite{carlsson2010zigzag}, correspondence modules~\cite{hang2020correspondence} or Leray (co)sheaves~\cite{curry2015topological}.
These constructions will be the subject of future work as they require more substantative theoretical and algorithmic developments.
By contrast, DMTs fit naturally into a pre-existing body of literature and a fleshed out code base.
Regardless, decorated merge trees and decorated Reeb graphs are just a small part of a broader research program to construct enriched TDA invariants that are more informative than classical barcodes.

The paper includes three appendices. 
We draw readers' attention to Appendices~\ref{sec:interval-topology} and~\ref{sec:existence_of_merge_tree}, which contain results on merge tree topology which are theoretically fundamental but fall outside of the narrative of the main body of the paper. Appendix~\ref{sec:technical_proofs} contains proofs of some technical results.

\subsection*{Acknowledgements}

JC would like to thank Rachel Levanger for discussions dating back to 2017 when the module-theoretic and lift-theoretic approaches to DMTs were first considered. 
JC would also like to thank Gabriel Bainbridge for teaching him about the category of parameterized objects during the summer of 2020. Gabe's use of the parameterized category is set to appear in~\cite{bainbridge2021}.
Finally, JC would like to acknowledge NSF Grant CCF-1850052 and NASA Contract 80GRC020C0016 for supporting his research.
WM acknowledges partial support by NSF grant DMS-1722995. TN would like to thank Facundo M\'{e}moli for useful feedback on an earlier draft of the paper.
HH would like to acknowledge NSF grant DMS-1854683.

\section{Decorated Merge Trees Three Different Ways}\label{sec:DMTs-3-ways}

In this paper we investigate topological signatures which go beyond standard persistent homology of filtered topological spaces. 
To this end, we formally define three notions of decorated merge trees---\emph{categorical}, \emph{concrete} and \emph{barcode-decorated}---which were described informally in the introduction. 
Before doing so, we review some preliminary definitions.

\subsection{Preliminaries}

We now introduce the general class of objects to which our decorated merge tree constructions will be applied.

\begin{definition}\label{defn:persistence-space}
A \define{persistent space} is a functor $F:(\R,\leq) \to \Top$, where $(\R,\leq)$ is considered as a poset category. Concretely, a persistent space associates to each $s\in \R$ a topological space $F(s)$ and to each ordered pair $s\leq t$ a continuous map $F(s\leq t):F(s) \to F(t)$. These maps collectively satisfy the usual composition rules of a functor.
When each map $F(s\leq t)$ is an injection, we will call such a persistent space a \define{filtration}.
\end{definition}

We now describe several situations where persistent spaces arise, focusing on the case of filtrations.

\begin{example}
Given a continuous function $f:X\to \R$, one can consider the \define{sublevel-set filtration}, which is a persistent space where
$$F(s):=f^{-1}(-\infty,s]=\{x\in X \mid f(x) \leq s\}.$$
By equipping each sublevel-set with the subspace topology, the inclusion maps $F(s)\subseteq F(t)$ are continuous and define the maps $F(s\leq t)$ for the persistent space. 
\end{example}

The following gives a flexible class of instances of sublevel-set filtrations.

\begin{example}
Given a subset $Z$ of a metric space $X$, we define the \define{offset function} $f_Z : X \to \R$ to be
\[
    f_Z(x):=\inf_{z\in Z} d(x,z).
\]
The \define{offset filtration} $F_Z$, defined by
\[
F_Z(s):= \{x\in X \mid d(x,Z):=\inf_{z\in Z} d(x,z) \leq s\},
\]
is the sublevel-set filtration of $f_Z$.
\end{example}

Sublevel-set filtrations can be extended to certain non-continuous functions on simplicial complexes which arise frequently in topological data analysis.

\begin{example}\label{ex:sublevel_set_simplicial}
Let $X$ be a simplicial complex and $f:X \to \R$ a function which is constant on each simplex. 
If $f$ is \define{monotone}, that is whenever $\sigma$ is a face of $\tau$ then $f(\sigma)\leq f(\tau)$, then the sublevel-set filtration $F(s) = |f^{-1}(-\infty,s]|$ defines a persistent space, where $|\cdot|$ denotes geometric realization.
\end{example}

We now introduce some basic concepts of Topological Data Analysis (TDA). 
These are invariants built to study persistent spaces. 
We assume that the reader is familiar with the fundamentals of TDA and mainly use the definitions here to set terminology and notation---see the survey \cite{carlsson2014topological} for more background.
For the remainder of this paper we let $\pi_0$ denote the connected components functor $\pi_0: \Top \to \Set$ and let $H_n$ denote some choice of homology theory with coefficients in a field $\Bbbk$.

\begin{definition}\label{defn:generalized-MT}

A functor $S:(\R,\leq) \to \Set$ is called a \define{persistent set}~\cite{carlsson2013classifying,curry2018fiber}.
If $F:(\R,\leq)\to\Top$ is a persistent space, then the \define{persistent set of components} is the composition of $F:(\R,\leq) \to \Top$ with the connected components functor $\pi_0:\Top \to \Set$, i.e.
\[
    \pi_0 \circ F :(\R,\leq) \to \Set \qquad \text{where} \qquad s \quad \squigrightarrow \quad \pi_0(F(s)).
\]
Associated to any persistent set $S:(\R,\leq) \to \Set$ is its \define{display poset},
which is simply the disjoint union of all the sets that appear in $S$, i.e.
\[
    \mathcal{S}:=\bigsqcup_{t\in \R} S(t) := \bigcup S(t)\times \{t\}
\]
The poset structure on $\mathcal{S}$ is defined by declaring
\[
    (x,s) \cle (y,t) \qquad \text{if and only if} \qquad S(s\leq t)(x)=y.
\]
The \define{generalized merge tree} of $F$ is the display poset $(\cM_F,\cle)$ associated to the persistent set $\pi_0\circ F$.
\end{definition}

We use the term ``generalized'' to distinguish it from the classical merge tree.

\begin{definition}\label{defn:merge-tree-as-reeb-graph}
Let $X$ be a topological space and $f:X \to \R$ a continuous function. The \define{epigraph of $f$} is the set $E_f := \{(x,r) \mid f(x) \leq r\}$. 
The \define{classical merge tree} $\cM_f$ is the Reeb graph of the \define{projection function} $\pi_f:E_f \to \R:(x,r) \mapsto r$.
That is, the merge tree is the quotient space $\cM_f := E_f/\sim$, where $\sim$ is the equivalence relation $p = (x,r) \sim (x',r') = p'$ if and only if $r = r'$ and $p$ and $p'$ lie in the same connected component of the level set $\pi_f^{-1}(r)$. 
Let $\tilde{\pi}_f:\cM_f \to \R$ denote the projection map induced by $\pi_f$.
\end{definition}

\begin{remark}
The differences between the classical and generalized merge tree have gone largely uncommented on in the literature and are often treated as interchangeable.
Observe that $\cM_F$ is defined for any persistent space $F$ and comes endowed with a poset structure. 
On the other hand, $\cM_f$ is defined for a continuous function $f$ on a space $X$ and comes equipped with a quotient topology.
This difference only creates problems when one notices that there are two notions of interleaving, which might not agree always.
In Appendix~\ref{sec:interval-topology} we prove that for functions on compact spaces with finitely many critical points, this difference can be ignored.
\end{remark}

Beyond using $\pi_0$ to discriminate persistent spaces, we can use homology. 

\begin{definition}
A \define{persistence module} is a functor from the poset category $(\R,\leq)$ into $\Vect_{\Bbbk}$, the category of vector spaces over the field $\Bbbk$. A persistence module is \define{pointwise finite dimensional} if its image lies in $\vect_{\Bbbk}$, the category of finite dimensional vector spaces over $\Bbbk$.

Let $F:(\R,\leq) \to \Top$ be a persistent space. For any non-negative integer $n\geq 0$ the \define{$n^{th}$ persistent homology module} $F_n:(\R,\leq) \to \Vect_{\Bbbk}$ is the persistence module
\[
    F_n := H_n \circ F : (\R,\leq) \to \Vect_{\Bbbk} \qquad \text{with} \qquad s \quad \squigrightarrow \quad  H_n(F(s);\Bbbk).
\]
We will frequently drop the subscript $\Bbbk$, with the understanding that a field of coefficients has been fixed. 
\end{definition}

One of the central results in the theory of TDA is the theoreom of Crawley-Boevey \cite[Theorem 1.1]{crawley2015decomposition} which states that pointwise finite-dimensional persistence modules always decompose into direct sums of simple indecomposables. We introduce the relevant terminology and notation below.

\begin{definition}
Let $I \subset \R$ be an interval. The \define{interval module associated to $I$} is the persistence module $\Bbbk_I:(\R,\leq) \to \vect_{\Bbbk}$ with $\Bbbk_I(s) = \Bbbk$ if and only if $s \in I$ and otherwise $\Bbbk_I(s)$ is the zero vector space. We define $\Bbbk_I(s \leq t) = \mathrm{id}_{\Bbbk}$ if and only if $s,t \in I$ and otherwise $\Bbbk_I(s \leq t)$ is the zero map.

Crawley-Boevey's theorem says that any pointwise finite-dimensional persistence module is isomorphic to a direct sum of interval modules, and that this representation is unique up to permuting factors. A \define{barcode} $B=\{(I,m_I)\}$ is a multiset of intervals in the real line, i.e. $I\subseteq \R$ is an interval and $m_I\in \NN$ indicates its multiplicity. It follows from the discussion above that any pointwise finite dimensional $\R$-module $F$ has a uniquely associated barcode $B(F)$. Let $\BCS$ denote the set of all barcodes.

The existence of barcode representations of persistence modules allows for various methods of visualization and analysis of the multiscale topology of a filtered space. A barcode can be represented visually by drawing the collection of intervals in the plane (see the righthand column of Figure \ref{fig:two-subsets}) or as a multiset of points in the plane called a \define{persistence diagram} (see the righthand column of Figure \ref{fig:DMT_first_example}); here, the endpoints of each interval are plotted as an ordered pair. 
\end{definition}

We now consider the simple example illustrated in Figure \ref{fig:two-subsets} to see these concepts in action and to motivate the definitions introduced below. The spaces $X,Y \subset \R^2$ give rise to persistent spaces via their respective offset filtrations. Despite the fact that $X$ and $Y$ are topologically distinct, the degree-0 and degree-1 persistent homology barcodes extracted from these persistent spaces are the same. This brings us to the goal of defining richer topological signatures (decorated merge trees),  which are able to track \emph{interactions} of topological features in a persistent space. 

\begin{figure}
  \includegraphics[width=0.8\linewidth]{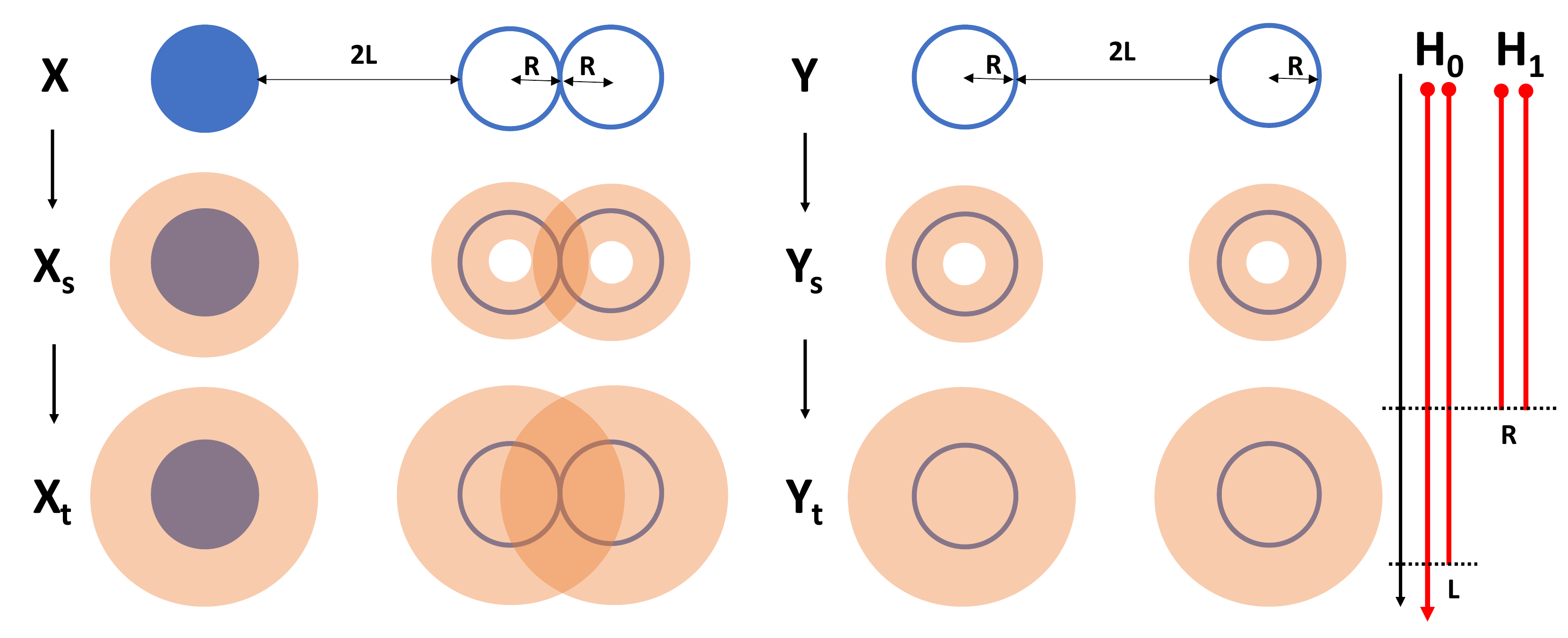}
  \caption{A motivating example for the need for decorated merge trees. Two subsets $X$ and $Y$ of $\R^2$ along with two times in their offset filtration are shown. Degree-0 and degree-1 persistent homology fails to distinguish them as the number of components and the number of holes are the same across all stages in the filtration. This is witnessed by their identical persistent homology barcodes, shown to the right.}
  \label{fig:two-subsets}
\end{figure}

\subsection{The Categorical Decorated Merge Tree}

The first definition of a decorated merge tree will be given as a purely category-theoretic construction. 
This perspective allows for streamlined proofs of stability theorems below, but has the downside that the connection to merge trees may not be as transparent. 
This is remedied with alternative constructions in the following subsections.

The key observation for defining the categorical decorated merge tree is that homology decomposes into a direct sum of the homology of each connected component.
This is expressed in the well known fact that
\[
    \text{if} \qquad X \cong \bigsqcup X_i:=\bigcup X_i \times \{i\} \qquad \text{then} \qquad H_n(X) \cong \bigoplus_i H_n(X_i).
\]
Moreover, if $f:X \to Y$ is a continuous map of spaces then we can view $f$ as a map between disjoint unions that send each factor in the domain to a unique factor in the range.
In other words, the continuous map
\[
    f= \sqcup f_{i} : \bigsqcup_{i\in \pi_0(X)} X_i \to \bigsqcup_{j\in \pi_0(Y)} Y_j
\]
can be \emph{parameterized} by the underlying map of sets $\pi_0(f) \colon \pi_0(X) \to \pi_0(Y)$.
This indicates that we can parameterize maps inside of a persistent space along its associated persistent set of connected components.
This requires some further restrictions on properties of the topological spaces involved such as local connectedness.
First we isolate an important categorical construction.

\begin{definition}\label{defn:cat-of-parameterized-objects}
Let $\cat$ be a category. The \define{category of discretely parameterized objects} in $\cat$, written $\pcat$, has for objects functors $I: S \to \cat$ where $S$ is a set viewed as a discrete category, i.e.~the only morphisms in $S$ are identity morphisms.
The functor $I$ amounts to a choice of object of $\cat$ for each $s\in S$.
We will refer to such a functor as an \define{$S$-parameterized object}.
A morphism from an $S$-parameterized object $I:S \to \cat$ to a $T$-parameterized object $J: T \to \cat$ consists of a map of sets $m:S \to T$ and a natural transformation from the functor $I$ to the pullback of $J$ along $m$, i.e. a morphism is a natural transformation $\alpha: I \Rightarrow m^*J$ where $m^*J:=J\circ m$.
\end{definition}

We note that if $\cat$ has coproducts, then $\pcat$ participates in the following diagram of categories and functors:
\begin{equation*}
\begin{tikzcd}
& \pcat \ar[dr, "\cop"]  \ar[dl, "\dom"'] & \\
\Set &  & \cat
\end{tikzcd}
\end{equation*}
The functor $\dom$ sends any $S$-parameterized object $I:S \to \cat$ to the underlying parameterizing set $S$.
The functor $\cop$ sends the diagram $I:S \to \cat$ to its colimit, which is the coproduct in this case. 
Before exploiting the above diagram further, we state the result that was used implicitly at the outset of this subsection.

\begin{lemma}\label{lem:loc-connected}
Let $\Top^{\con}$ denote the category of connected and locally connected topological spaces.
Let $\Top^{\lc}$ denote the category of locally connected spaces.
The coproduct functor induces an equivalence between these categories:
\[
    \cop:\pTop^{\con} \to \Top^{\lc} \qquad \text{where} \qquad I:S \to \Top^{\con} \quad \squigrightarrow \quad \bigsqcup_{s\in S} I(s).
\]
\end{lemma}

\begin{proof}

To complete the proof, it suffices to show that $\cop:\pTop^{\con} \to \Top^{\lc}$ is full, faithful and essentially surjective~\cite[Thm. 1.5.9]{riehl2017category}.
The essentially surjective property is true by virtue of the fact that every locally connected space is naturally homeomorphic to the coproduct (i.e., disjoint union) of its components.
\emph{Full} and \emph{faithful} mean that if $I:S \to \Top^{\con}$ and $J:T \to \Top^{\con}$ are two parameterized connected spaces, then the map 
\[
    \Hom_{\pTop^{\con}}(I,J) \to \Hom_{\Top^{\lc}}(\bigsqcup_{s\in S} I(s), \bigsqcup_{t\in T} J(t) )
\]
is surjective and injective, respectively.
To show surjectivity (fullness), we have to show that every continuous map
\[
    f: \bigsqcup_{s\in S} I(s) \to \bigsqcup_{t\in T} J(t)
\]
is realized by some morphism $(m,\alpha)$ in $\pTop^{\con}$.
Here connectivity of each $I(s)$ is an essential part of the hypothesis because it allows us to associate to each $s\in S$ a unique $t\in T$ so that $f(I(s))\subseteq J(t)$.
This specifies the map of sets $m:S \to T$.
The restriction of the continuous map $f$ to each $I(s)$ specifies the components of a natural transformation $\alpha:I \Rightarrow m^*J$.
Injectivity is not difficult to see because if $(m,\alpha)$ and $(n,\beta)$ are two morphisms that induce the same map between the disjoint unions, then set-theoretically they are equal as well. Recalling the set-theoretic definition of the disjoint union, this means that
\[
    \sqcup \alpha_s =\sqcup \ell_s : \bigcup_{s\in S} I(s)\times \{s\} \to \bigcup_{t\in T} J(t)\times \{t\}
\]
and in particular that $m=n$ and $m^*\alpha=n^*\beta$.
\end{proof}

One consequence of Lemma~\ref{lem:loc-connected} is that we can define another functor that serves as sort of ``inverse" to $\cop$, up to natural isomorphism.

\begin{definition}
The \define{parameterized by components} functor
\[
    \pbc: \Top^{\lc} \to \pTop^{\con}
\]
takes each locally connected space $X$ to the object $I:\pi_0(X) \to \Top^{\con}$ which takes the label $i\in \pi_0(X)$ for an equivalence class to the underlying subset of $X$ carved out by this equivalence class, equipped with the subspace topology.
A map of spaces $f:X \to Y$ is taken to the morphism $(m,\alpha)$ where $m=\pi_0(f)$ is the map recording which connected component of $X$ maps to which component of $Y$ and $\alpha$ is the natural transformation that records the restriction of $f$ to each component.
\end{definition}

By virtue of~\cite[Def. 1.5.4]{riehl2017category}, an alternative proof to Lemma~\ref{lem:loc-connected} is that
\[
    \pbc \circ \cop \cong \id_{\pTop^{\lc}} \qquad \text{and} \qquad \cop \circ \pbc \cong \id_{\Top^{\lc}}.
\]

We leverage the above identities to provide our first refinement of persistent spaces into functors from $(\R,\leq) \to \pTop^{\con}$.
This is the heart of the definition of a categorical decorated merge tree.

\begin{lemma}\label{lem:persistent-space-factors}
Any persistent space $F: (\R,\leq) \to \Top^{\lc}$ has an associated \define{persistently parameterized space} 
\[
    \tilde{F}:=\pbc \circ F: (\R,\leq) \to \pTop^{\con}.
\]
The functor $\tilde{F}$ fits into the following diagram, which commutes up to natural isomorphism.
\begin{equation*}
\begin{tikzcd}
& (\R,\leq) \ar[d,"\tilde{F}"] \ar[ddl,"\pi_0\circ \tilde{F}"'] \ar[ddr,"F"] \ar[ddr,"\cong"'] & \\
& \pTop^{\con} \ar[dr, "\cop"']  \ar[dl, "\dom"] & \\
\Set &  & \ar[ll,"\pi_0"] \Top^{\lc}
\end{tikzcd}
\end{equation*}
\end{lemma}
\begin{proof}

The natural isomorphism $\cop\circ \pbc \cong \id_{\Top^{\lc}}$ from the remark above can be restricted to the image of $F$ to yield
\[
    \cop\circ \pbc \circ F \cong F \qquad \Leftrightarrow \qquad  \cop \circ \tilde{F} \cong F.
\]
Explicitly this means that for every $s\in \R$ the spaces $\cop \circ \tilde{F}(s)$ and $F(s)$ are homeomorphic.
Since homeomorphic spaces have isomorphic sets of components, we know that the persistent sets $\pi_0\circ F$ and $\dom \circ \tilde{F}$ are naturally isomorphic as well.
\end{proof}

We are now able to define the categorical decorated merge tree of a persistent (locally connected) space.

\begin{definition}\label{defn:param-DMT}
Let $F:(\R,\leq) \to \Top^{\lc}$ be a persistent space where every space is locally connected and let $\tilde{F}:(\R,\leq) \to \pTop^{\con}$ denote the persistently parameterized space from Lemma~\ref{lem:persistent-space-factors}.
The \define{categorical decorated merge tree in degree $n$} is the functor
\[
    \tilde{F}_n := H_n \circ \tilde{F}: (\R,\leq) \to \pVect \qquad \text{where} \qquad s \squigrightarrow I(s): \pi_0(F(s)) \to \Vect.
\]
The functor $I(s):\pi_0(F(s)) \to \Vect$ is a parameterized vector space (in the sense of Definition~\ref{defn:cat-of-parameterized-objects}) that sends each component index $i\in \pi_0(F(s))$ to the homology vector space of that component, i.e.~$H_n(F(s)_i)$.
\end{definition}

\begin{figure}
  \includegraphics[width=0.8\linewidth]{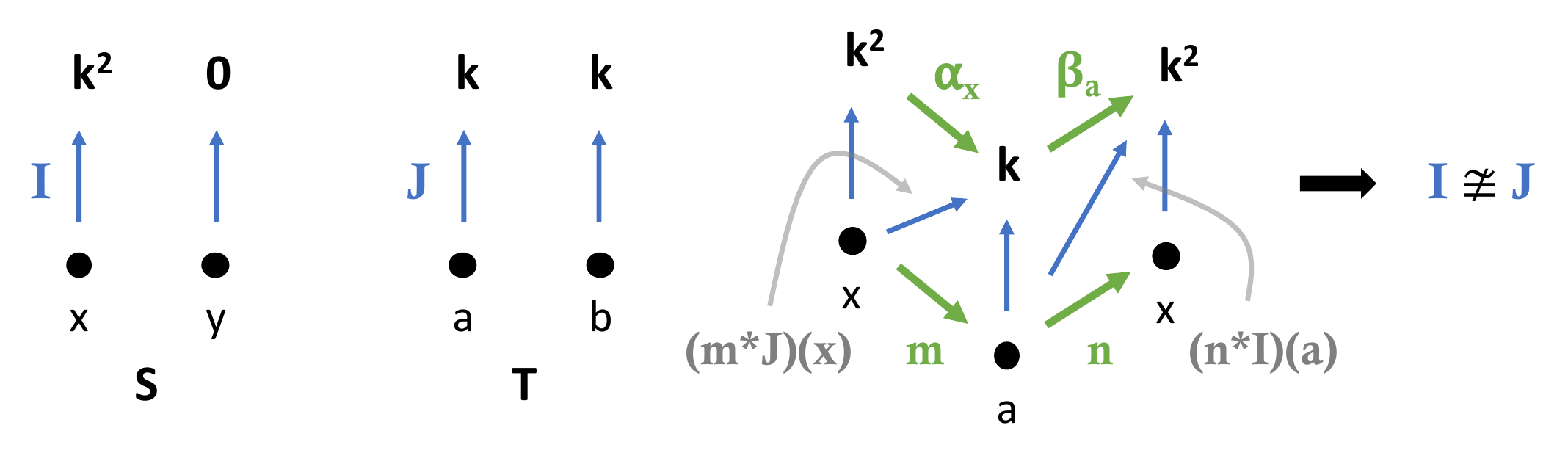}
  \caption{Associated to the offset filtration of the two subsets $X$ and $Y$ of $\R^2$ from Figure~\ref{fig:two-subsets} are two categorical decorated merge trees $\tilde{F}_1$ and $\tilde{G}_1$, as defined in Definition~\ref{defn:param-DMT}. This figure shows the mechanics of checking if the parameterized vector spaces at offset 0, $\tilde{F}_1(0)=I:S \to \Vect$ and $\tilde{G}_1(0)=J:T \to \Vect$, are isomorphic. They are not, which proves that our categorical decorated merge tree can distinguish these spaces. See Example~\ref{ex:non-iso-pVect} for more details.}
  \label{fig:non-iso-pVect}
\end{figure}

\begin{example}[Our Motivating Example, Reconsidered]\label{ex:non-iso-pVect}
In Figure~\ref{fig:two-subsets}, we considered the offset filtrations $F$ and $G$ associated to two different subsets of the plane $X$ and $Y$, respectively.
Following Definition~\ref{defn:param-DMT} we can associate two categorical decorated merge trees in degree 1, $\tilde{F}_1$ and $\tilde{G}_1$, to $X$ and $Y$.
To verify that $\tilde{F}_1\ncong \tilde{G}_1$ it suffices to show that their values at filtration value $0$, which we denote by $I:\{x,y\} \to \Vect$ and $J:\{a,b\} \to \Vect$, are not isomorphic in the category $\pVect$.

Two parameterized vector spaces $I: S \to \Vect$ and $J: T \to \Vect$ are \define{isomorphic} if there are set maps $m:S \to T$ and $n:T \to S$ and natural transformations $\alpha:I \Rightarrow m^*J$ and $\beta:J\Rightarrow n^*I$ satisfying
\[
	m^*\beta \circ \alpha =\id_I \quad \text{and} \quad n^*\alpha \circ \beta = \id_J;
\]
in particular,
\[
n\circ m = id_S \quad \text{and} \quad m\circ n =\id_T.
\]
By considering the parameterized vector spaces at 0 in our example, $I:\{x,y\} \to \Vect$ and $J:\{a,b\} \to \Vect$, where $I(x)=\Bbbk^2$, $I(y)=0$, $J(a)=\Bbbk$ and $J(b)=\Bbbk$, we can easily show that no isomorphism is possible because any bijection between $S=\{x,y\}$ and $Y=\{a,b\}$ will force a linear transformation of the form
\[
    \Bbbk^2 \to \Bbbk \to \Bbbk^2,
\]
which can never be an isomorphism.
\end{example}

\subsection{The Concrete Decorated Merge Tree}
The categorical notion of a decorated merge tree boils down to the following sequence of assignments: to each real number $s\in \R$ a set $I(s)$ is assigned and then to each element $i\in I(s)$ a (homology) vector space is assigned.
This process is reminiscent of specifying an element of
$\Hom(A, \Hom(B, C))$, which amounts to assigning to each element of $A$ a map from $B$ to $C$.
The reader then might find it useful to consider the adjunction between products and exponentials gotten by currying:
\[
    \Hom(A, \Hom(B, C)) \cong \Hom(A\times B,C).
\]
In this section we work with, in essence, the right hand side of this isomorphism, where $A\times B$ is replaced with the generalized merge tree $\cM_F$ and $C$ is replaced with the category of vector spaces.
The analog of currying in this section is \emph{concretization}, the namesake of the \emph{concrete decorated merge tree} defined below.

\begin{definition}[The Concrete Decorated Merge Tree]\label{defn:concrete-DMT}
Let $F:(\R,\leq) \to \Top^{\lc}$ be a persistent space with all $F(s)$ locally connected and let $(\cM_F,\cle)$ denote its generalized merge tree. The \define{concrete decorated merge tree in degree $n$} is the functor
\[
    \mathcal{F}_n:(\cM_F,\cle) \to \Vect_{\Bbbk} \qquad \text{where} \qquad (i,s) \quad \squigrightarrow \quad H_n(F(s)_i;\Bbbk)
\]
that records the $n^{th}$ homology of the $i^{th}$ component of $F(s)$. 

\end{definition}

\begin{remark}\label{rmk:generalized_concrete_DMT}
    The definition of a concrete decorated merge tree can be abstracted in a way that does not refer to a persistent space $F$ at all.
    That is, we can define a concrete decorated merge tree more generally to be a functor $\cF:(\cM_F,\cle) \to \Vect$ on a generalized merge tree (considered as a poset) associated to some persistent set.
    When we wish to displace emphasis on the originating persistent space $F$, we will use the term \define{tree module} to refer to $\cF$.
\end{remark}

\subsection{The Barcode Decorated Merge Tree}

To a persistent space $F:(\R,\leq) \to \Top^{\lc}$ we have already shown how to associate two (equivalent) devices to record homology in degree $n$ as it varies across components and filtration values $s\in \R$. Unfortunately, unlike ordinary persistent homology modules, neither of these devices have simple summaries such as barcodes or persistence diagrams.
This is due to the fact that the underlying poset $(\cM_F,\cle)$ is not totally-ordered.
However, if one considers the restriction of a tree module $\mathcal{F}_n$ to the principal up set at a point $p=(i,s)\in \cM_F$, then we \emph{do} obtain a module indexed by a totally ordered set and can call this the ``barcode at $p$.''
This motivates the following definitions.

\begin{definition}\label{defn:barcode_decorated_merge_tree}
A \define{barcode decorated merge tree} is a map from a generalized merge tree to the set of barcodes,
\[
    \cB : (\cM_F,\cle) \to \BCS.
\]
We say that a barcode decorated merge tree is \define{determined by restriction} if whenever $(i,s)=:p\cle q:=(j,t) \in \cM_F$, we have that
\[
    \cB(q) = \cB(p) \cap [s,\infty).
\]
If the generalized merge tree \define{has leaves}, meaning that every maximal chain in $(\cM_F,\cle)$ has a minimal element, then we call a barcode decorated merge tree that is determined by restriction a \define{leaf-decorated merge tree}.
\end{definition}

\begin{figure}
  \includegraphics[width=0.8\linewidth]{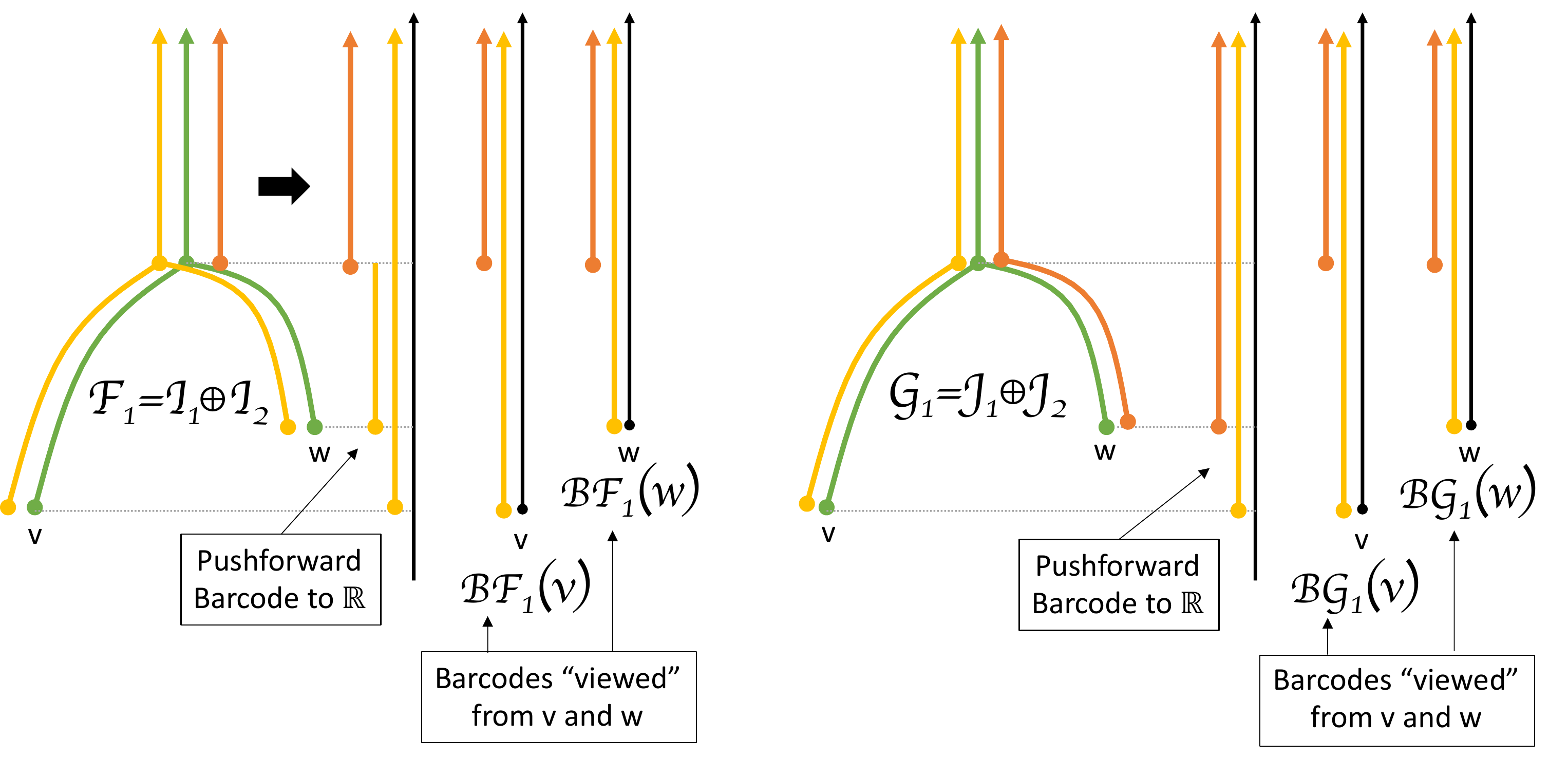}
  \caption{The Barcode Decorated Merge Tree comes from restricting a tree module to its leaf nodes, one at a time, and then calculating the barcode associated to the restriction of this tree module to the principal up set at each leaf node. In this figure two non-isomorphic tree modules are shown to have identical barcodes, when ``viewed'' from each of its leaf nodes. This example proves that the association of tree modules to their associated barcode decorations is not injective.}
  \label{fig:non-injective-barcode-decorations}
\end{figure}

\begin{definition}
Suppose $(\cM_F,\cle)$ is the generalized merge tree associated to the persistent set $\pi_0\circ F$.
Given a tree module $\mathcal{F}:(\cM_F,\cle) \to \Vect$ and a point $p=(i,s)\in \cM_F$, we define the \define{restriction of $\mathcal{F}$ to the principal up set $U_p$} to be the $\R$-module
\[
    \mathcal{F}|_{U_p} :(\R,\leq) \to \Vect \qquad \text{where, for } s\leq t, \qquad  \mathcal{F}|_{U_p}(t)=\mathcal{F}(\pi_0\circ F(s\leq t)(i),t).
\]
For $r < s$ we defined $\mathcal{F}|_{U_p}(r)=0$.
\end{definition}

\begin{proposition}\label{prop:barcode_transform}
Assume that the generalized merge tree $(\cM_F,\cle)$ has leaves.
To any pointwise finite dimensional tree module $\mathcal{F}:(\cM_F,\cle) \to \vect$, we have a leaf-decorated merge tree
\[
    \mathcal{BF}:(\cM_F,\cle) \to \BCS \qquad \text{where} \qquad \mathcal{BF}(p)=\BC(\mathcal{F}|_{U_p}).
\]
Here $U_p:=\{q\in \cM_F \mid p \cle q\}$ is the principal up set at $p$.
Note that since $(\cM_F,\cle)$ has leaves, it suffices to compute this barcode for every leaf node $v$.
\end{proposition}

\begin{definition}[Barcode Transform]
The map $\mathcal{BF}$ whose existence is implied by Proposition \ref{prop:barcode_transform} is referred to as the \define{barcode transform of $\mathcal{F}$}.
\end{definition}

\begin{proof}[Proof of Proposition \ref{prop:barcode_transform}]
Since the tree module is already pointwise finite dimensional, the restriction at each principal up set $U_p$ will be a pointwise finite dimensional $\R$-module. Crawley-Boevey's Theorem~\cite{crawley2015decomposition} then implies that this restricted tree module has a barcode.
Obviously the barcode decoration is determined by restriction because for any pair of comparable points $p\cle q$ the restriction at $U_q$ can be obtained by restricting the module at $U_p$.
\end{proof}

We have the following immediate corollary.

\begin{corollary}\label{cor:barcode-DMT}
If $F:(\R,\leq) \to \Top$ is a persistent space whose associated generalized merge tree $\cM_F$ has leaves and whose associated concrete decorated merge tree in degree $n\geq 0$ $\mathcal{F}_n:(\cM_F,\cle) \to \vect$ is pointwise finite dimensional, then $F$ has an associated leaf-decorated merge tree in degree $n$, $\mathcal{BF}_n$. 
\end{corollary}

\begin{remark}
The barcode transform associates an (indexed) ensemble of barcodes to a filtered space through a certain ``slicing" operation (i.e., slicing along upsets from leaves). This operation is analogous to several other recent methods in the TDA literature; we provide a few examples here. The \emph{persistent homology transform}  \cite{curry2018many,turner2014persistent} associates to an embedded simplicial complex a barcode for each direction, obtained by using projection onto the direction axis as a filtration function. A \emph{fibered barcode} \cite{hang2020correspondence,lesnick2015interactive} is a collection of barcodes associated to a multiparameter persistence module by slicing the parameter space by affine lines. The \emph{barcode embedding} \cite{dey2015comparing,oudot2017barcode} associates a barcode to each point in a metric graph by computing extended persistence with respect to a radial filtration centered at that point.
\end{remark}

Leaf decorated merge trees are the most computationally tractable of the three variants of DMTs introduced in this section. Algorithms for analyzing DMTs are described in Sections \ref{sec:computational_aspects} and \ref{sec:examples}.

\section{Continuity and Stability of Decorated Merge Trees}\label{sec:stability-of-DMTs}

In this section we focus on metric theoretic aspects of decorated merge trees---namely, continuity of the barcode transform $\mathcal{BF}:\mathcal{M}_F \to \BCS$ of a decorated merge tree and stability of various pseudometrics on the space of decorated merge trees. To properly state our results, we first briefly review standard metrics in the TDA literature.

\subsection{Metrics on Persistence Modules}

Let $B,B' \in \BCS$ be barcodes.

\begin{definition}\label{defn:bottleneck-distance}
A \define{matching} of barcodes $B$ and $B'$ is a bijection $\xi$ between subsets $\mathrm{dom}(\xi)\subset B$ and $\mathrm{ran}(\xi) \subset B'$---such a $\xi$ is also commonly referred to as a \define{partial bijection} between $B$ and $B'$. The \define{cost} of a matching $\xi$ is 
\[
\max \left\{\max_{I \in \mathrm{dom}(\xi)} \|I - \xi(I)\|_\infty, \max_{I \in B \setminus \mathrm{dom}(\xi)} \|I\|_\Delta, \max_{I' \in B' \setminus \mathrm{ran}(\xi)} \|I'\|_\Delta \right\},
\]
where, for $I$ and $I'$ with endpoints $b \leq d$ and $b'\leq d'$, respectively, 
\[
\|I - I'\|_\infty := \max\{|b-b'|,|d-d'|\} \qquad \mbox{ and } \qquad \|I\|_\Delta := \frac{d-b}{2}.
\]
If a matching $\xi$ has cost less than or equal to $\epsilon$, we refer to $\xi$ as an \define{$\epsilon$-matching}. The \define{bottleneck distance} between $B$ and $B'$ is 
\[
d_B(B,B'):= \inf\{\epsilon \geq 0 \mid \mbox{there exists an $\epsilon$-matching of $B$ and $B'$}\}.
\]
A matching realizing the bottleneck distance will be called an \define{optimal matching}.
\end{definition}

The bottleneck distance is an especially natural metric due to its connection to a category-theoretic metric on persistence modules called \emph{interleaving distance}. In fact, interleaving distance can be defined for more general functor categories \cite{bubenik2015metrics,bubenik2014categorification}. Generalizing interleaving distances to increasingly abstract classes of functors is an active field of research (e.g., \cite{de2018theory,stefanou2018dynamics}), we define the interleaving distance below at the level of generality which will be most useful to us. In what follows, $\Ccat$ denotes a fixed but arbitrary category.

\begin{definition}[Interleaving Distance]\label{defn:interleaving-distance}
Consider $(\R,\leq)$ as a poset category. 
For $\e \in [0,\infty)$, we define the poset map $\sigma^\e:\R \to \R$ via $\sigma^\e:s \mapsto s +\e$. 
Now let $F:(\R,\leq) \to \Ccat$ be a functor, which is an object in the functor category $\Fun(\R,\Ccat)$. 
We define the \define{$\e$-shift of $F$}, written $F^{\e}:(\R,\leq) \to \Ccat$, as $F^\e := F \circ \sigma^{\e}$. 
This shift is functorial, i.e.~to every morphism $\varphi:F\Rightarrow G$ in $\Fun(\R,\Ccat)$ we have another morphism $\varphi^{\e}:F^{\e}\Rightarrow G^{\e}$, which defines the \define{$\e$-translation functor} $(\bullet)^{\e}:\Fun(\R,\Ccat)\to \Fun(\R,\Ccat)$.
Since each object $F$ has a naturally associated \define{internal $\e$-shift},
\[
    \eta^{\e}_F: F \Rightarrow F^{\e} \quad \text{where} \quad \eta^{\e}_F(s): F(s) \to F(\sigma^{\e}(s)) \quad \text{is} \quad F(s\leq s + \e),
\]
there is a natural transformation from the identity functor $\id_{\Ccat^{\R}}$ on $\Fun(\R,\Ccat)$ to $(\bullet)^{\e}$.

Two objects $F,G:(\R,\leq) \to\Ccat$ are \define{$\e$-interleaved} if there are morphisms $\varphi:F\Rightarrow G^{\e}$ and $\psi:G\Rightarrow F^{\e}$ such that
\[
    \psi^{\e}\circ \varphi = \eta^{2\e}_F \qquad \text{and} \qquad \varphi^{\e}\circ \psi = \eta^{2\e}_G.
\]
We define the \define{interleaving distance between $F$ and $G$} as 
\[
\dint(F,G)=\inf\{\epsilon \geq 0 \mid F \text{ and } G \mbox{ are }\e\mbox{-interleaved}\}.
\]
\end{definition}

When considering persistence modules $F,G:(\R,\leq) \to \Vect$, the above definition reduces to the classical interleaving distance of persistence modules as introduced in the landmark paper~\cite{chazal2009proximity}. Building on this work, it was shown in \cite{lesnick2015theory} that  the map taking a persistence module to its barcode is an isometry with respect to the  interleaving and bottleneck distances.

\subsection{Continuity of the Barcode Transform}

We now establish the continuity of the barcode transform $\mathcal{BF}: \cM_F \to \BCS$ associated to a concrete decorated merge tree $\cF$.
Continuity is measured with respect to a whole family of $\ell^p$-type metrics on the underlying merge tree $\cM_F$ for $p\in[1,\infty]$. 
These metrics measure distance between pairs of points via their least common ancestor.
The metric on $\BCS$ is the bottleneck distance (Definition \ref{defn:bottleneck-distance}).

\begin{definition}\label{defn:LCA-merge-height}
Let $(\cM_F,\cle)$ be a generalized merge tree where $\pi_F:\cM_F\to\R$ is the projection function. The \define{merge height} of points $u,v \in \cM_F$ is
\[
\mathrm{merge}_F(u,v) := \inf \{\pi_F(w) \mid u,v \cle w\}.
\]
We define the \define{least common ancestor} of $u$ and $v$ to be the least common upper bound of $u$ and $v$,
\[
\mathrm{LCA}_{F}(u,v) := \mathrm{argmin}\{\pi_F(w) \mid u,v \cle w\},
\]
when it exists. In this case, 
\[
\mathrm{merge}_F(u,v) = \pi_F(\mathrm{LCA}_F(u,v)).
\]
We will drop the subscript $F$ when convenient.
Finally, if every pair of points in a merge tree have a least common ancestor, then we say the merge tree is \define{connected}.
\end{definition}

\begin{remark}
Note that for a generalized merge tree the least common ancestor need not exist.
For example if $S:(\R,\leq)\to\Set$ is the persistent set that assigns $S(t)=\{x,y\}$ for $t\leq 0$ and $S(t)=\{z\}$ for $t>0$, then there will not be a least upper bound for $x,y$.
The assumption that $S$ is constructible (Definition \ref{defn:constructible-persistent-set}) will, however, guarantee the existence of least common ancestors.
\end{remark}

\begin{definition}
If the generalized merge tree $\cM_F$ is connected,
then the \define{$\ell^p$ metric on $(\cM_F,\cle)$} for $1\leq p \leq \infty$ associates to every pair of points $u,v\in\cM_F$ the value

\[
    d^p_{\cM_F}(u,v)=\|(\mathrm{merge}(u,v)-\pi(u),\mathrm{merge}(u,v)-\pi(v))\|_p.
\]

\end{definition}

Using general estimates for $\ell^p$ norms on $\R^2$, one sees that the $\ell^p$ metrics are bi-Lipschitz equivalent and therefore induce the same topology. Moreover, we have the following characterization when $\cM_F$ comes from a sublevel set filtration.

\begin{proposition}\label{prop:merge_tree_metrics}
Let $f:X \to \R$ be a continuous map such that the associated merge tree $\cM_f$ has finitely many leaves. Then the topology induced by each $\ell^p$ metric coincides with the quotient space topology.
\end{proposition}

We defer the proof of the proposition to Appendix \ref{sec:interval-topology}, since it relies on other technical results about merge tree topologies.

\begin{theorem}[Continuity of Barcode DMTs]\label{thm:continuity_barcode_DMT}
Let $(\cM_F,\cle)$ be a connected generalized merge tree associated to a persistent set $\pi_0\circ F$. Endow $\cM_F$ with the extended metric $d^p_{\cM_F}$ and $\BCS$ with the bottleneck distance $d_B$. For any pointwise finite dimensional tree module $\mathcal{F}:(\cM_F,\cle)\to \vect$, the associated barcode decorated merge tree $\mathcal{BF} :\cM_F \to \BCS$ is $2^{1-1/p}$-Lipschitz for $p \in [1,\infty)$ and $2$-Lipschitz for $p = \infty.$
\end{theorem}

\begin{proof}
Suppose that $p=(i,s)$ and $q=(j,t)$ are carried to $z=(k,r)$ via $\pi_0\circ F$. 
This means that $d^1_{\cM_F}(p,q)\leq (r-s) + (r-t)$
There is an obvious $\epsilon_1:=r-s$ interleaving between the $\R$-modules $\mathcal{F}|_{U_p}$ and $\mathcal{F}|_{U_z}$.
To see this, note that there is a natural morphism of $\R$-modules $\mathcal{F}|_{U_z} \to \mathcal{F}|_{U_p}$ given by the 0 map up to, but not including, $r$. For real values greater than $r$ this morphism is the identity map.
The other morphism that participates in an interleaving is given by the internal morphisms from $\mathcal{F}|_{U_p}$ to $\mathcal{F}|_{U_z}^{\epsilon_1}$.
This proves that there is an $\epsilon_1$-interleaving between these restrictions.
The exact same argument shows that there is an $\epsilon_2=r-t$ interleaving between the restrictions $\mathcal{F}|_{U_q}$ and $\mathcal{F}|_{U_z}$.
The triangle inequality for the interleaving distance proves that there is at most an $\epsilon_1+\epsilon_2$ interleaving between $\mathcal{F}|_{U_p}$ and $\mathcal{F}|_{U_r}$.
By Lesnick's isometry theorem \cite{lesnick2015theory}, this implies that the bottleneck distance between $\BC(\mathcal{F}|_{U_p})$ and $\BC(\mathcal{F}|_{U_r})$ is at most $\epsilon_1+\epsilon_2$. This proves that the map is $1$-Lipschitz with respect to $d^1_{\cM_F}$ and the remaining cases follow by the general bounds $\|\cdot\|_1 \leq 2^{1-1/p}\|\cdot\|_p$, for $p \in [1,\infty)$ and $\|\cdot\|_1 \leq 2\|\cdot\|_\infty$ in $\R^2$.
\end{proof}

\subsection{Interleavings of Merge Trees}\label{sec:merge-tree-interleavings-comparison}

Implicit in Definition \ref{defn:interleaving-distance} is a way of comparing generalized merge trees.
If $F$ and $G$ are persistent spaces, then the interleaving distance between $\pi_0 \circ F$ and $\pi_0\circ G$, which are both objects of $\Fun(\R,\Set)$, is covered by that definition.
However, the original definition of the merge tree interleaving distance given in~\cite{mbw13} used the classical merge tree construction (Definition \ref{defn:merge-tree-as-reeb-graph}) and assumed continuity of the interleaving maps.
In this section we review the construction of~\cite{mbw13} and prove that under suitable hypotheses this classical interleaving distance is the same as the interleaving distance implied by Definition \ref{defn:interleaving-distance}, which we have isolated as Definition \ref{D:dint-tree}.

\begin{definition}[Classical Merge Tree Interleaving Distance \cite{mbw13}]
Let $f:X \to \R$ and $g: Y \to \R$ be continuous functions and let $\cM_f$ and $\cM_g$ be their associated merge trees, as defined in Definition~\ref{defn:merge-tree-as-reeb-graph}.
For $\epsilon\geq 0$ we define an \define{$\epsilon$-map} to be a continuous (with respect to quotient space topologies) map $\alpha:\cM_f\rightarrow \cM_g$ such that 
\[
\tilde{\pi}_{g}\circ\alpha([x],t)=t+\epsilon
\]
for all points $([x],t)\in \cM_f$. Two $\epsilon$-maps $\alpha:\cM_f\rightarrow \cM_g$ and $\beta:\cM_g\rightarrow \cM_f$ are said to be \define{$\epsilon$-compatible} if
$\beta\circ\alpha=\eta_f^{2\epsilon}$ and $\alpha\circ\beta=\eta_g^{2\epsilon}$. The \define{Morozov-Bekatayev-Weber (MBW) interleaving distance between merge trees $\cM_f$ and $\cM_g$} is
\[
\theta^{MBW}_I(\cM_f,\cM_g):=\inf\{\epsilon\geq 0 \mid \exists \text{ $\e$-compatible maps $\alpha:\cM_f \to \cM_g$ and $\beta:\cM_g \to \cM_f$}\}.
\]

\end{definition}

Given two functions $f:X \to \R$ and $g:Y\to \R$, we note that their associated sublevel set filtrations determine persistent spaces $F,G:(\R,\leq) \to \Top$. 
By post-composing these functors with $\pi_0:\Top\to\Set$, we get two persistent sets, which can be compared with Definition \ref{defn:interleaving-distance} when $\Ccat=\Set$.
We isolate this special case now.

\begin{definition}[Modern Merge Tree Interleaving Distance~\cite{bubenik2015metrics}]\label{D:dint-tree}
The \define{interleaving distance between merge trees} $\cM_f$ and $\cM_g$ associated to $f:X\to \R$ and $g:Y\to \R$ is 
defined to be the interleaving distance between the persistent sets $\pi_0\circ F$ and $\pi_0\circ G$ where $F:(\R,\leq) \to \Top$ is defined by $F(t):=f^{-1}(-\infty,t]$ and $G:(\R,\leq) \to \Top$ is defined by $G(t):=g^{-1}(-\infty,t]$.
We introduce the special notation
\[
    \theta_I(\cM_f,\cM_g):=d_I(\pi_0\circ F,\pi_0\circ G).
\]
\end{definition}

We now compare these two definitions of interleaving distance.
As a preliminary observation, note that
an $\epsilon$-map $\alpha: \cM_f\to \cM_g$ carries components of $f^{-1}(-\infty,t]$ to $g^{-1}(-\infty,t+\epsilon]$.
This is exactly the expression that $\alpha$ defines a natural transformation $\alpha: \pi_0 \circ F \Rightarrow \pi_0 \circ G^{\epsilon}$. 

\begin{lemma}\label{lem:tree-interleaving-bound}
If $\alpha:\cM_f \to \cM_g$ and $\beta:\cM_g \to \cM_f$ are $\e$-maps that are $\e$-compatible, then they induce an $\e$-interleaving between the persistent sets $\pi_0\circ F$ and $\pi_0\circ G$ of the sublevel set filtrations of $f:X \to \R$ and $g: Y \to \R$. 
\end{lemma}

\begin{proof}

The discussion above has already established that and $\e$-map $\alpha$ is equivalent to specifying a natural transformation $\alpha: \pi_0\circ F \Rightarrow \pi_0\circ G^{\epsilon}$ and that $\beta$ is likewise tantamount to a natural transformation $\beta: \pi_0\circ G \Rightarrow \pi_0\circ F^{\epsilon}$ \emph{without referring to continuity}.
By comparing with Definition~\ref{defn:interleaving-distance} it is now obvious that setting $\alpha$ to be $\varphi$ and $\beta$ to be $\psi$ and enforcing the $\epsilon$-compatibility condition then implies the interleaving condition defined there.
\end{proof}

We note that because an $\e$-interleaving of persistent sets does not necessarily guarantee continuity of the maps between the merge trees $\cM_f$ and $\cM_g$, Lemma \ref{lem:tree-interleaving-bound} establishes the bound:
\[
    \theta_I(\cM_f,\cM_g)\leq \theta_I^{MBW}(\cM_f,\cM_g).
\]
In the next proposition we provide hypotheses under which these two interleaving distances are the same. 
This fills a gap in the literature that is not usually remarked upon and helps bridge the gap between~\cite{mbw13} and~\cite{bubenik2015metrics}.

\begin{proposition}\label{prop:two-interleavings-same}
Suppose $f:X \to \R$ and $g:Y\to \R$ are continuous maps defined on compact spaces $X$ and $Y$ so that their (classical) merge trees $\cM_f$ and $\cM_g$ have finitely many leaves.
Let $F$ and $G$ denote the sublevel set filtrations of $f$ and $g$.
Every $\epsilon$-interleaving of the persistent sets $\pi_0\circ F$ and $\pi_0 \circ G$ defines a pair of $\epsilon$-compatible maps between the merge trees $\cM_f$ and $\cM_g$. Consequently,
\[
\theta_I^{MBW}(\cM_f,\cM_g)=\theta_I(\cM_f,\cM_g)=d_I(\pi_0\circ F,\pi_0\circ G).
\]
\end{proposition}

\begin{proof}
Suppose $\phi:\pi_0\circ F \Rightarrow \pi_0\circ G^{\e}$ and $\psi:\pi_0\circ G \Rightarrow \pi_0\circ F^{\e}$ specify an $\e$-interleaving.
Now consider the display posets of $\pi_0\circ F$ and $\pi_0\circ G$, written $\cM_F$ and $\cM_G$ to distinguish them from the merge trees $\cM_f$ and $\cM_g$.
It is obvious that $\cM_F$ and $\cM_f$ are identical as posets, similarly for $\cM_G$ and $\cM_g$.
Following the discussion in Appendix~\ref{sec:interval-topology} we can equip $\cM_F$ and $\cM_G$ with the \emph{interval topology} (see Definition~\ref{defn:interval-topology}).
By Lemma~\ref{lem:upper_set} the principal up set of any point is closed in this topology.
Since $\cM_F$ and $\cM_G$ have finitely many leaves, consider the closed cover of each by the principal up sets at each of the leaf nodes.
It is easy to see that the map $\phi:\cM_F \to \cM_G$ is continuous when restricted to the up set of any leaf node in $\cM_F$. Since a map is continuous if its restricted to any member of a finite closed cover, we conclude that $\phi:\cM_F \to \cM_G$ is continuous with respect to the interval topology.
A completely symmetric argument proves that $\psi:\cM_G \to \cM_F$ is continuous with respect to the interval topology as well.
By Proposition~\ref{prop:interval_topology} we can conclude that the interval topology and the quotient topology are the same, thus $\phi$ and $\psi$ can be viewed as $\e$-maps (with the continuity assumption) that are $\e$-compatible. The claim now follows.
\end{proof}

\subsection{A Decorated Bottleneck Distance}\label{sec:decorated-bottleneck-distance}

In this section we leverage the barcode decorated merge tree perspective to define a second distance that uses an interleaving of merge trees along with an $\e$-matching of the barcodes over each of the merge trees, as determined by their decorations.
Although finding optimal interleavings between merge trees is NP-Hard \cite{agarwal2018computing}, we provide in Section \ref{sec:computational_aspects} an algorithm for estimating them.

\begin{definition}\label{defn:matchings_DMTs}

Assume that $\cM_F$ and $\cM_G$ are generalized merge trees associated to the persistent sets $\pi_0\circ F$ and $\pi_0\circ G$.
Given two barcode decorated merge trees
\[
\mathcal{B}_F:\cM_F \to\BCS \qquad \text{and} \qquad \mathcal{B}_G:\cM_G\to \BCS
\]
we define an \define{($\e,\delta)$-matching of $\mathcal{B}_F$ and $\mathcal{B}_G$} to consist of
\begin{itemize}
    \item an $\e$-interleaving of the underlying generalized merge trees i.e. natural transformations $\phi:\pi_0\circ F \Rightarrow \pi_0 \circ G^{\e}$ and $\psi:\pi_0\circ G\Rightarrow \pi_0 \circ F^{\e}$ that satisfy $\psi^{\e}\circ\phi=\eta_F^{2\e}$ and $\phi^{\e}\circ \psi = \eta_G^{2\e}$. We will abuse notation and write $\phi:\cM_F \to \cM_G$ and $\psi:\cM_G \to \cM_F$ to make clear that interleaving of persistent sets always provide $\e$-compatible maps in the sense of~\cite{mbw13} (without the continuity assumption) between the associated generalized merge trees;
    
    \item a $\delta$-matching of the barcodes $\mathcal{B}_F(p)$ and $\mathcal{B}_G(\phi(p))$ for every $p\in \cM_F$ and a $\delta$-matching of the barcodes $\mathcal{B}_G(q)$ and $\mathcal{B}_F(\psi(q))$ for every $q\in \cM_G$.
\end{itemize}
\end{definition}

\begin{definition}\label{defn:decorated-bottleneck-distance}
Let $p \in [1,\infty]$. The \define{decorated bottleneck $p$-distance} between two barcode decorated merge trees $\mathcal{B}_F$ and $\mathcal{B}_G$ is defined as
\[
d_{B,p}(\mathcal{B}_F,\mathcal{B}_G) := \inf \{\|(\e,\delta)\|_p \mid \mbox{$\exists$ $(\e, \delta)$-matching of $\mathcal{B}_F$ and $\mathcal{B}_G$}\}.
\]
When $p=\infty$, we refer to the metric simply as the \define{decorated bottleneck distance} and write $d_B := d_{B,\infty}$.
\end{definition}

\begin{remark}
We use the notation $d_B$ for both the bottleneck distance between barcodes and for the decorated bottleneck $\infty$-distance to emphasize the connection between the metrics. The meaning of $d_B$ should always be clear from context.
\end{remark}

\begin{remark}
We define the general family of decorated bottleneck $p$-distances in anticipation that they will be be useful in data analysis tasks down the line. In this paper, we develop the theory of the $p=\infty$ metric $d_B$ and will almost exclusively refer to the decorated bottleneck distance in what follows.
\end{remark}

The ($p = \infty$) decorated bottleneck distance can be formulated more simply than the general $p$ version. An $(\epsilon,\delta)$-matching of barcode decorated merge trees where $\epsilon = \delta$ will be referred to as an \define{$\epsilon$-matching}. The proof of the following proposition is elementary but technical. We relegate it to Appendix \ref{sec:technical_proof_decorated_bottleneck_reformulation}.

\begin{proposition}\label{prop:decorated_bottleneck_reformulation}
Let $\cB_F$ and $\cB_G$ be barcode decorated merge trees which are determined by their leaves. The decorated bottleneck distance can be expressed as 
\[
d_B(\cB_F,\cB_G) = \inf \{\epsilon \geq 0 \mid \mbox{$\exists$ $\e$-matching of $\mathcal{B}_F$ and $\mathcal{B}_G$}\}.
\]
\end{proposition}

The last result of this subsection says that ``determined by leaves'' assumption allows us to check one of the matching conditions only at leaves. Its proof uses technical tools from the proof of Proposition \ref{prop:decorated_bottleneck_reformulation}, so we delay it to Appendix \ref{sec:technical_proof_decorated_bottleneck_reformulation}.

\begin{proposition}\label{prop::determinedleaves}

Let $\mathcal{B}_F:\cM_F \to \BCS$ and $\mathcal{B}_G:\cM_G \to \BCS$ be two leaf-decorated merge trees.
Let $\phi$ and $\psi$ be an $\e$-interleaving of the underlying generalized merge trees $\cM_F$ and $\cM_G$ where for each leaf $v\in \cM_F$ and each leaf $w\in \cM_G$ we have a $\delta$-matching between $\mathcal{B}_F(v)$ and $\mathcal{B}_G(\phi(v))$ and a $\delta$-matching between $\mathcal{B}_G(w)$ and $\mathcal{B}_F(\psi(w))$. Then there is an $(\e,\delta)$-matching between the entire barcode decorated merge trees $\mathcal{B}_F$ and $\mathcal{B}_G$.
\end{proposition}

\subsection{Hierarchy of Distances}
In this section we prove the main two results of this paper.

\begin{theorem}[The Hierarchy of Stability Results for DMTs]\label{thm:hierarchy-stability}
Let $f:X \rightarrow \R$ and $g:Y \rightarrow \R$ be continuous functions whose sublevel sets are locally connected.
Let $F$ and $G$ be the sublevel set filtrations, viewed as persistent spaces.
The categorical decorated merge trees for homological degree $n$ are denoted $\tilde{F}_n$ and $\tilde{G}_n$ and the associated barcode decorated merge trees are $\mathcal{BF}_n$ and $\mathcal{BG}_n$.
Our various distances satisfy
\[
\theta_I(\cM_F,\cM_G) \leq d_B(\mathcal{BF}_n,\mathcal{BG}_n) \leq d_I(\tilde{F}_n,\tilde{G}_n) \leq \delta_I(X_f,Y_g).
\]
Moreover, if $X = Y$ then we have that
$$
\delta_I(X_f,Y_g)\leq ||f-g||_\infty.
$$
\end{theorem}

The distance $\delta_I$ is a special case of the \emph{persistent homotopy type distance} introduced in \cite{frosini2019persistent}. We recall the definition below.

\begin{definition}\label{D:theoretical}
An \define{$\R$-space} is a topological space $X$ endowed with a continuous function $f:X \to \R$. An \define{$\e$-interleaving of $\R$-spaces} $f:X \to \R$ and $g:Y \to \R$ is a pair of continuous maps $\Phi:X\to Y$ and $\Psi: Y \to X$ along with homotopies $H_X : X\times [0,1] \to X$ and $H_Y: Y \times [0,1] \to Y$ connecting the identity maps $\id_X$ and $\id_Y$ with $\Psi\circ \Phi$ and $\Phi\circ\Psi$, respectively.
We require further that the following four properties hold for $\Phi$, $\Psi$, $H_X$ and $H_Y$:
\begin{enumerate}
    \item $\Phi (X_{\leq s}) \subseteq Y_{\leq s+\e}$ for all $s\in\R$
    \item $\Psi(Y_{\leq s}) \subseteq X_{\leq s + \e}$ for all $s\in\R$
    \item $f\circ H_X(x,t) \leq f(x) + 2\e$  for all $x\in X$ and $t\in [0,1]$
    \item $g\circ H_Y(y,t) \leq g(y) + 2\e$ for all $y\in Y$ and $t\in [0,1]$
\end{enumerate}
The \define{persistent homotopy type distance} between $\R$-spaces $X_f := f:X \to \R$ and $Y_g := g:Y \to \R$
is defined as
\[
\delta_I (X_f, Y_g) := \inf \,\{\epsilon \mid X_f \text{ and } Y_g 
\text{ are $\epsilon$-interleaved}\}.
\]
If no interleaving exists, we set $\delta_I (X_f, Y_g) = \infty$.
\end{definition}

This metric was used in \cite{hang2019topological} to develop persistent homology of $\R$-spaces with different, but homotopic domains. 
A similar metric on $\R$-spaces, called the \define{homotopy interleaving distance}, was defined in \cite{bl17}, but the exact connection between homotopy type distance and homotopy interleaving distance remains to be studied. 

\begin{proof}[Proof of Theorem \ref{thm:hierarchy-stability}]
The leftmost inequality $\theta_I(\cM_F,\cM_G) \leq d_B(\mathcal{BF}_n,\mathcal{BG}_n)$ follows from the definition of the decorated bottleneck distance, since the data of an interleaving of merge trees is part of the definition.
Moreover it is obvious that $d_I(\tilde{F}_n,\tilde{G}_n)\geq \theta_I(\cM_F,\cM_G)$ since one can always apply the functor $\dom$ to any $\epsilon$-interleaving $\tilde{\phi}:\tilde{F}_n\Rightarrow \tilde{G}_n^{\epsilon}$ and $\tilde{\psi}:\tilde{G}_n\Rightarrow \tilde{F}_n^{\epsilon}$ of categorical decorated merge trees to obtain an $\epsilon$-interleaving $\phi$ and $\psi$ of the underlying persistent sets or generalized merge trees; this is the Lipschitz stability result of~\cite[Thm. 2.3.6]{bubenik2015metrics} applied to the functor $\dom$.
Moreover, if we take a point $p\in\cM_F$ then $\tilde{\phi}:\tilde{F}_n\Rightarrow \tilde{G}_n^{\epsilon}$ specifies, by restriction, a morphism between the restriction of $\tilde{F}_n$ to the upset at $p$ and the restriction of $\tilde{G}_n^{\epsilon}$ to the upset at $\phi(p)$.
Similarly, since $\psi(\phi(p))$ must equal the $2\epsilon$ translate of $p$ in $\cM_F$, we can conclude that the module $\tilde{F}_n|_{U_p}$ and $\phi^*\tilde{G}^{\epsilon}_n|_{U_p}$ are $\epsilon$ interleaved, as $\R$-modules.
By the Algebraic Stability Theorem~\cite{bauer2020persistence,chazal2009proximity}, which states that an $\epsilon$-interleaving of $\R$-modules induces an $\epsilon$-matching of barcodes, we have an $\epsilon$-matching between the barcodes $\mathcal{BF}_n(p)$ and $\mathcal{BG}_n(\phi(p))$ for all $p$.
A completely symmetric argument works where we consider $q\in \cM_G$ and the natural transformation $\tilde{\psi}:\tilde{G}_n\Rightarrow \tilde{F}_n^{\epsilon}$.
This proves that
\[
    d_B(\mathcal{BF}_n,\mathcal{BG}_n) \leq d_I(\tilde{F}_n,\tilde{G}_n).
\]

To prove the rightmost inequality, first observe that if $\delta_I(X_f,Y_f) = \infty$ then we are done, so assume not. Given $\epsilon > \delta_I (X_f, Y_g)$, there exists an $\epsilon$-interleaving of the $\R$-spaces
$X_f$ and $Y_g$.
In particular, for each $t\in \R$ we have that the map $\Phi$ restricts to a continuous map between the sublevel sets $\Phi: F(t):= X_{\leq t} \to Y_{\leq t+\e}=: G(t+\e)$.
Since $\Phi$ is continuous and defined globally on $X$, it specifies a natural transformation of the persistent spaces $\phi: F \Rightarrow G^{\e}$.
Moreover, since the continuous image of a (path) connected set is connected, the natural transformation $\phi$ defines a natural transformation between the associated persistent spaces that are parameterized by their (path) components $\tilde{\phi}: \tilde{F} \Rightarrow \tilde{G}^{\e}$.
Interchanging the roles of $X$ and $Y$ in the above discussion implies that $\Psi:Y \to X$ defines a natural transformation $\tilde{\psi}:\tilde{G} \Rightarrow \tilde{F}^{\e}$.

It should be noted that $\tilde{\phi}$ and $\tilde{\psi}$ \emph{do not} define an interleaving of the functors $\tilde{F}$ and $\tilde{G}$ because that would require that for each $t\in \R$ the composition $\psi(t+\e)\circ \phi(t)$ equals the inclusion map $X_{\leq t}\subseteq X_{\leq t+2\e}$.
This is not true, but condition (3) of Definition~\ref{D:theoretical} does require that $\Psi\circ \Phi$ restrict to a map that is homotopic to the inclusion map $X_{\leq t}\subseteq X_{\leq t+2\e}$ for all $t$.
Since homotopic maps induce the same map on homology, we can conclude that the natural transformations $\tilde{\phi}:\tilde{F}\Rightarrow \tilde{G}^{\e}$ and $\tilde{\psi}:\tilde{G} \Rightarrow \tilde{F}^{\e}$ define an $\e$-interleaving of the persistent component homology modules $\tilde{F}_n$ and $\tilde{G}_n$.

In the case that $X = Y$, the desired inequality is \cite[Proposition 2.11]{frosini2019persistent}. This completes the proof.
\end{proof}

Although the above result establishes that categorical DMTs are more sensitive than undecorated merge trees, the relationship to higher dimensional persistent homology is not described above. 
The next result establishes that categorical DMTs are, in fact, more sensitive than ordinary persistent homology.

\begin{theorem}\label{thm:cat-DMT-more-sensitive-than-PH}
For locally connected persistent spaces $F,G:(\R,\leq) \to \Top^{\lc}$, the categorical decorated merge tree in homological degree $n$ is more sensitive than the interleaving distance of the persistent homology modules in degree $n$, i.e
\[
    d_I(H_n\circ F,H_n\circ G) \leq d_I(\tilde{F}_n,\tilde{G}_n).
\]
\end{theorem}
\begin{proof}
 
We have
\[
    \cop\circ \tilde{F}_n(s):=\bigoplus_{i\in \pi^{-1}(s)} H_n(F(s)_i)\cong H_n(\bigsqcup_{i\in \pi^{-1}(s)} F(s)_i) \cong  H_n\circ F(s),
\]
and similarly for $G$.

Note that the last isomorphism on the right uses the assumption that each $F(s)$ is locally connected.
Since naturally isomorphic modules have the same interleaving distance, the result now follows from the Lipschitz stability result of~\cite[Thm. 2.3.6]{bubenik2015metrics} applied to the functor $\cop$, i.e. applying $\cop$ is a 1-Lipschitz map from $\Fun(\R,\pVect)$ to $\Fun(\R,\Vect)$.
\end{proof}

\begin{remark}[No Contradiction to Universality]~\label{rmk:universality}
By reconsidering the example from Figure~\ref{fig:two-subsets}, we know that the above inequality is not an equality, in general. 
Moreover, when taken together Theorem~\ref{thm:cat-DMT-more-sensitive-than-PH} and Theorem~\ref{thm:hierarchy-stability} may appear to contradict Lesnick's Universality result~\cite[Thm. 5.5]{lesnick2015theory}, which asserts that the interleaving distance on persistent homology modules is the most sensitive distance that is bounded above by the $L^{\infty}$ distance.
However, there is no contradiction because universality theorems of the type stated and conjectured in~\cite{lesnick2015theory} only apply to functors from $\R$ (or $\R^d$) to $\Vect$ and not $\pVect$.
\end{remark}

\begin{remark}[No One-Way Relationship of Bottleneck Distances]\label{rmk:no-bottleneck}
In the case where locally connected persistent spaces $F$ and $G$ give rise to pointwise finite dimensional persistent homology modules, we know that $d_I(H_n\circ F,H_n\circ G)$ coincides with the bottleneck distance by virtue of the isometry theorem~\cite[Thm. 3.4]{lesnick2015theory}.
Unfortunately, the decorated bottleneck distance $d_B(\mathcal{BF}_n,\mathcal{BG}_n)$ can be greater than or less than the bottleneck distance between $H_n\circ F$ and $H_n\circ G$.
If one considers the motivating example in Figure~\ref{fig:two-subsets}, then the decorated bottleneck distance there is obviously greater than the bottleneck distance between $H_1\circ F_X$ and $H_1\circ F_Y$. For this example, the bottleneck distance is 0, but the decorated bottleneck distance is $R/2$, corresponding to half the radius of the circles in Figure~\ref{fig:two-subsets}.
On the other hand, Figure~\ref{fig:non-injective-barcode-decorations} gives an example of two non-isomorphic tree modules (concrete DMTs) with identical barcode decorations, so the decorated bottleneck distance for $\calF$ and $\calG$ in those examples is 0. 
However, the underlying persistent homology modules, which arise via the pushforward to $\R$ are separated by the length of the edge from the merge point to $w$ in the bottleneck distance.
\end{remark}

\section{Representations of Tree Posets and Lift Decorations}\label{sec:representations}

This section treats questions about the  decomposability of decorated merge trees. We will see that the situation is much more subtle than in the classical persistence setting, due to the fact that merge tree posets are not totally ordered.

\subsection{Indecomposables with totally ordered support}
A crucial feature of pointwise finite dimensional (PFD) persistence modules is that they decompose as a direct sum of interval modules; these decompositions are used ubiquitously in visualization and analysis tasks in TDA. 
Although a PFD concrete decorated merge tree $\mathcal{F}:\cM_F \to \vect$ also admits a (Remak) decomposition into a direct sum of indecomposables $\mathcal{F} \approx \bigoplus_i \mathcal{F}_i$, where 
the direct sum is taken pointwise \cite[Theorem 1.1]{botnan2020decomposition}, the indecomposables need not have any special structure. 
However, we will consider a construction where the DMT does decompose into modules that are supported on totally ordered subsets of the merge tree, which will give us a structure theorem analogous to the usual one in persistence.

\begin{definition}\label{def:real_interval_decomposable}
A concrete decorated merge tree $\mathcal{F}:\cM_F \to \Vect_{\Bbbk}$ is \define{real interval decomposable} if it can be expressed as a direct sum $\mathcal{F} = \bigoplus_i \mathcal{F}_i$, where each $\mathcal{F}_i:\cM_F \to \Vect_{\Bbbk}$ is indecomposable and has totally ordered support.
\end{definition}

The next example shows that an arbitrary DMT does not have to be real interval decomposable.

\begin{figure}
\centering
  \includegraphics[width=0.8\linewidth]{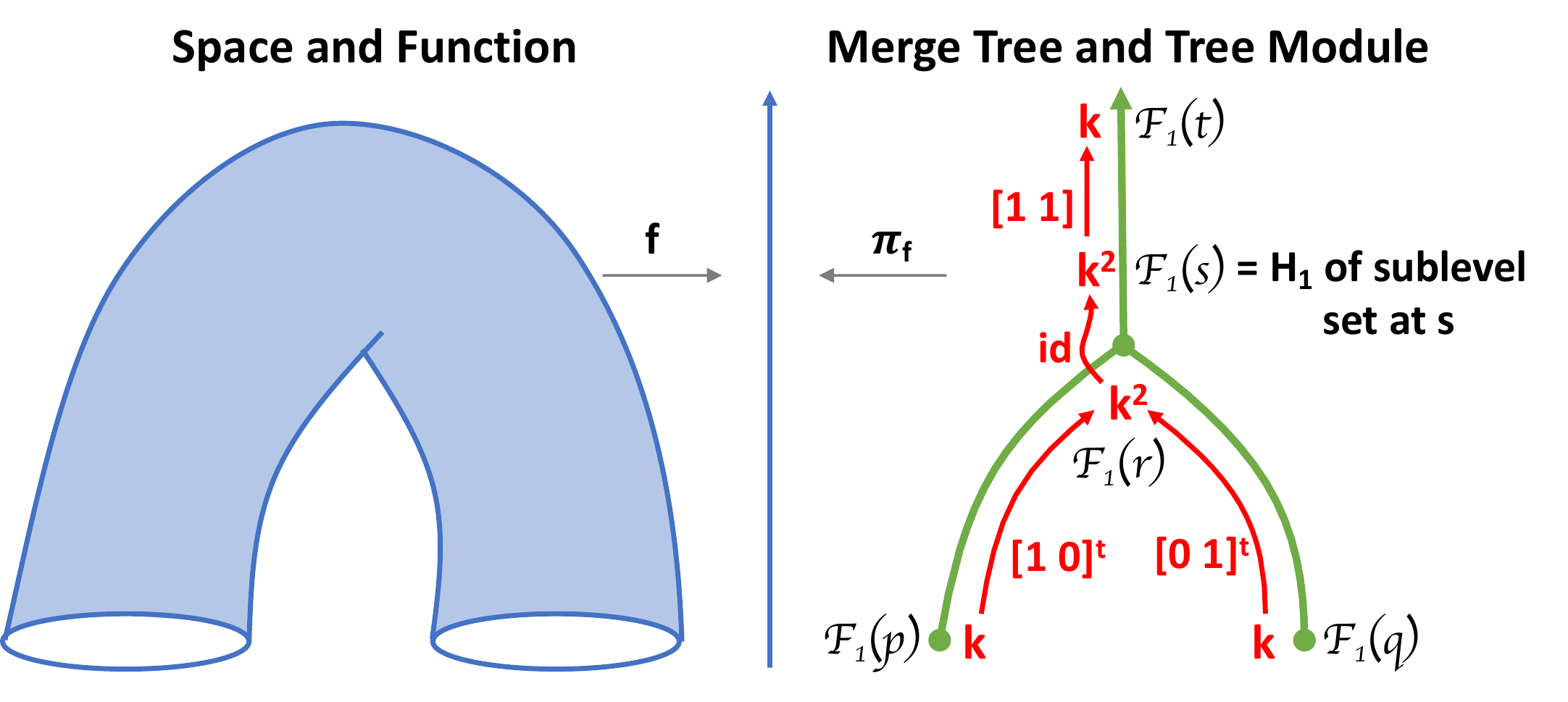}
  \caption{Consider a cylinder with height function induced by the implied embedding. The concrete decorated merge tree in degree-1 is an indecomposable tree module that is two dimensional at the merge point and for values up to, but not including, the maximum. This is equivalent to an indecomposable tree module, which is equivalent to one of the 12 indecomposable representations of the Dynkin Diagram $D_4$. This shows that not every tree module is real interval decomposable.}
  \label{fig:indecomposable-macaroni}
\end{figure}

\begin{example}\label{ex:indecomposable-macaroni}
Consider the space and map shown in Figure~\ref{fig:indecomposable-macaroni}. The concrete decorated merge tree $\calF_1:(\cM_F,\cle)\to\vect$ reduces to the study of the four vector spaces and the maps between them, indicated in red.
It is a fact that this tree module is equivalent to one of the twelve indecomposable representations of the Dynkin Diagram $D_4$.
This gives a natural, Morse-theoretic example of a function whose associated tree module is not  real-interval decomposable.
\end{example}

The aim of this subsection is to establish a sufficient condition for a decorated merge tree to be real interval decomposable.

\begin{definition}\label{def:untwisted}
A collection of vector subspaces $V_1,\ldots,V_m$ of a fixed vector space $V$ is \define{independent} if for all $i$
$V_i\cap \big(\sum_{j\neq i}V_j\big)= 0.$
Under this hypothesis we have that
\[
    V_1+\cdots+ V_m = V_1\oplus \cdots \oplus V_m.
\]
Let $(\cM_F,\cle)$ be a generalized merge tree and
let $\mathcal{F}:(\cM_F,\cle) \to \Vect$ be a tree module.
We say that $\mathcal{F}$ is \define{untwisted} if whenever we have a set of incomparable elements $p_1,p_2,\cdots,p_n\in \cM_F$ and any upper bound $p\in \cM_F$ of these, 
the collection of vector subspaces 
\[
    \Ima\mathcal{F}(p_1\cle p), \Ima\mathcal{F}(p_2\cle p) \cdots, \Ima\mathcal{F}(p_n\cle p)
\]
is independent.
\end{definition}

\begin{remark}
It should be noted that this notion of untwisted is completely unrelated to the notion used in~\cite{kim2018generalized}, which is used to describe Reeb graphs which admit sections when restricted to certain intervals.
\end{remark}

To state our sufficient condition for real interval decomposability, we impose a mild geometric constraint on our generalized merge trees (Definition~\ref{defn:generalized-MT}).

\begin{definition}[cf.~Defn.~2.2 of~\cite{patel2018generalized}]\label{defn:constructible-persistent-set}
A persistent set $S:(\R,\leq)\to\set$ is \define{constructible} if there exists a collection of times $\tau=\{t_0 < t_1 < \cdots < t_n\}$ such that
\begin{itemize}
    \item $S(t)=\varnothing$ for $t<t_0$,
    \item $S(t)=\{\star\}$ for $t>t_n$, and
    \item $S(t\leq s)$ is a bijection for every pair $t\leq s \subset [t_i,t_{i+1})$.
\end{itemize}
Assuming these conditions, the associated display poset $(\mathcal{S},\cle)$ is said to be \define{tame}.
\end{definition}

It is known that a constructible persistent set has an associated display poset that can be topologized as a finite cell complex, making it into a bona fide tree~\cite{curry2018fiber,stefanou2018dynamics}. There are a few other consequences of this tameness condition that are worth noting.

\begin{proposition}\label{prop:constructible-PFD}
Suppose $(\cM_F,\cle)$ is a \emph{tame} generalized merge tree, i.e. the defining persistent set $\pi_0\circ F$ is constructible.
The following consequences of this assumption are obvious:
\begin{itemize}
    \item Any maximal chain $C\subseteq \cM_F$ has a minimum, which we call a \define{leaf node}. The set of leaf nodes of $\cM_F$ is finite.
    \item Every leaf node has a (unique) least upper bound $p_{\infty}=(\star,t_n)$, which we call the \define{root node}.
    \item Any PFD tree module $\mathcal{F}:(\cM_F,\cle)\to\vect$ pushes forward along the natural projection map $\pi:\cM_F \to \R$ to a PFD $\R$-module $\pi_*\mathcal{F}:(\R,\leq) \to \vect$, which is defined pointwise as
    \[
        \pi_*\mathcal{F}(s):=\bigoplus_{p\in \pi^{-1}(s)} \mathcal{F}(s).
    \]
\end{itemize}
\end{proposition}

\begin{remark}
When the need arises, we use notation $\pi_F:\mathcal{M}_F \to \R$ for the projection map. The extra notation will be necessary when comparing merge trees $\mathcal{M}_F$ and $\mathcal{M}_G$ later in the paper.
\end{remark}

\begin{remark}
In Appendix \ref{sec:existence_of_merge_tree}, we give an alternative construction of a merge tree which is guaranteed to produce a tree, in the metric sense, without any tameness assumption.
\end{remark}

\begin{theorem}\label{thm:untwisted}
Let $(\cM_F,\cle)$ be a tame generalized merge tree. A pointwise finite dimensional tree module $\mathcal{F}:(\cM_F,\cle) \to \vect$ is real interval decomposable if and only if it is untwisted.
\end{theorem}

\begin{proof}

We proceed by induction on the number of leaf nodes $n$ of a generalized merge tree $(\cM_F,\cle)$.
If $n=1$, then $\cM_F$ is totally ordered and any tree module is automatically untwisted.
The result then follows from the usual decomposition theorem for $\R$-indexed persistence modules.
If the result is true for any PFD tree module on a generalized merge tree with $\ell$ leaves where $1\leq \ell \leq m$, then we now show the result is true for tree modules on trees with $m+1$ leaf nodes.

To begin, we assume that $(\cM_F,\cle)$ has $m+1$ leaf nodes and a root node $p_{\infty}$.
Write $\cM_F$ as the union of two subsets $\cM_1$ and $\cM_2$ where each subset has at least one leaf node and $\cM_1\cap \cM_2 = U_{p_{\infty}}$---the intersection is the principal up set at the root node, which is totally ordered.
Associated to any PFD tree module $\mathcal{F}:(\cM_F,\cle)\to\vect$ and such a decomposition $\cM_F=\cM_1\cup \cM_2$ are two natural submodules $\mathcal{F}_1$ and $\mathcal{F}_2$.
We define $\mathcal{F}_1$ via three cases:
\begin{enumerate}
    \item If $p\in M_1$ and $p<p_{\infty}$, then $\mathcal{F}_1(p):=\mathcal{F}(p)$.
    \item If $p\geq p_{\infty}$, then $\calF_1(p):=\sum_{q\in M_1|q<p_{\infty}} \Ima \mathcal{F}(q <p)$.
    \item If $p\in M_2\setminus M_1$, then $\mathcal{F}(p)=0$.
\end{enumerate}
$\mathcal{F}_2$ is defined exactly in the same manner, simply by interchanging $1$ and $2$ in the above definition.
By the induction hypothesis both $\cF_1$ and $\cF_2$ are untwisted and hence real interval decomposable, i.e. $\cF_1 \cong \oplus_{\alpha \in I_1} \cF_{\alpha}$ and $\cF_2 \cong \oplus_{\beta \in I_2} \cF_{\beta}$ and each of $\cF_{\alpha}$ and $\cF_{\beta}$ are supported on totally ordered subsets of $\cM_1$ and $\cM_2$, respectively.

By Proposition~\ref{prop:constructible-PFD}, $\pi_*\mathcal{F}$, $\pi_*\cF_1$ and $\pi_*\cF_2$ are each PFD $\R$-modules.
Moreover, since $\pi_*:\Fun(\cM_F,\vect) \to \Fun(\R,\vect)$ is an additive functor, we know that
\[
    \pi_*\cF_1 \cong \pi_* \oplus_{\alpha} \cF_{\alpha} \cong \oplus \pi_*\cF_{\alpha},
\]
which agrees, up to permuting factors with the Remak decomposition of $\pi_*\cF_1$ guaranteed by the usual decomposition theorem in persistence. The symmetric statement is true for $\pi_*\cF_2$ as well.
This implies that we can collect the interval modules that appear in the Remak decomposition of $\pi_*\cF$ into three terms:
\[
    \pi_*\cF \cong F_1 \oplus F_2 \oplus F_3 \qquad \text{where} \qquad F_1 \cong \pi_*\cF_1 \qquad \text{and} \qquad F_2\cong \pi_*\cF_2.
\]
We claim that $F_3$ is the direct sum of real interval modules that are born at or after $\pi(p_{\infty})$. 
In particular, the support of $F_3$ is contained in $[\pi(p_{\infty}),\infty)$.
This is clear because for any $t< \pi(p_{\infty})$ the set of pre-images $\pi^{-1}(t)$ can be partitioned into points that are in $\cM_1$ or $\cM_2$ exclusively, because $\pi^{-1}(t)\cap \cM_1 \cap \cM_2 =\varnothing$.
This implies that
\[
    \text{for } t< \pi(p_{\infty}) \qquad \pi_*\cF(t) \cong \pi_*\cF_1 (t) \oplus \pi_*\cF_2(t).
\]
We now use this observation to define $\cF_3:=\pi^*F_3$.

The proof of this direction concludes by showing that $\cF\cong \cF_1\oplus \cF_2\oplus \cF_3$.
This decomposition obviously holds for the $\cM_F$-module when restricted to the set of points strictly below $p_{\infty}$.
Since $\cF_i$ for $i=1,2,3$ are naturally submodules of $\cF$, it suffices to show that whenever $p\geq p_{\infty}$ and $v_i\in \cF_i(p)$ are chosen such that $v_1+v_2+v_3=0$, then each $v_i=0$.
Of course if $v_1+v_2+v_3=0$, then $v_3\in (\cF_1(p)\oplus \cF_2(p))\cap \cF_3(p)=0$, so $v_3=0$.
Moreover $v_1$ and $v_2$ are the images of $w_1$ and $w_2$ under the maps $\cF_1(q_1 < p)$ and $\cF_2(q_2 < p)$, so the untwisted condition implies that $v_1$ and $v_2$ are independent, hence 0 as well.

A real interval module satisfies the untwisted condition and this property is inherited by direct sums. This proves the converse statement and completes the proof of the theorem.
\end{proof}

\subsection{Injectivity of the Barcode Transform}\label{subsec:injectivity-barcode-transform}

The class of real interval decomposable tree modules is nice for various reasons, but one reason of interest to us is that the \emph{barcode transform} is injective when restricted to real interval decomposable modules.

\begin{theorem}\label{thm:injective-barcode-transform}
Suppose that $\calF:(\cM,\cle) \to \vect$ and $\calF':(\cM',\cle') \to \vect$ are two real interval decomposable tree modules with finitely many intervals in their Remak decompositions. If their associated barcode decorated merge trees $\mathcal{BF}:\cM \to \BCS$ and $\mathcal{BF}':\cM' \to \BCS$ are isomorphic then $\calF\cong \calF'$.
\end{theorem}

To prove the theorem, we introduce some notation. Let $\cI \subset \cM_F$ be a totally ordered subset. We use $\Bbbk_{\calI}:(\cM_F,\cle) \to \vect_{\Bbbk}$
to denote the functor that assigns the ground field $\Bbbk$ to points $p\in \cI$ and where any ordered pair $p\cle q$ in $\cI$ is assigned the identity map. If a decorated merge tree $\mathcal{F}:\cM_F \to \Vect$ is real interval decomposable, then it can be expressed as
\begin{equation}\label{eqn:interval_decomposition}
\mathcal{F} \approx \bigoplus_{\mathcal{R}(\calF)} \Bbbk_{\calI},
\end{equation}
where $\mathcal{R}(\calF):=\{ (\cI,m_{\cI})\}$
indicates the collection of intervals $\cI$ that appear in the Remak decomposition of $\calF$ along with their multiplicity; the sum in \eqref{eqn:interval_decomposition} is taken with multiplicity.

The following proposition is obvious.

\begin{proposition}\label{prop:barcode-decoration-expression}
Suppose that the tree module $\calF:(\cM_F,\cle) \to \vect$ is real interval decomposable. Then the associated barcode decorated merge tree $\mathcal{BF}:\cM_F \to \BCS$
is determined at each point by the formula
\[
    \mathcal{BF}(p) := \{ (U_p \cap \cI, m_{\cI}) \mid (\cI,m_{\cI}) \in \mathcal{R}(\calF)\}
\]
where $U_p$ is the principal up set at $p$ in $\cM_F$. Moreover, when $\cM_F$ is tame, the associated barcode decorated merge tree is completely described as the disjoint union of the barcodes when viewed from each leaf node $v$, i.e.
\[
    \mathcal{BF} = \bigsqcup_{v\in L(\cM)} \mathcal{BF}(v),
\]
where $\mathcal{BF}(v)$ is determined via intersection with the principal up set $U_v$ as indicated above. We note that every bar $(I,m_I)$ in the barcode viewed from $v$ can be expressed as a sum
\[
    (I,m_I,v) = \sum_{\cI\in\mathcal{R}(\calF)} U_v \cap (\calI,m_{\calI}) \qquad \text{where} \qquad \calI \cap U_v = I.
\]
Informally, we may write this as $m_I=m_{\calI_1}+\cdots+m_{\calI_k}$
\end{proposition}

\begin{proof}[Proof of Theorem \ref{thm:injective-barcode-transform}]
Two barcode decorated merge trees are isomorphic if there exists an underlying isomorphism of their generalized merge trees $\varphi:\cM \to \cM'$ and $\psi:\cM' \to \cM$ so that $\mathcal{BF}(x)=\mathcal{BF}'(\varphi(x))$ and $\mathcal{BF}'(x')=\mathcal{BF}(\psi(x'))$ for all $x\in\cM$ and $x'\in\cM'$.
The argument proceeds by proving that $\calF\cong \calF'\circ \varphi$ and symmetrically $\calF'\cong\calF\circ \psi$.
Both arguments would be completely symmetric over the two generalized merge trees, so it suffices to consider the case where $\cM=\cM'$ and where we can assume that $\varphi$ and $\psi$ are identity maps.
We now proceed to showing that if $\mathcal{BF}=\mathcal{BF}'$ then $\calF\cong \calF'$ as $\cM$-modules.

The argument proceeds by showing that every real interval $(\calI,m_{\calI})$ that appears in the Remak decomposition for $\calF$ must appear in the Remak decomposition for $\calF'$ with no intervals left over. This will establish the isomorphism of $\calF$ and $\calF'$.
We start by considering an \emph{oldest} interval $(I,m_I,v)$ in $\calB[\calF]$, this means that if we consider the projection of $I$ to $\R$ via the natural map $\pi:\cM \to \R$ then $\inf\{\pi(I)\}\leq \inf\{\pi(J)\}$ for any $(J,m_J,w)\in \mathcal{BF}$; such an oldest interval exists due to our finiteness assumption.
By Proposition~\ref{prop:barcode-decoration-expression} we know that any $(I,m_I,v)$ has a unique (up to permutation) expression
\[
    m_I = m_{\calI_1} + \cdots + m_{\calI_k}
\]
for $\calI_j\in \mathcal{R}(\calF)$ with $\calI_j\cap U_v = I$.
We claim that if $(I,m_I,v)$ is an oldest interval then there exists an $(\calI_j,m_{\calI_j})\in \mathcal{R}(\calF)$ with $(I,m_I,v)=(\calI_j,m_{\calI_j})$.
Indeed, if there does not exist a unique contributor, then this contradicts the assumption that $(I,m_I,v)$ was oldest.
The interval $\calI_j$ is the oldest contributing factor in the sum for $m_I$ displayed above.
Since the barcode decoration associated to $\calF'$ is equal to the barcode decoration associated to $\calF$, then $(I,m_I,v)$ must also be an oldest interval and there must be the same oldest contributor $\calI_j\in\mathcal{R}(\calF')$.
This proves that $\calF'\cong \calG'\oplus \Bbbk_{\calI_j}^{m_{\calI_j}}$ and of course $\calF\cong \calG \oplus \Bbbk_{\calI_j}^{m_{\calI_j}}$ and $\mathcal{BG}'=\mathcal{BG}$.
The argument repeats with the barcode decorated merge tree $\calB[\calG]$ until it is empty.
\end{proof}

\subsection{Lift Decorated Merge Trees}\label{sec:lift_decorated}

The process of extracting a decorated merge tree from a dataset is an algorithmic challenge that we need to address.  
When a persistent space $F:(\R,\leq) \to \Top$ is a filtration of a simplicial complex, existing software can be used to extract a merge tree $\cM_F$ and a degree-$k$ persistent homology barcode $B$.
One would hope to use this easily accessible data to reverse engineer the underlying decorated merge tree $\cF:\cM_F \to \vect$, for which the persistent homology barcode is the pushforward; i.e., $B = \pi_\ast \cF$, where $\pi:\cM_F \to \R$ is the projection function.
Indeed, algorithms presented in Section \ref{sec:examples} are based on such a reverse engineering process and can be used to produce decorated merge trees from point cloud and filtered graph data.
In this subsection, we provide the theoretical underpinning of these algorithms.

Let $F:(\R,\leq) \to \Top$ be a filtration of a simplicial complex with a tame merge tree $\cM_F$. 
The standard algorithm for extracting the barcode $B$ records a \define{birth simplex} $\sigma$ for each interval $I=[b,d)$ in the barcode \cite{elz2000}. 
This is to say that $\sigma$ is a simplex with filtration value $b$ whose birth creates a representative homology class for $I$.
The \define{birth point of $I$} is the unique point $p \in \cM_F$ with $\pi(p) = b$ and where $p$ is an ancestor of all the vertices of $\sigma$.
Let $q \in \cM_F$ be the unique ancestor of $p$ such that $\pi(q) = d$.

\begin{definition}\label{def:lift_decorated}
    With the notation above, define the \define{lift of $I$} to be the decorated merge tree $\mathcal{I}:\cM_F \to \vect$ defined by
    \[
    \mathcal{I}(r) := \left\{\begin{array}{cc}
    k & p \cle r \cle q \\
    0 & \mbox{otherwise.}
    \end{array}\right.
    \]
    The \define{lift of $B$} is the decorated merge tree
    \[
    \widehat{\cF} := \bigoplus_{I \in B} \mathcal{I}.
    \]
    A decorated merge tree is called a \define{lift decorated merge tree} if it is obtained as the lift of some barcode.
\end{definition}

Based on the above discussion, it is straightforward to determine the lift decorated merge tree directly from the filtered simplicial complex. However, the lift DMT is not guaranteed to be isomorphic to the \emph{true} DMT $\cF:\cM \to \vect$; indeed, we have $\widehat{\cF} \approx \cF$ if and only if $\cF$ is real  interval decomposable, in the sense of Definition \ref{def:real_interval_decomposable}. The task which we aim to address for the rest of this subsection is to determine a condition on the filtration $F$ which
\begin{itemize}
    \item can be verified directly from the merge tree $\cM_F$ and degree-$k$ barcode $B$, and
    \item implies that $\widehat{\cF} \approx \cF$. 
\end{itemize}
Our proposed condition is the following.

\begin{definition}\label{def:disjointness}
    Let $F:(\R,\leq) \to \Top$ be a filtered simplicial complex with merge tree $\cM_F$ and degree-$k$ barcode $B$. 
    We say that $F$ has the \define{$H_k$-disjointness property} if for any pair of bars $I = [b,d)$ and $I'=[b',d')$ in $B$ with incomparable birth points $p$ and $p'$, their death times are less than their merge height (Definition~\ref{defn:LCA-merge-height}), i.e.
    \[
    \min\{d,d'\} < \mathrm{merge}(p,p').
    \]
    Intuitively, this means that the lifts of $I$ and $I'$ have disjoint support.
\end{definition}

\begin{proposition}\label{prop:disjointness}
    If a filtered simplicial complex $F:(\R,\leq) \to \Top$ has the $H_k$-disjointness property, then the lift decorated merge tree and the concrete decorated merge tree are isomorphic, i.e.~$\widehat{\cF} \approx \cF$.
\end{proposition}

This result says that $H_k$-disjointness is sufficient to guarantee that the lifting procedure produces a valid decorated merge tree. Our proof is straightforward but technical, so we delay it to Appendix \ref{sec:technical_proofs_disjointness}.

\section{Computing Interleaving Distances}\label{sec:computational_aspects}

The goal of this section is to develop a practically feasible algorithm for approximating the matching distance between barcode decorated merge trees (Definition \ref{defn:barcode_decorated_merge_tree}). This relies on the notion of \emph{labeling} a DMT.

\subsection{Labeled Distance}\label{sec:labeled_distance}

Throughout this section, we consider generalized merge trees $\cM_F$ arising from persistent spaces $F:(\R,\leq) \rightarrow \Top$ such that the persistent set $\pi_0 \circ F$ is constructible and $\cM_F$ is tame, as in Proposition \ref{prop:constructible-PFD}. 
The finite set of leaves of $\cM_F$ is denoted $L(\cM_F)$. 
The canonical projection function $\cM_F \to \R$ is denoted $\pi_F$. 

\begin{remark}
The results of this section do not require that decorated merge trees arise from persistent spaces at all; they could be derived formally treating tree modules as the fundamental objects. 
\end{remark}

We restrict our attention to the case of barcode decorated merge trees $\mathcal{B}_F:\cM_F \to \BCS$ which are determined by restriction and which have leaves, i.e.~$\mathcal{B}_F$ is a leaf-decorated merge tree. 
Recall that we denote an element of $\BCS$ as a multiset $\{(I,m_I)\}_{I \in \mathcal{I}}$, where $I$ is an interval and $m_I$ is its multiplicity. We make the (realistic, in practice) assumption that all barcodes are finite.

Inspired by a similar result for merge trees in \cite{gasparovic2019intrinsic}, we will now explain how the bottleneck distance between leaf decorated merge trees can be expressed in terms of matrices obtained from labelings of merge trees.

\begin{definition}
 A \define{labeling} of a merge tree $\cM_F$ is a map 
 \[
 \lambda_F: [n]:=\{1,\ldots,n\} \rightarrow \cM_F,
 \]
 for some integer $n \geq 1$, which is surjective onto the set of leaves $L(\cM_F)$; that is, $L(\cM_F) \subseteq \mathrm{Im}(\lambda_F)$.
\end{definition}
  
  To each labeling $\lambda_F: [n] \rightarrow \cM_F$ one associates an $n\times n$ matrix $\Lambda_F$ of merge heights (Definition~\ref{defn:LCA-merge-height}), referred to as the \emph{least common ancestor (LCA) matrix} and defined by
\begin{equation}\label{eqn:ultramatrix}
\Lambda_F(i,j) = \pi_F(\mathrm{LCA}_F(\lambda_F(i),\lambda_F(j))) = \mathrm{merge}_F(\lambda_F(i),\lambda_F(j)).
\end{equation}
It is shown in \cite[Proposition 4.1]{gasparovic2019intrinsic} that interleaving distance between merge trees can be expressed in terms of the $\ell_\infty$ distance between these LCA matrices. To adapt this result to the decorated merge tree setting, we introduce a more involved objective function.

\begin{definition}
Let $\mathcal{B}_F:\cM_F \to \BCS$ and $\mathcal{B}_G:\cM_G \to \BCS$ be leaf decorated merge trees. Define the \define{matching cost} of labelings $\lambda_F:[n] \to \cM_F$ and $\lambda_G:[n] \to \cM_G$ to be
\begin{equation}\label{eqn:DMT_matching_cost}
\mathrm{cost}(\lambda_F,\lambda_G) = \max \left\{
\left\| \Lambda_F - \Lambda_G \right\|_\infty, 
\max_i d_B\left(\mathcal{B}_F(\lambda_F(i)),\mathcal{B}_G(\lambda_G(i))\right)\right\},
\end{equation}
where $\|\cdot\|_\infty$ is the $\ell_\infty$-norm on $n \times n$ matrices.
\end{definition}

\begin{proposition}\label{prop:labeled_distance}
The decorated bottleneck distance between leaf decorated merge trees $\mathcal{B}_F:\cM_F \to \BCS$ and $\mathcal{B}_G:\cM_G \to \BCS$ is given by 
\[
d_B(\mathcal{B}_F,\mathcal{B}_G) = \min_{\lambda_F,\lambda_G} \mathrm{cost}(\lambda_F,\lambda_G),
\]
where the minimum is taken over labelings $\lambda_F$ and $\lambda_G$ with common domain $[n]$, where $n$ is the sum of the number of leaves in $\cM_F$ and the number of leaves in $\cM_G$.
\end{proposition}

The proof adapts the proof of \cite[Proposition 4.1]{gasparovic2019intrinsic} to this new setting. We provide a sketch of the idea and omit technical, but ultimately routine, details.

\begin{proof}[Proof Sketch.]
Throughout the proof sketch, fix leaf decorated merge trees $\mathcal{B}_F:\cM_F \to \BCS$ and $\mathcal{B}_G:\cM_G \to \BCS$. Suppose that $\cM_F$ has $k$ leaves and $\cM_G$ has $\ell$ leaves.

First assume that we have $\epsilon$-compatible maps $(\phi,\psi)$ of $\cM_F$ and $\cM_G$; recall from Definition \ref{defn:matchings_DMTs} that we slightly abuse notation and conflate the notion of natural transformation with that of $\epsilon$-compatible maps. Also assume that
\[
d_B(\mathcal{B}_F(u),\mathcal{B}_G(\phi(u)),d_B(\mathcal{B}_G(v),\mathcal{B}_F(\psi(v))) < \epsilon.
\]
We construct labelings $\lambda_F$ and $\lambda_G$ on the common set of labels $[k + \ell]$ as follows.

\begin{enumerate}
    \item Label each leaf of $\cM_F$ with a unique element of $\{1,\ldots,k\}$; this defines $\lambda_F$ on $[k]$. Each label $i$ is also assigned to $\cM_G$ as $\lambda_G(i) = \psi(\lambda_F(i))$.
    \item Similarly, for each $j \in \{k+1,\ldots,k+\ell\}$ define $\lambda_G(j)$ to be a unique leaf of $\cM_G$ and define $\lambda_F(j) = \psi(\lambda_G(j))$.
\end{enumerate}
One must then show that the labelings $\lambda_F$ and $\lambda_G$ defined above have the property that, for any $i,j \in [k +\ell]$,
\begin{equation}\label{eqn:labeling_claim_0}
    \left|\pi_F(\mathrm{LCA}_F(\lambda_F(i),\lambda_F(j))) - \pi_G(\mathrm{LCA}_G(\lambda_G(i),\lambda_G(j)))\right| < \epsilon
\end{equation}
and
\begin{equation}\label{eqn:labeling_claim_1}
d_B(\mathcal{B}_F(\lambda_F(i)),\mathcal{B}_G(\lambda_G(i))) < \epsilon. 
\end{equation}

Conversely, suppose that we have labelings $\lambda_F$ and $\lambda_G$ such that $\mathrm{cost}(\lambda_F,\lambda_G) < \epsilon$. For each $i \in [k + \ell]$ set $u_i = \lambda_F(i)$ and $v_i = \lambda_G(i)$.

\begin{enumerate}
    \item We first define the map $\phi:\cM_F \to \cM_G$. Let $x \in \cM_F$ and choose an arbitrary labeled point $u_i \in \cM_F$ with $u_i \leq x$ (always possible because all leaves of $\cM_F$ are labeled). Define $\phi(x) \in \cM_G$ to be the unique ancestor of $v_i$ such that $\pi_G(\phi(x)) = \pi_F(x) + \epsilon$.

    \item The map $\psi:\cM_G \rightarrow \cM_F$ is constructed similarly.
\end{enumerate}

We first note that the maps $\phi$ and $\psi$ are well defined. Indeed, if $u_i$ and $u_j$ are a pair of labeled points in $\cM_F$ with $u_i,u_j \leq x$, then $\mathrm{LCA}_F(u_i,u_j) \leq x$. This means that $\pi_G(\mathrm{LCA}_G(v_i,v_j)) \leq \pi_F(x) + \epsilon$, by our assumption on the cost of the labelings $\lambda_F$ and $\lambda_G$, so that $v_i$ and $v_j$ have the same ancestor at height $\pi_F(x) + \epsilon$. By the same reasoning, $\psi$ is well defined. The proof is completed by showing that the maps $(\phi,\psi)$ define an $\epsilon$-matching of barcode decorated merge trees. 
\end{proof}

\subsection{Approximating Merge Tree Interleaving via Gromov-Wasserstein Distance}

Computing interleaving distance between merge trees is known to be NP-Hard \cite{agarwal2018computing}. A dynamic programming approach to computing interleaving distance is taken in \cite{farahbakhsh2019fpt}, illustrating that the computation is at least fixed parameter tractable (roughly, precise upper bounds on interleaving distance are computable in time with polynomial growth in the number of the nodes, provided one has a uniform upper bound on node degrees). An alternative and more efficient method for analyzing merge trees is to apply  \cite[Proposition 4.1]{gasparovic2019intrinsic}: given any labeling of merge trees, one obtains an upper bound on their interleaving distance, so one could rely on a good heuristic for estimating optimal labelings. Such a heuristic, based on replacing the problem with the simpler one of finding optimal bipartite graph matchings, is proposed in \cite{yan2019structural}, yielding a comprehensive framework for statistical analysis of merge tree ensembles. 

In this section, we propose an alternative heuristic for estimating optimal labelings using the concept of Gromov-Wasserstein (GW) distances. Different formulations of GW distances were introduced around a decade ago by Sturm \cite{sturm2006geometry,sturm2012space}, to study abstract convergence of sequences of metric measure spaces, and M\'{e}moli \cite{dgh-sm,memoli2011gromov}, with a view towards applications to object matching. GW distances have found recent popularity in the machine learning community \cite{alvarez2019towards,chapel2020partial,xu2019scalable,xu2020learning}, largely due to the observation by Peyr\'{e} et al.\ \cite{peyre2016gromov} that they can be used to measure distances between general kernel matrices; this point of view was formalized mathematically in \cite{DBLP:journals/corr/abs-1808-04337}. We now briefly present GW distances from a computational point of view---see the references above more thorough treatments.

\begin{definition}[cf.\ \cite{DBLP:journals/corr/abs-1808-04337}]
A \define{(finite) measure network} is a pair $C = (C,\mu)$, where $C$ is an $n\times n$ matrix and $\mu$ is a length-$n$ probability vector (that is, $\sum \mu(i) = 1$ and $\mu(i) \geq 0$). A \define{coupling} between probability vectors $\mu_1 \in \mathbb{R}^n$ and $\mu_2 \in \mathbb{R}^m$ is a matrix $\nu \in \mathbb{R}^{n \times m}$ such that $\nu(i,j) \geq 0$ and 
\[
\sum_i \nu(i,j) = \mu_2(j), \qquad \sum_j \nu(i,j) = \mu_1(i).
\]
That is, $\nu$ is a joint probability distribution with marginals $\mu_1$ and $\mu_2$. The set of couplings between fixed $\mu_1$ and $\mu_2$ is denoted $\mathcal{C}(\mu_1,\mu_2) \subset \R^{n \times m}$. 

The \define{Gromov-Wasserstein $p$-distance} between  finite measure networks $(C_1,\mu_1)$ and $(C_2,\mu_2)$ is defined by
\begin{equation}\label{eqn:gw_p_distance}
d_{\mathrm{GW},p}\left((C_1,\mu_1),(C_2,\mu_2)\right) = \min_{\nu \in \mathcal{C}(\mu_1,\mu_2)} J_p(\nu)^{1/p},
\end{equation} 
where
\begin{equation}\label{eqn:gw_p_loss}
J_p(\nu) = \sum_{i,j,k,\ell} \left(C_1(i,k) - C_2(j,\ell)\right)^p \nu(i,j) \nu(k,\ell)
\end{equation}
denotes the \define{Gromov-Wasserstein $p$-distortion functional}. 
\end{definition}

To handle computational aspects of merge tree interleaving, we introduce the following discrete representation of a merge tree.

\begin{definition}[cf.\ \cite{gasparovic2019intrinsic}]
A \define{computational merge tree} is a (discrete) graph together with a height function on its nodes satisfying the following conditions:
\begin{enumerate}
    \item the graph must be a tree---that is, there is a unique edge path between any two nodes;
    \item exactly one of the leaf nodes (called the \define{root}) of the graph gets assigned height $\infty$;
    \item adjacent nodes do not have equal function value; and
    \item every non-root node has exactly one neighbor with higher function value.
\end{enumerate}
A \define{computational decorated merge tree} is a computational merge tree together with an assignment of a barcode to each of its nodes.
\end{definition}

\begin{remark}
In what follows, we use the notation $\cM_F$ for a computational merge tree and $\calB_F:\cM_F \to \BCS$ for a computational decorated merge tree. This is in agreement with the notation used for their non-computational counterparts earlier in the paper, but the distinction should be clear from context. We also generally drop the ``computational" qualifier.
\end{remark}

\begin{Algo}\label{algo:merge_tree_interleaving}
The GW framework is incorporated into the pipeline for estimating the interleaving distance between merge trees $\cM_F$ and $\cM_G$ as follows (see also the schematic of the pipeline in Figure \ref{fig:interleaving_pipeline}):

\begin{enumerate}
    \item Probability distributions $\mu_F$ and $\mu_G$, respectively, are chosen for the nodes of the trees. We use uniform distributions in our experiments
    \item An optimal coupling $\nu$ of $\mu_F$ and $\mu_G$ is estimated by numerically solving the GW problem \eqref{eqn:gw_p_distance}.
    Intuitively, the GW problem promotes large values of $\nu(i,j)$ when node $i$ of $\cM_F$ and node $j$ of $\cM_G$ are structurally similar.
    \item We use $\nu$ to estimate interleaving maps of $\cM_F$ and $\cM_G$: for each leaf $u$ in $\cM_F$, we locate the maximum entry of the row of $\nu$ corresponding to this leaf and define this to be $\phi(u) \in \cM_G$. We likewise define a map $\psi$ from the leaf set of $\cM_G$ to $\cM_F$ by examining columns of $\nu$.
    \item We use $\phi$ and $\psi$ to construct labelings $\lambda_F$ and $\lambda_G$ with domain $[k + \ell]$ of the merge trees as in the algorithm in the proof of Proposition \ref{prop:labeled_distance}, where $k$ is the number of leaves in $\cM_F$ and $\ell$ is the number of leaves in $\cM_G$.
    \item From $\lambda_F$ and $\lambda_G$, we construct matrices $\Lambda_F$ and $\Lambda_G$  and compute $\|\Lambda_F - \Lambda_G\|_\infty$, yielding a principled upper estimate of interleaving distance.
\end{enumerate}
\end{Algo}

\begin{figure}
    \centering
    \includegraphics[width = 0.8\textwidth]{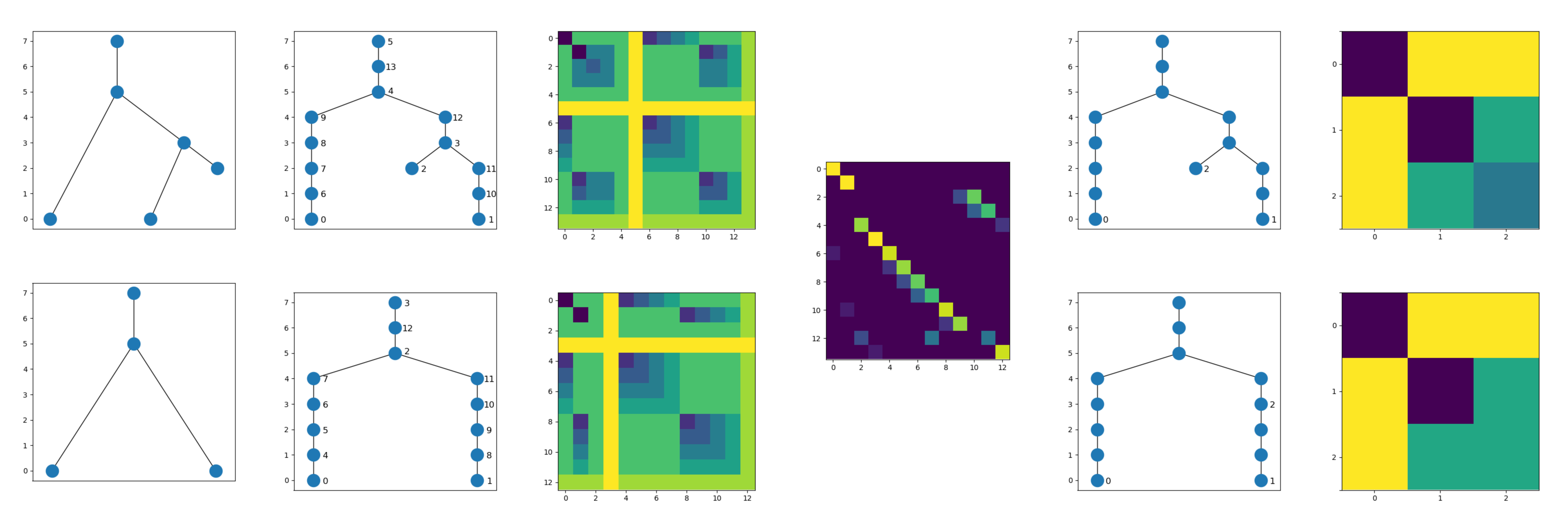}
    \caption{Pipeline for estimating interleaving distance between merge trees. From left to right: we aim to compute the interleaving distance between merge trees shown in the left column. We sample the merge trees with a user-defined mesh (here the mesh is equal to 1).
    We assign arbitrary labels to the nodes in each tree independently; LCA matrices for each merge tree (with respect to these labels) are shown in the third column. The fourth column shows an optimal coupling between the sampled nodes of the merge tree. From the coupling, we estimate an optimal labeling of the merge trees (fifth column), and this easily yields an upper estimate on the interleaving distance between the merge trees.}
    \label{fig:interleaving_pipeline}
\end{figure}

\begin{remark}
Computing GW distance is an instance of quadratic programming problem with a nonconvex objective function and is therefore NP-Hard to compute exactly \cite{memoli2011gromov}. However, since the space of couplings $\mathcal{C}(\mu_1,\mu_2)$ forms a convex polytope, it is possible to approximate GW distance via Frank-Wolfe-style projected gradient descent; our computations of GW distance will be handled by the Python Optimal Transport package \cite{flamary2017pot}. Gradient updates have computational complexity $O(n^3 \log(n))$ when $p=2$ \cite{peyre2016gromov}, so we focus on this case for the sake of efficiency.

The quality of the gradient descent-based estimation in step (2) of the algorithm can be improved by upsampling the trees---i.e., adding degree-$2$ nodes at user-specified heights to better approximate the continuous nature of the true merge trees. We have found empirically that this procedure for estimating interleaving distance from above tends to give meaningful labelings of merge trees. This is illustrated by an experiment on synthetic data in Section \ref{sec:merge_tree_interleaving_experiment}. 
\end{remark}

\begin{remark}\label{remark:connection_to_GW_methods}
Similar applications of Gromov-Wasserstein distance have recently appeared in the literature. In \cite{li2021sketching}, a framework for summarizing sets of merge trees is developed using the Riemannian structure of the GW metric on measure networks developed in \cite{chowdhury2020gromov}. There, the authors use GW distance to align merge trees, but the measure network representation of the merge trees used there is different than the one proposed here and has a less clear connection to \cite{gasparovic2019intrinsic} and the computation of interleaving distance. The authors of \cite{memoli2021ultrametric} study a variant of GW distance on the space of \emph{ultrametric spaces} (metric spaces which satisfy a stronger version of the triangle inequality). Applications to merge tree interleaving are not directly discussed in \cite{memoli2021ultrametric}, but our matrix representations of merge trees are inherently utilizing an underlying ultrametric structure (or, using the terminology of \cite{memoli2021ultrametric}, ultradissimilarity structure), which is not the usual geodesic metric one considers when representing a merge tree as a metric tree. Similar ideas for comparing metric spaces go back to \cite{smith2016hierarchical} and \cite{memoli2019gromov} (the latter describing specific connections to merge tree interleavings), but to our knowledge the specific algorithm and implementation described here for approximating merge tree interleaving is novel.
\end{remark}

\subsection{Matchings of DMTs via Fused Gromov-Wasserstein Matchings}\label{sec:fused_gromov_wasserstein} 

The GW framework described above has thus far only been applied to estimate interleavings between (undecorated) merge trees. In order to incorporate barcodes into comparisons between \emph{decorated} merge trees, we employ the more general Fused Gromov-Wasserstein (FGW) framework~\cite{vayer2020fused}. The FGW framework is used to compare measure networks with additional structure. 

\begin{definition}
Let $Z = (Z,d_Z)$ be a metric space. A \define{$Z$-structured (finite) measure network} is a triple $(C,\mu,B)$ consisting of a finite measure network $(C,\mu)$, where $C \in \R^{n \times n}$, together with an $n$-tuple $B$ of points in $Z$. 

Let $(C_1,\mu_1,B_1)$ and $(C_2,\mu_2,B_2)$ be $Z$-structured measure networks. For $\alpha \in [0,1]$, we define the \define{Fused Gromov-Wasserstein (FGW) distance} for parameter $\zeta \in [0,1]$ as
\begin{equation}\label{eqn:FGW}
d_{\mathrm{FGW},\zeta}((C_1,\mu_1,B_1),(C_2,\mu_2,B_2))^2 = \min_{\nu \in \mathcal{C}(\mu,\nu)} \left((1-\zeta) I_2(\nu) + \zeta J_2(\nu)\right),
\end{equation}
where $J_2$ is the GW $p$-loss \eqref{eqn:gw_p_loss} with $p=2$ and 
\[
I_2(\nu) = \sum_{i,j} d_Z(B_1(i),B_2(j))^2 \nu(i,j)
\]
is the standard $2$-Wasserstein loss from classical optimal transport (see, e.g., \cite{villani2008optimal}).
\end{definition}

The FGW framework is used to augment Algorithm \ref{algo:merge_tree_interleaving} to estimate matching distance between DMTs $\calB_F:\cM_F \to \BCS$ and $\calB_G:\cM_G \to \BCS$. The difference is that the optimal coupling of step (2) is obtained by solving \eqref{eqn:FGW} for some choice of hyperparameter $\zeta$, where the functional $J_2$ is the GW $2$-distortion functional with respect to the measure networks $(\bar{\Lambda}_F,\mu_F)$ and $(\bar{\Lambda}_G,\mu_G)$, exactly as in the previous subsection. The $I_2$ term is Wasserstein $2$-loss with respect to bottleneck distance between node barcodes. Examples of FGW-based estimation of DMT matching distance are provided in Sections  \ref{sec:clustering_point_clouds} and \ref{sec:network_matching}.

\section{Algorithmic Details and Examples}\label{sec:examples}

In this section, we outline theoretical and practical aspects of generating and comparing decorated merge trees. 
We provide several computational examples coming from real and synthetic data. 
Implementations (in Python) of all of these experiments as well as source code are freely available in our GitHub repository\footnote{https://github.com/trneedham/Decorated-Merge-Trees}. 
The code uses standard Python data science packages (\texttt{scikit-learn} \cite{scikit-learn}, \texttt{scipy} \cite{2020SciPy-NMeth}, etc.), as well as more specialized packages for topological data analysis (\texttt{gudhi} \cite{gudhi:urm}, \texttt{giotto-tda} \cite{tauzin2020giottotda}, \texttt{ripser} \cite{1908.02518} and \texttt{scikit-tda} \cite{scikittda2019}) and optimal transport (Python Optimal Transport, \texttt{pot} \cite{flamary2017pot}). 

\subsection{Merge Tree Interleaving Distances}\label{sec:merge_tree_interleaving_experiment}

To test the reliability of our method for estimating merge tree interleaving distances, we create a simple classification experiment (summarized in Figure \ref{fig:scalar_classification}). We define a parametric model for generating merge trees by considering merge trees $\mathcal{M}_F$ given by sublevel-set filtrations $F$ of functions $f:[0,1] \rightarrow \R$ given by
\begin{equation}\label{eqn:merge_tree_model}
f(t) = \sin(\rho_1 \pi t) + \cos(\rho_2 \pi t) + g(t),
\end{equation}
where $\rho_j$ are parameters and $g(t) \in [0,1/2]$ is uniformly distributed random additive noise. We construct a dataset of 120 random merge trees by drawing 10 samples for each combination of parameter choices $\rho_1 \in \{1,2,4,8\}$ and $\rho_2 \in \{1,3,5\}$. The random additive noise has the effect of generating many spurious local minima for the functions which yield extraneous leaves in the merge trees, making the matching process nontrivial.
We compute a $120 \times 120$ pairwise distance matrix using our interleaving distance estimator with uniform node weights and height sampling mesh size equal to $0.5$. The leave-one-out nearest neighbor classification rate for this distance matrix is 98.33\%, indicating that our interleaving distance computation provides a meaningful comparison for the merge tree dataset. 

Since the labelings in the GW-based interleaving computation are obtained via optimization of an $\ell_2$ loss, we also give a small tweak to the interleaving distance formulation and compute the \define{$\ell_2$-interleaving distance} between merge trees $\cM_F$ and $\cM_G$, defined to be 
\[
\min_{\lambda_F,\lambda_G} \|\Lambda_F - \Lambda_G\|_2,
\]
where the minimum is, as usual, over labelings of the merge trees and $\|\cdot\|_2$ is the standard $\ell_2$ norm on matrices. The classification rate for this distance is 100\%. Exploring theoretical properties of $\ell_p$ versions of merge tree interleaving distance will be a direction of future work.

\begin{figure}
    \centering
    \includegraphics[width = 0.8\textwidth]{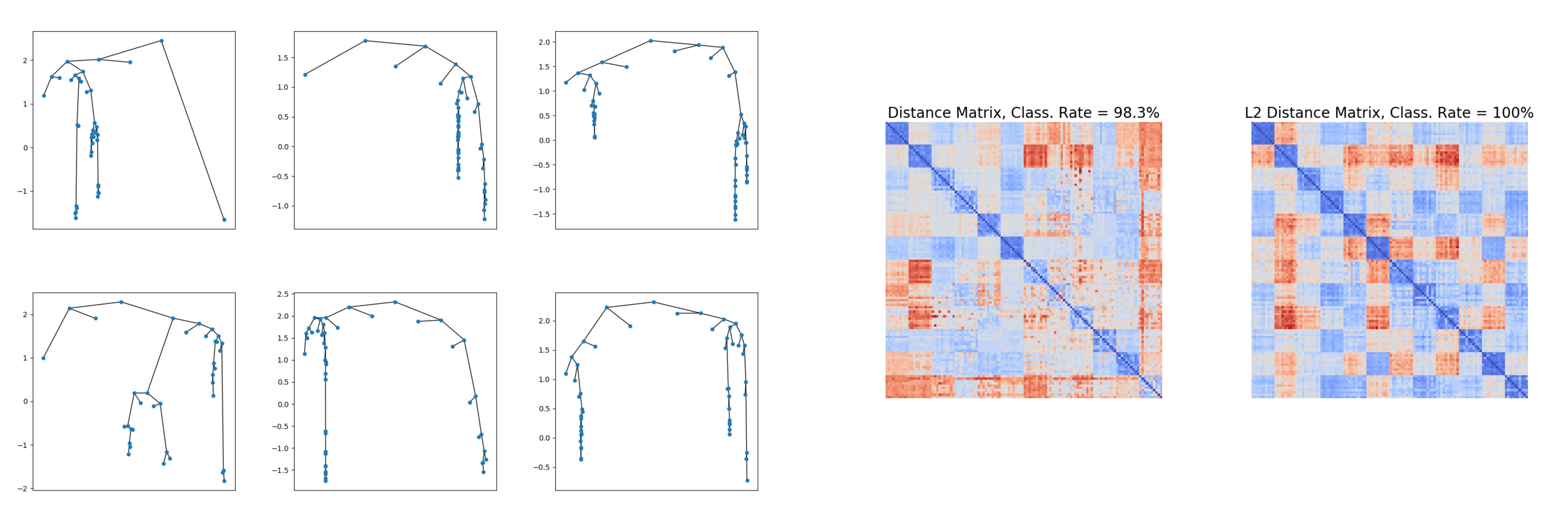}
    \caption{Merge tree classification experiment. The merge trees shown on the left are samples from 6 different classes generated from the family of scalar functions \eqref{eqn:merge_tree_model}. Pairwise distance matrices and leave-one-out nearest neighbor classification scores for $\ell_\infty$ and $\ell_2$ interleaving distances are shown on the right.}
    \label{fig:scalar_classification}
\end{figure}

\subsection{Decorated Merge Trees for Point Clouds}\label{sec:point_clouds_computation}

Vietoris-Rips complexes of point clouds in metric spaces provide a ubiquitous source for persistence diagrams in topological data analysis. In this subsection we provide details on the construction of decorated merge trees and on the methods used to visualize them. These constructions fit into the general theory by considering the DMTs as being generated by a filtered simplicial complex, viewed as a persistent space as in Example \ref{ex:sublevel_set_simplicial}. Throughout the rest of this section, we will use the term DMT to refer specifically to tame leaf decorated merge trees (Definitions \ref{defn:barcode_decorated_merge_tree} and \ref{defn:constructible-persistent-set}).

Let $X$ be a finite subset of a metric space (in our examples, the metric space is $\R^n$, but the procedure we describe here can be performed on any distance matrix). From the Vietoris-Rips complex of $X$, we are easily able to compute a (discrete representation of a) merge tree $\cM_F$ ($F$ denoting the Vietoris-Rips filtration for the complete simplex on $X$)---i.e., a single linkage hierarchical clustering dendrogram of $X$ \cite{carlsson2013classifying}---via off-the-shelf functions available in standard packages (e.g., \texttt{scipy}). 

It remains to determine barcodes for the leaves of $\cM_F$. A straightforward algorithm produces a lift decorated merge tree, as was considered in Section \ref{sec:lift_decorated}. The inputs are the tree $\cM_F$ and a degree-$k$ Vietoris-Rips barcode, generated by existing software (\texttt{ripser} or \texttt{gudhi}). For convenience, we assume the generic condition that all pairwise distances in $X$ are distinct, but this condition can be removed with a bit more work. The basic idea of the algorithm is that we can \emph{decorate} the merge tree with the bars of the barcode by locating the unique (under our generic assumption on the distances) point in the tree where each bar in the barcode is born. Recall from Section \ref{sec:lift_decorated} that the lift decorated merge tree is not necessarily isomorphic to the true decorated merge tree---this will be the case if and only if the true DMT is real interval decomposable (Definition \ref{def:real_interval_decomposable}). Fortunately, Proposition \ref{prop:disjointness} gives a simple method for certifying correctness of the lift decorated merge tree. If a given merge tree does not meet this certification, the true leaf decorations of the merge tree for the dataset can be computed via a more laborious algorithm, where the idea  is to construct a modification of the Vietoris-Rips complex for each leaf and to then assign the associated persistent homology barcode to that leaf. This is guaranteed to produce the correct leaf barcodes, but comes with higher computational cost and less intuitive visualizations.

To visualize the lift decorated merge trees, we draw the merge tree together with offset edges indicating homology bars. Examples of decorated merge trees generated via this algorithm are shown in Figure \ref{fig:DMT_first_example}. We can also visualize merge trees decorated with bars from several higher homology degrees---see Figure \ref{fig:multiple_degrees}. Our visualizations employ some tricks to improve legibility:
\begin{enumerate}
    \item on the DMT, we typically only plot bars from the barcode which have persistence longer than a (tunable, user-defined) threshold;
    \item we truncate the merge tree itself at a relatively low (tunable, user-defined) height and then extend the truncated edges to zero.
\end{enumerate}
Both of these options are used to simplify the visualizations by removing unstable, noisy topological features (transient higher-dimensional features and changes in connectivity at small radii, respectively). Moreover, the simplified DMTs are frequently used in matching distance computations---since this is an $\ell_\infty$-type distance, it is stable under these small adjustments, but the simplifications provide a significant computational speedup.

\begin{figure}
    \centering
    \includegraphics[width = 0.8\textwidth]{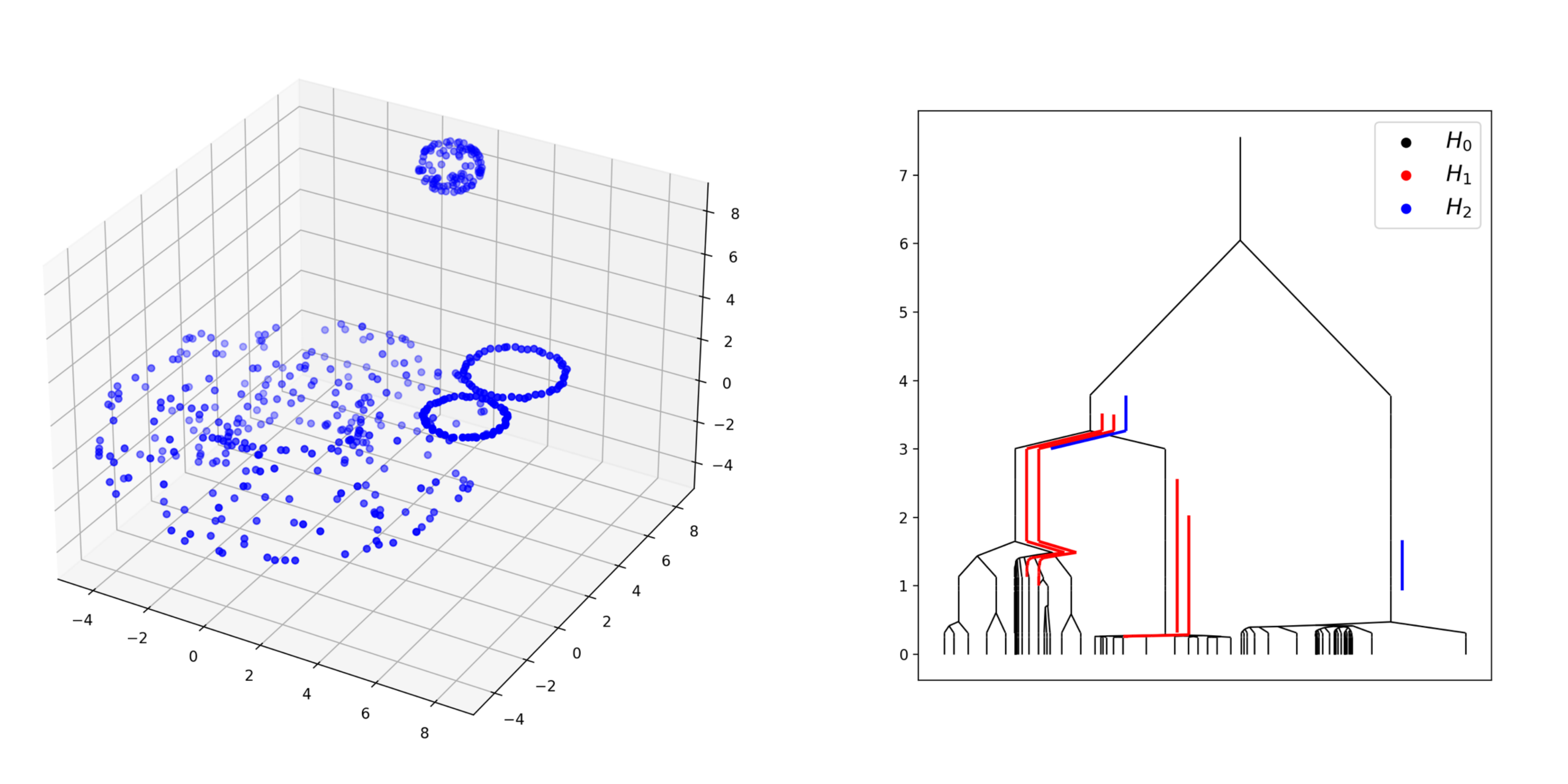}
    \caption{Summarizing the topology in all degrees. The figure on the left shows a toy point cloud dataset consisting of points sampled from a torus, a sphere and a figure-8. The figure on the right shows the associated merge tree decorated with both degree-1 and degree-2 persistent homology data (the merge tree and the persistence diagrams have been thresholded for visual clarity---see text). This summarizes the topology of the dataset in all degrees simultaneously, essentially giving a complete description of its overall topology.}
    \label{fig:multiple_degrees}
\end{figure}

Figure \ref{fig:indecomposable_example} shows an example of a pointcloud where the hypotheses of Proposition \ref{prop:disjointness} (i.e. $H_1$-disjointness) fail. In fact, the true DMT for this dataset is not real interval decomposable and the true leaf barcodes are determined by the more computationally taxing algorithm. In this case, a different technique for visualizing decorated merge trees must be used---see the caption to the figure for details. We found in practice that almost all of our other  examples naturally satisfied the hypotheses of Proposition \ref{prop:disjointness}.

\begin{figure}
    \centering
    \includegraphics[width = 0.8\textwidth]{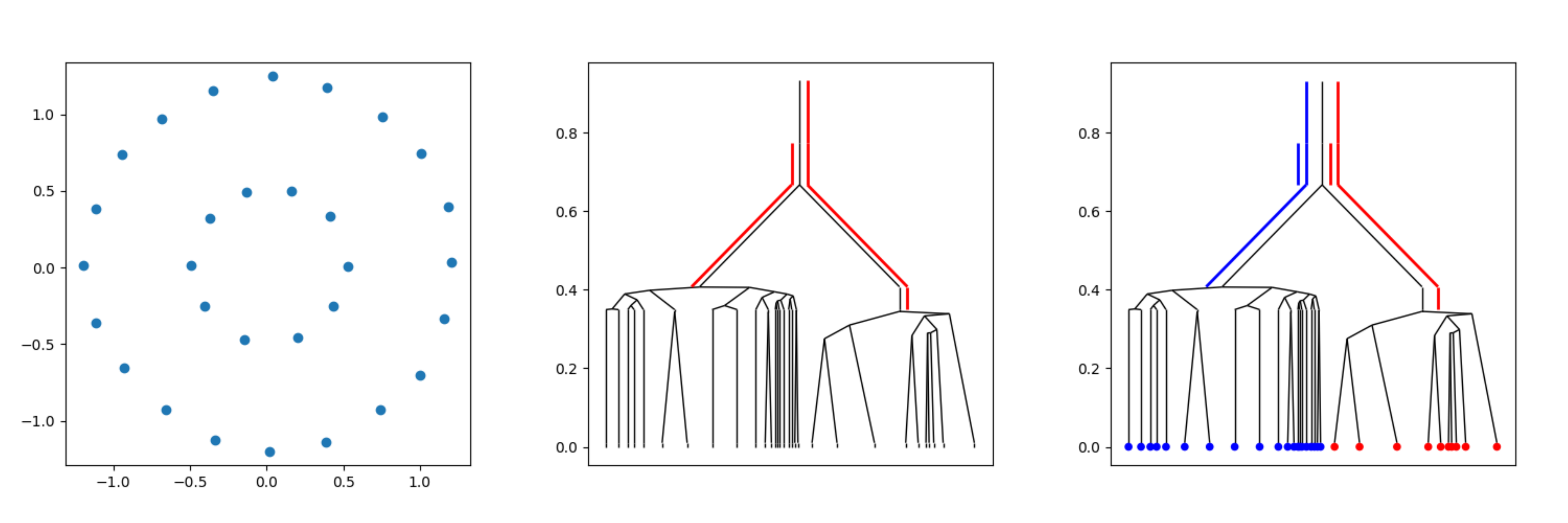}
    \caption{A point cloud with non-real interval decomposable merge tree. The figure on the left shows a dataset in $\R^2$. The DMT shown in the middle is the one produced by the lifting algorithm---observe that this does not satisfy the hypotheses of Proposition \ref{prop:disjointness}. In fact, the output of the more involved algorithm disagrees with the lifted version, meaning that the true DMT for the point cloud is not real interval decomposable. The true DMT is depicted in on the right. In the non-real interval decomposable case, we must visualize the leaf barcode for each leaf independently. All leaves colored blue (respectively, red) have the same barcode, also depicted in blue (resp., red).}
    \label{fig:indecomposable_example}
\end{figure}

\subsection{Sliding Window Embeddings} Established in \cite{perea2015sliding}, sliding window embeddings provide a method for applying the techniques of point cloud TDA to 1-dimensional signals. Given a function $f:[0,T] \to \R$, one defines for each pair of parameters $d \in \mathbb{N}$ and $\tau \in (0,T/d)$ the \define{sliding window embedding} (also referred to as the \define{Takens embedding}, in reference to Takens's theorem from dynamical system theory \cite{takens1981detecting}) as
\begin{align*}
    \mathrm{SW}_{d,\tau}(f):[0,T-d\tau] &\to \R^{d + 1} \\
    t &\mapsto (f(t), f(t + \tau), f(t + 2\tau), \cdots, f(t + d\tau)).
\end{align*}
The sliding window embedding assigns a $(d+1)$-dimensional Euclidean point cloud to each collection of finite samples of $f$. Methods of TDA can then be applied to this point cloud; this approach to signal analysis has found success in applications such as wheeze detection in audio recordings of breathing \cite{emrani2014persistent} and action recognition from scalar measurements of joint movement \cite{venkataraman2016persistent}.

When performing this sort of analysis, the focus is typically on degree-1 homological features, which indicate periodicity in the signal. We hypothesize that there may also be interesting degree-0 features in the case that the signal includes a sudden shift. We illustrate this behavior in the examples shown in Figure \ref{fig:takens_embedding}. Given a signal with an apparent shift, we:
\begin{enumerate}
    \item construct a point cloud via a sliding window embedding; parameters $d$ and $\tau$ are chosen automatically via statistical tools in the \texttt{giotto-tda} package;
    \item subsample the resulting point cloud by density---this has the effect of accentuating disconnected clusters, which correspond to the pre- and post-shift regimes in the original signal;
    \item create a decorated merge tree from the result.
\end{enumerate}
The examples in Figure \ref{fig:takens_embedding} (on both synthetic and real time series data) show that the DMTs resulting from this process uncover interesting features of the signals. 

From a real interval decomposable DMT, we obtain a simple method for locating the connected component of a dataset which contains a given higher degree homology cycle. For an interval $I$ in the higher degree barcode of the dataset, the points belonging to the connected component of the homology cycle are those corresponding to descendent leaves of the birth point of $I$ in the merge tree. Using our cycle component location procedure, we can determine portions of the signals which generate various degree-1 homology classes---these are indicated by color coding in Figure \ref{fig:takens_embedding}.

\begin{figure}
    \centering
    \includegraphics[width = 0.8\textwidth]{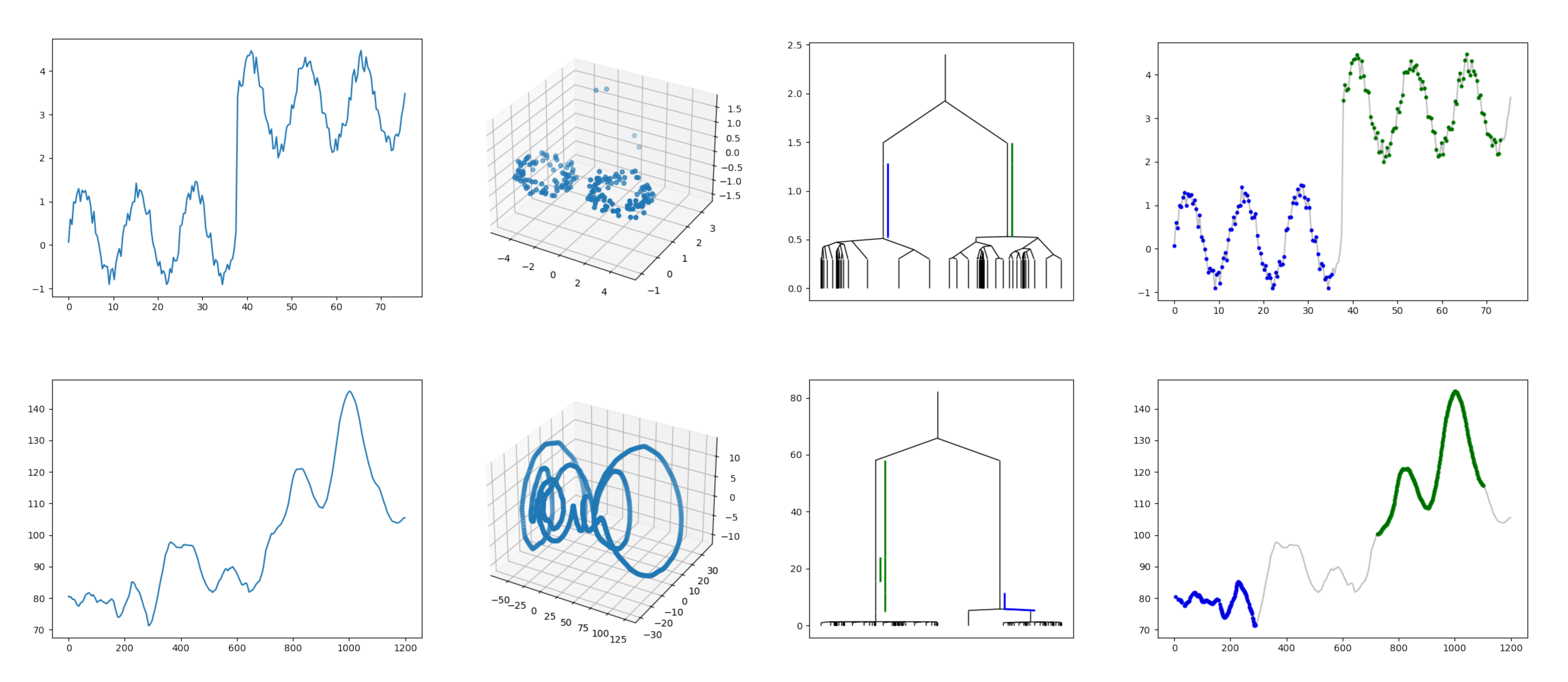}
    \caption{Decorated Merge Trees from Takens embeddings. Each row shows a different experiment. Starting with a time series (left), we produce an embedded point cloud via a Takens embedding (parameters determined automatically). The figure second from the left shows a PCA projection of the embedded point cloud. From the embedded point cloud, we produce a decorated merge tree (second from right) after light preprocessing via density-based subsampling in order to remove outliers (see text). Using the method described in the text, we are able to locate samples in the original time series which correspond to a point in the connected component where any particular degree-1 bar is born (shown on the right, color-coded to the bars on the DMT). The top row shows a synthetic time series and the bottom row shows a real world time series (heartrate data from \cite{reiss2019deep}). The heartrate data shows component-specific periodic behavior at the beginning and at the end of the time series, with the middle part apparently filtered out as a `transitionary phase' when subsampling.}
    \label{fig:takens_embedding}
\end{figure}

\subsection{Clustering Point Clouds}\label{sec:clustering_point_clouds}

In this example, we demonstrate the ability of DMT matching distance to distinguish point clouds with subtle topological differences. Figure \ref{fig:pointcloudsA} shows samples from six classes of synthetically-generated point clouds consisting of noisy blobs and circles. Each class contains 3 examples, with examples within class differing only due to noise effects. All classes have similar degree-$0$ topological structure, classes in the top row have three main degree-$1$ features and classes in the bottom row have two main degree-$1$ features. The distribution of the cycles among the connected components differ from class-to-class, but are the same within class.

For each point cloud, we construct various topological descriptors: a degree-$0$ persistence diagram, a degree-$1$ persistence diagram and a decorated merge tree. Pairwise distance matrices are computed for all 18 samples with respect to various metrics: bottleneck distance on degree-$0$ features, bottleneck distance on degree-$1$ features, maximum of bottleneck distance between degree-$0$ and degree-$1$ features and an estimation of decorated merge tree interleaving distance computed via the Fused Gromov-Wasserstein algorithm of Section \ref{sec:fused_gromov_wasserstein}. The resulting matrices and corresponding MDS plots are shown in Figure \ref{fig:pointcloudsB}.

\begin{figure}[h!]
    \centering
     \begin{subfigure}{.49\textwidth}
     \vspace{.08in}
        \includegraphics[width=.99\linewidth]{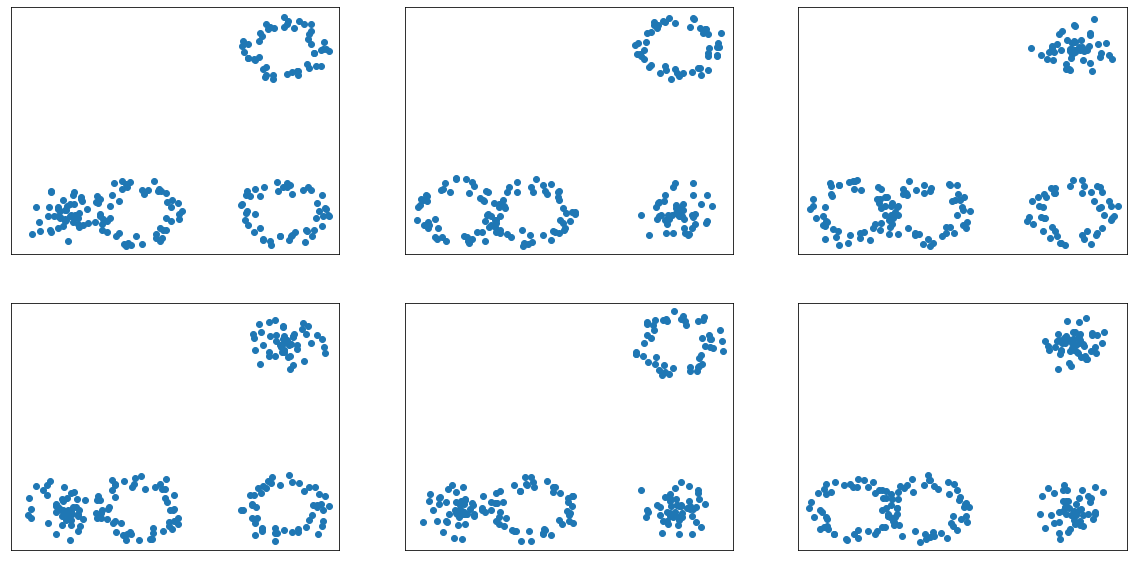}
      \caption{}
      \label{fig:pointcloudsA}
    \end{subfigure}%
    \centering
     \begin{subfigure}{.49\textwidth}
        \includegraphics[width=.99\linewidth]{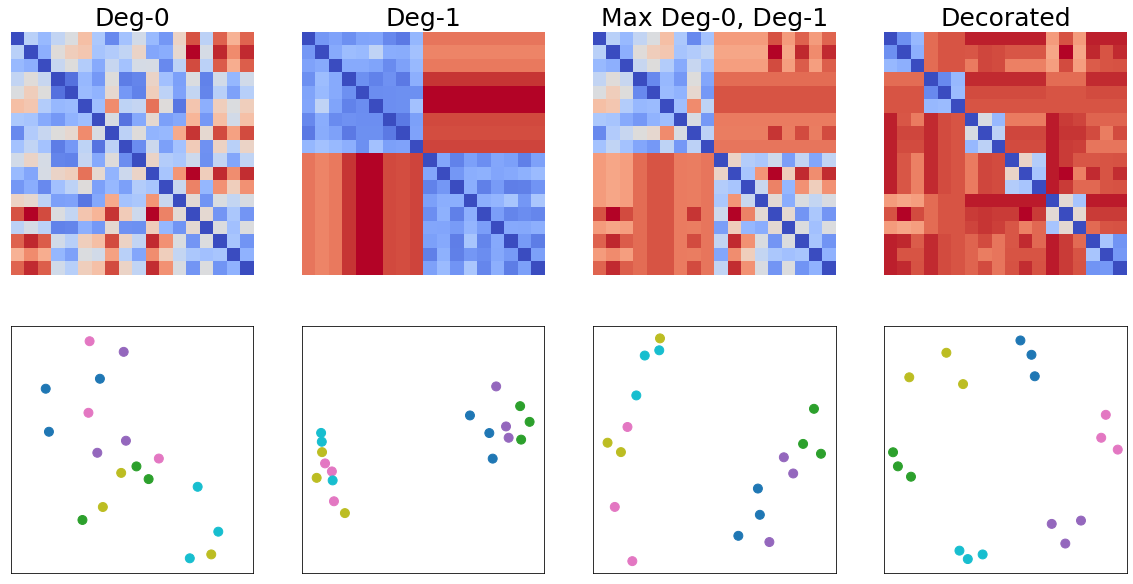}
      \caption{}
      \label{fig:pointcloudsB}
    \end{subfigure}%
    \caption{ {\bf (a)} Samples from six classes of simple pointclouds in $\R^2$. Each class contains 3 examples of pointclouds with similar topological patterns.  {\bf (b)} Pairwise distance matrices for various methods shown in the top row, with corresponding MDS plots shown in the bottom row.}
     \label{fig:point_cloud_classification}
 \end{figure}

\subsection{Networks}\label{sec:networks}

Beyond Vietoris-Rips complexes of point clouds, we can generate DMTs for more general filtered simplicial complexes. A class of such objects which is important from a data science perspective is the class of filtered networks. To make this precise, a \define{filtered network} is a graph $G = (V,E)$ together with a function $f:V \to \R$. This function is extended to each edge $\{v,w\} \in E$ by the rule
\[
f(\{v,w\}) := \max\{f(v),f(w)\},
\]
and this yields a sublevel-set filtration in the sense of Example \ref{ex:sublevel_set_simplicial}. We found it useful to add 2-dimensional simplices to the graphs by computing the 2-skeleton of the Vietoris-Rips complex with respect to shortest path distance; this has the effect of allowing degree-1 bars to die at finite times which reflect the size of their representative cycles. We are also able to threshold out small bars corresponding to triangles in the networks when displaying DMTs. This method is admittedly ad hoc, but it  results in more informative topological summaries which capture both the topology of the filtration and the distance structure of the graph.

Given a filtered network, we produce a DMT by first building the merge tree itself by iteratively adding nodes to the tree in order of function value. Degree-1 homological information can be added to the merge tree by lifting bars (guaranteed to produce the correct DMT if the disjointness condition (Definition \ref{def:disjointness}) is satisfied) or by a more involved construction, as in the case of point clouds.

An important source of network data comes from images represented as node-weighted graphs. Figure \ref{fig:imageNetworkPipeline} shows a decorated merge tree describing the topological features of real fMRI brain scan image data. This is accomplished by considering an image as a weighted grid graph; that is, each pixel of the image is a node, nodes are connected in a grid to their neighbors and weights are assigned according to grayscale value. This procedure is also summarized in the figure.

\begin{figure}
    \centering
    \includegraphics[width = 0.8\textwidth]{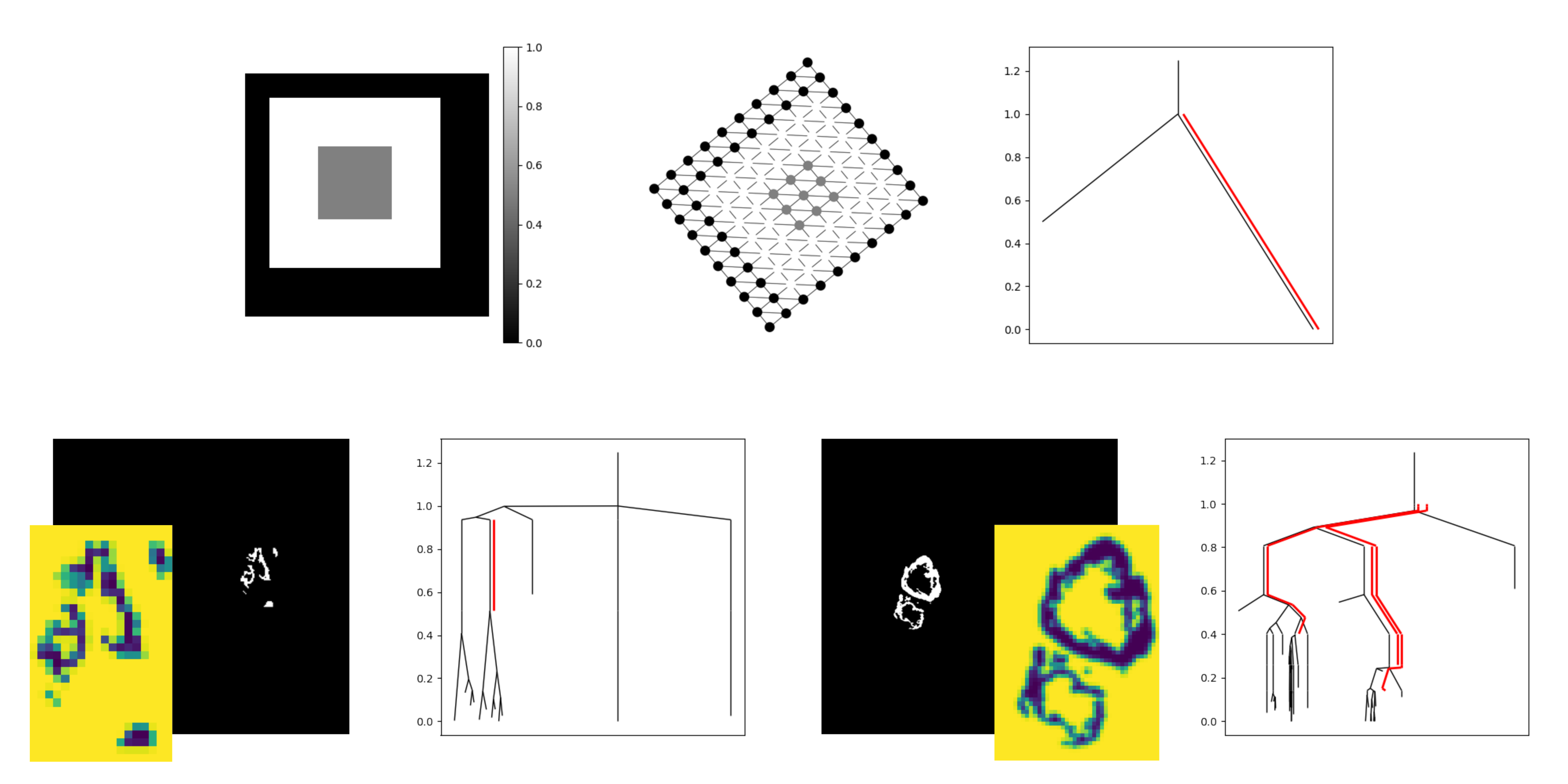}
    \caption{Grayscale images as a filtered networks. The {\bf top row} illustrates the pipeline for extracting a decorated merge tee from a grayscale image. Beginning with a toy image on the left, we convert it to a filtered network (a regular graph with nodes weighted by grayscale value). From this, we produce a decorated merge tree (right). The {\bf bottom row} shows two examples of decorated merge trees extracted from  MRI images of Glioblastoma Multiforme tumors from The Cancer Imaging Archive \cite{clark2013cancer,scarpace2016radiology}, with the segmentations coming from \cite{crawford2020predicting}. Each example shows the original image, with an inset showing a more detailed and lightly preprocessed version of the region of interest.}
    \label{fig:imageNetworkPipeline}
\end{figure}

\subsection{Network Matching}\label{sec:network_matching}

Computation of the decorated merge tree matching distance provides rich information about correspondences between points in the merge trees, since it involves the computation of an optimal coupling between these points. This is especially informative in the setting of filtered networks, since each node in the merge tree produced via our algorithm corresponds to a node in the original network. The optimal coupling then provides a probabilistic matching between nodes of the filtered networks which captures the topology of the filtration function. 

Estimation of node correspondences between graphs is a classical problem which has recently been a focus of Gromov-Wasserstein-based methods. The optimal coupling used to match two graphs depends on the how they are represented as measure networks in the GW matching problem \eqref{eqn:gw_p_distance}; for example a graph can be represented by its geodesic distance matrix \cite{hendrikson2016using}, its adjacency matrix \cite{xu2019scalable} or by a heat kernel matrix at a chosen scale \cite{chowdhury2020generalized}. The coupling produced in the process of computing the decorated merge tree matching distance differs from existing approaches in that it matches based on purely topological features of the graphs. A simple example illustrating the qualitative differences in graph matchings is shown in Figure \ref{fig:matching_networks}. In this example, the node filtration for each graph is given by a diffusion Fr\'{e}chet function, which measures local node density \cite{martinez2019probing}. Learning the best filtration function for a given task in the context of classical TDA is an active field of research \cite{hajij2020graph,hofer2020graph}; extending this line of research to learn filtration functions which capture the interplay between degree-0 and degree-1 features will be taken up in future work.

\begin{figure}
    \centering
    \includegraphics[width = 0.8\textwidth]{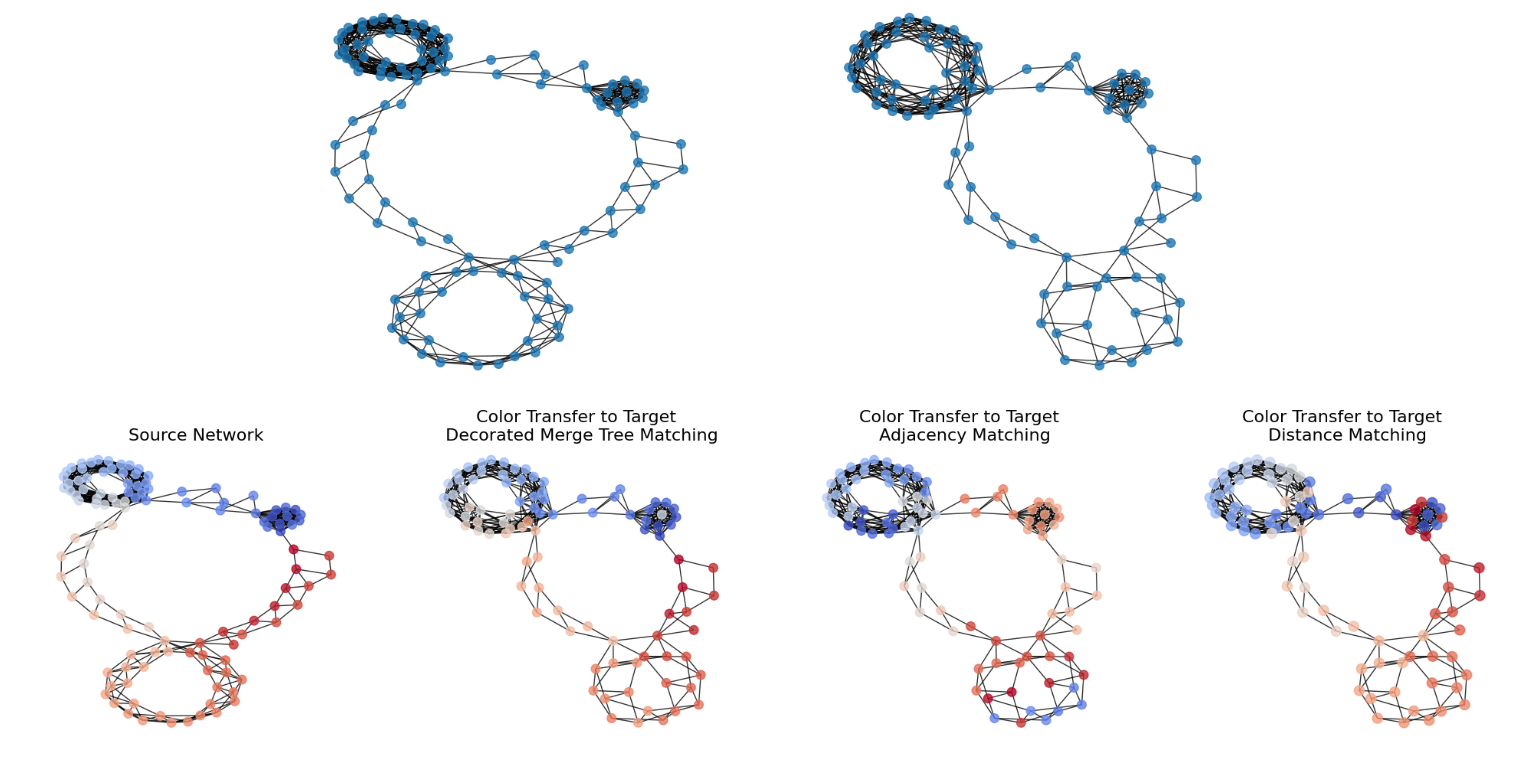}
    \caption{Matching networks. The top row shows two networks, constructed to have similar topological motifs. The bottom row illustrates node matchings between the graph; each matching is illustrated by transferring the node colors from the source graph (the node colors here are arbitrary) to the target graph via a probabilistic coupling of its nodes---node $j$ in the target graph receives the color of  target node $i$ when the $(i,j)$ entry is the largest in the $j$th column of the coupling matrix. The first color transfer utilizes the coupling obtained during the computation of DMT matching distance via the Fused GW framework, while the latter two couplings come from the GW framework applied to adjacency matrices and geodesic distance matrices, respectively. Qualitatively, the matching produced using DMTs most strongly preserves topological features.}
    \label{fig:matching_networks}
\end{figure}

\section{Discussion}\label{sec:discussion}

In this paper, we defined several variants of the notion of a decorated merge tree. We introduced a metric on the space of DMTs as well as a metric on the simpler space of leaf decorated merge trees which is more amenable to computation. A stability result was demonstrated for these metrics, and stability was extended to compare these metrics to several other metrics which have appeared in the literature. Several use cases for the DMT framework were demonstrated via computational examples. 

There are many directions for future research on DMTs and related ideas. 
On the computational side, we plan to refine our algorithms to efficiently handle DMT computations on real  datasets. 
This challenge comes with some interesting theoretical questions. For example, we are currently using off-the-shelf methods for approximating Gromov-Wasserstein matchings between merge trees; the hierarchical structure of merge trees should allow for specialized algorithms which utilize this structure to produce faster and more accurate estimates. 
To scale DMT-based analyses to handle large datasets in a modern machine learning setting, we also plan to explore principled vectorizations of these signatures. 

There are also several directions for research on the theoretical side. 
Two natural ways to generalize the concept of a decorated merge tree are to vary the object being decorated (the merge tree) and to vary the type of decoration (persistence barcodes). 
In future work, we plan to extend these ideas to treat decorated Reeb graphs. 
We also plan to explore decorations by zig-zag persistence modules and extended persistence modules.

Furthermore, Remark~\ref{rmk:universality} suggests that perhaps the interleaving distance on categorical decorated merge trees is universal~\cite{lesnick2015theory}.
One avenue for studying this result could leverage an observation of Gunnar Carlsson: that the cup product structure on cohomology determines $\pi_0$ information. Perhaps a universality result that preserves the richer algebraic structure of cup productsis the right framework for framing this result.
Finally, one might take up the question of how to define ``higher'' decorated merge trees---via cochains, for example~\cite{Mandell-Cochains}---that might offer a complete homotopical invariant of persistent spaces.
Repeating much of the standard template of results in algebraic topology from the past hundred years for persistent spaces is an active and interesting line of research.

\printbibliography

@article{carlsson2010zigzag,
  title={Zigzag persistence},
  author={Carlsson, Gunnar and De Silva, Vin},
  journal={Foundations of computational mathematics},
  volume={10},
  number={4},
  pages={367--405},
  year={2010},
  publisher={Springer}
}

@article{carlsson2014topological,
  title={Topological pattern recognition for point cloud data},
  author={Carlsson, Gunnar},
  journal={Acta Numerica},
  volume={23},
  pages={289--368},
  year={2014},
  publisher={Cambridge University Press}
}

@inproceedings{dey2015comparing,
  title={Comparing Graphs via Persistence Distortion},
  author={Dey, Tamal K and Shi, Dayu and Wang, Yusu},
  booktitle={31st International Symposium on Computational Geometry (SoCG 2015)},
  year={2015},
  organization={Schloss Dagstuhl-Leibniz-Zentrum fuer Informatik}
}

@article{oudot2017barcode,
  title={Barcode embeddings for metric graphs},
  author={Oudot, Steve and Solomon, Elchanan},
  journal={arXiv preprint arXiv:1712.03630},
  year={2017}
}

@article{lesnick2015interactive,
  title={Interactive visualization of 2-D persistence modules},
  author={Lesnick, Michael and Wright, Matthew},
  journal={arXiv preprint arXiv:1512.00180},
  year={2015}
}

@article{curry2018many,
  title={How many directions determine a shape and other sufficiency results for two topological transforms},
  author={Curry, Justin and Mukherjee, Sayan and Turner, Katharine},
  journal={arXiv preprint arXiv:1805.09782},
  year={2018}
}

@article{turner2014persistent,
  title={Persistent homology transform for modeling shapes and surfaces},
  author={Turner, Katharine and Mukherjee, Sayan and Boyer, Doug M},
  journal={Information and Inference: A Journal of the IMA},
  volume={3},
  number={4},
  pages={310--344},
  year={2014},
  publisher={Oxford University Press}
}

@article{memoli2019gromov,
  title={Gromov-Hausdorff distances on $ p $-metric spaces and ultrametric spaces},
  author={M{\'e}moli, Facundo and Smith, Zane and Wan, Zhengchao},
  journal={arXiv preprint arXiv:1912.00564},
  year={2019}
}

@inproceedings{smith2016hierarchical,
  title={Hierarchical representations of network data with optimal distortion bounds},
  author={Smith, Zane and Chowdhury, Samir and M{\'e}moli, Facundo},
  booktitle={2016 50th Asilomar Conference on Signals, Systems and Computers},
  pages={1834--1838},
  year={2016},
  organization={IEEE}
}

@article{hang2019topological,
  title={A topological study of functional data and Fr{\'e}chet functions of metric measure spaces},
  author={Hang, Haibin and M{\'e}moli, Facundo and Mio, Washington},
  journal={Journal of Applied and Computational Topology},
  volume={3},
  number={4},
  pages={359--380},
  year={2019},
  publisher={Springer}
}

@article{frosini2019persistent,
  title={The persistent homotopy type distance},
  author={Frosini, Patrizio and Landi, Claudia and M{\'e}moli, Facundo},
  journal={Homology, Homotopy and Applications},
  volume={21},
  number={2},
  pages={231--259},
  year={2019},
  publisher={International Press of Boston}
}

@article{bainbridge2021,
    title={Directional Cellular Sheaves for Multicast Network Routing},
    author={Bainbridge, Gabriel},
    journal={In Preparation}
}

@article{Mandell-Cochains,
     author = {Mandell, Michael A.},
     title = {Cochains and homotopy type},
     journal = {Publications Math\'ematiques de l'IH\'ES},
     pages = {213--246},
     publisher = {Springer},
     volume = {103},
     year = {2006},
     doi = {10.1007/s10240-006-0037-6},
     zbl = {1105.55003},
     mrnumber = {2233853},
     language = {en},
     url = {http://www.numdam.org/item/PMIHES_2006__103__213_0/}
}

@incollection{landi2018rank,
  title={The rank invariant stability via interleavings},
  author={Landi, Claudia},
  booktitle={Research in computational topology},
  pages={1--10},
  year={2018},
  publisher={Springer}
}

@incollection{oudot2020inverse,
  title={Inverse problems in topological persistence},
  author={Oudot, Steve and Solomon, Elchanan},
  booktitle={Topological Data Analysis},
  pages={405--433},
  year={2020},
  publisher={Springer}
}

@article{lesnick2015theory,
  title={The theory of the interleaving distance on multidimensional persistence modules},
  author={Lesnick, Michael},
  journal={Foundations of Computational Mathematics},
  volume={15},
  number={3},
  pages={613--650},
  year={2015},
  publisher={Springer}
}

@book{burago2001course,
  title={A course in metric geometry},
  author={Burago, Dmitri and Burago, Iu D and Burago, Yuri and Ivanov, Sergei and Ivanov, Sergei V and Ivanov, Sergei A},
  volume={33},
  year={2001},
  publisher={American Mathematical Soc.}
}

@book{villani2008optimal,
  title={Optimal transport: old and new},
  author={Villani, C{\'e}dric},
  volume={338},
  year={2008},
  publisher={Springer Science \& Business Media}
}

@article{chowdhury2020generalized,
  title={Generalized Spectral Clustering via Gromov-Wasserstein Learning},
  author={Chowdhury, Samir and Needham, Tom},
  journal={arXiv preprint arXiv:2006.04163},
  year={2020}
}

@article{patel2018generalized,
  title={Generalized persistence diagrams},
  author={Patel, Amit},
  journal={Journal of Applied and Computational Topology},
  volume={1},
  number={3},
  pages={397--419},
  year={2018},
  publisher={Springer}
}

@article{kim2018generalized,
  title={Generalized persistence diagrams for persistence modules over posets},
  author={Kim, Woojin and Memoli, Facundo},
  journal={arXiv preprint arXiv:1810.11517},
  year={2018}
}

@article{hendrikson2016using,
  title={Using Gromov-Wasserstein distance to explore sets of networks},
  author={Hendrikson, Reigo and others},
  journal={University of Tartu, Master Thesis},
  volume={2},
  year={2016}
}

@article{hajij2020graph,
  title={Graph Similarity Using PageRank and Persistent Homology},
  author={Hajij, Mustafa and Munch, Elizabeth and Rosen, Paul},
  journal={arXiv preprint arXiv:2002.05158},
  year={2020}
}

@inproceedings{hofer2020graph,
  title={Graph filtration learning},
  author={Hofer, Christoph and Graf, Florian and Rieck, Bastian and Niethammer, Marc and Kwitt, Roland},
  booktitle={International Conference on Machine Learning},
  pages={4314--4323},
  year={2020},
  organization={PMLR}
}

@article{martinez2019probing,
  title={Probing the geometry of data with diffusion Fr{\'e}chet functions},
  author={Mart{\'\i}nez, Diego H D{\'\i}az and Lee, Christine H and Kim, Peter T and Mio, Washington},
  journal={Applied and Computational Harmonic Analysis},
  volume={47},
  number={3},
  pages={935--947},
  year={2019},
  publisher={Elsevier}
}

@inproceedings{venkataraman2016persistent,
  title={Persistent homology of attractors for action recognition},
  author={Venkataraman, Vinay and Ramamurthy, Karthikeyan Natesan and Turaga, Pavan},
  booktitle={2016 IEEE international conference on image processing (ICIP)},
  pages={4150--4154},
  year={2016},
  organization={IEEE}
}

@article{emrani2014persistent,
  title={Persistent homology of delay embeddings and its application to wheeze detection},
  author={Emrani, Saba and Gentimis, Thanos and Krim, Hamid},
  journal={IEEE Signal Processing Letters},
  volume={21},
  number={4},
  pages={459--463},
  year={2014},
  publisher={IEEE}
}

@incollection{takens1981detecting,
  title={Detecting strange attractors in turbulence},
  author={Takens, Floris},
  booktitle={Dynamical systems and turbulence, Warwick 1980},
  pages={366--381},
  year={1981},
  publisher={Springer}
}

@article{perea2015sliding,
  title={Sliding windows and persistence: An application of topological methods to signal analysis},
  author={Perea, Jose A and Harer, John},
  journal={Foundations of Computational Mathematics},
  volume={15},
  number={3},
  pages={799--838},
  year={2015},
  publisher={Springer}
}

@ARTICLE{2020SciPy-NMeth,
  author  = {Virtanen, Pauli and Gommers, Ralf and Oliphant, Travis E. and
            Haberland, Matt and Reddy, Tyler and Cournapeau, David and
            Burovski, Evgeni and Peterson, Pearu and Weckesser, Warren and
            Bright, Jonathan and {van der Walt}, St{\'e}fan J. and
            Brett, Matthew and Wilson, Joshua and Millman, K. Jarrod and
            Mayorov, Nikolay and Nelson, Andrew R. J. and Jones, Eric and
            Kern, Robert and Larson, Eric and Carey, C J and
            Polat, {\.I}lhan and Feng, Yu and Moore, Eric W. and
            {VanderPlas}, Jake and Laxalde, Denis and Perktold, Josef and
            Cimrman, Robert and Henriksen, Ian and Quintero, E. A. and
            Harris, Charles R. and Archibald, Anne M. and
            Ribeiro, Ant{\^o}nio H. and Pedregosa, Fabian and
            {van Mulbregt}, Paul and {SciPy 1.0 Contributors}},
  title   = {{{SciPy} 1.0: Fundamental Algorithms for Scientific
            Computing in Python}},
  journal = {Nature Methods},
  year    = {2020},
  volume  = {17},
  pages   = {261--272},
  adsurl  = {https://rdcu.be/b08Wh},
  doi     = {10.1038/s41592-019-0686-2},
}

@article{scikit-learn,
 title={Scikit-learn: Machine Learning in {P}ython},
 author={Pedregosa, F. and Varoquaux, G. and Gramfort, A. and Michel, V.
         and Thirion, B. and Grisel, O. and Blondel, M. and Prettenhofer, P.
         and Weiss, R. and Dubourg, V. and Vanderplas, J. and Passos, A. and
         Cournapeau, D. and Brucher, M. and Perrot, M. and Duchesnay, E.},
 journal={Journal of Machine Learning Research},
 volume={12},
 pages={2825--2830},
 year={2011}
}

@misc{1908.02518,
	Author = {Ulrich Bauer},
	Title = {Ripser: efficient computation of Vietoris-Rips persistence barcodes},
	Month = Aug,
	Year = {2019},
	Eprint = {1908.02518},
	Note = {Preprint}
}

@misc{flamary2017pot,
title={POT Python Optimal Transport library},
author={Flamary, R{'e}mi and Courty, Nicolas},
url={https://pythonot.github.io/},
year={2017}
}

@book{gudhi:urm
, title        = "{GUDHI} User and Reference Manual"
, author      = "{The GUDHI Project}"
, publisher     = "{GUDHI Editorial Board}"
, edition =     "{3.4.1}"
, year         = 2021
, url =    "https://gudhi.inria.fr/doc/3.4.1/"
}

@misc{tauzin2020giottotda,
      title={giotto-tda: A Topological Data Analysis Toolkit for Machine Learning and Data Exploration},
      author={Guillaume Tauzin and Umberto Lupo and Lewis Tunstall and Julian Burella Pérez and Matteo Caorsi and Anibal Medina-Mardones and Alberto Dassatti and Kathryn Hess},
      year={2020},
      eprint={2004.02551},
      archivePrefix={arXiv},
      primaryClass={cs.LG}
}

@misc{scikittda2019,
    author       = {Nathaniel Saul and Chris Tralie},
    title        = {Scikit-TDA: Topological Data Analysis for Python},
    year         = 2019,
    doi          = {10.5281/zenodo.2533369},
    url          = {https://doi.org/10.5281/zenodo.2533369}
}

@inproceedings{chowdhury2020gromov,
  title={Gromov-wasserstein averaging in a riemannian framework},
  author={Chowdhury, Samir and Needham, Tom},
  booktitle={Proceedings of the IEEE/CVF Conference on Computer Vision and Pattern Recognition Workshops},
  pages={842--843},
  year={2020}
}

@article{li2021sketching,
  title={Sketching Merge Trees},
  author={Li, Mingzhe and Palande, Sourabh and Wang, Bei},
  journal={arXiv e-prints},
  pages={arXiv--2101},
  year={2021}
}

@article{memoli2021ultrametric,
  title={The ultrametric Gromov-Wasserstein distance},
  author={M{\'e}moli, Facundo and Munk, Axel and Wan, Zhengchao and Weitkamp, Christoph},
  journal={arXiv preprint arXiv:2101.05756},
  year={2021}
}

@book{riehl2017category,
  title={Category theory in context},
  author={Riehl, Emily},
  year={2017},
  publisher={Courier Dover Publications}
}

@article{carlsson2013classifying,
  title={Classifying clustering schemes},
  author={Carlsson, Gunnar and M{\'e}moli, Facundo},
  journal={Foundations of Computational Mathematics},
  volume={13},
  number={2},
  pages={221--252},
  year={2013},
  publisher={Springer}
}

@inproceedings{elz2000,
title={Topological persistence and simplification},
  author={Edelsbrunner, Herbert and Letscher, David and Zomorodian, Afra},
  booktitle={Proceedings 41st annual symposium on foundations of computer science},
  pages={454--463},
  year={2000},
  organization={IEEE}
}

@article{crawford2020predicting,
  title={Predicting clinical outcomes in glioblastoma: an application of topological and functional data analysis},
  author={Crawford, Lorin and Monod, Anthea and Chen, Andrew X and Mukherjee, Sayan and Rabad{\'a}n, Ra{\'u}l},
  journal={Journal of the American Statistical Association},
  volume={115},
  number={531},
  pages={1139--1150},
  year={2020},
  publisher={Taylor \& Francis}
}

@article{scarpace2016radiology,
  title={Radiology data from the cancer genome atlas glioblastoma multiforme [TCGA-GBM] collection},
  author={Scarpace, Lisa and Mikkelsen, L and Cha, T and Rao, S and Tekchandani, S and Gutman, S and Pierce, D},
  journal={The Cancer Imaging Archive},
  volume={11},
  number={4},
  pages={1},
  year={2016}
}

@article{clark2013cancer,
  title={The Cancer Imaging Archive (TCIA): maintaining and operating a public information repository},
  author={Clark, Kenneth and Vendt, Bruce and Smith, Kirk and Freymann, John and Kirby, Justin and Koppel, Paul and Moore, Stephen and Phillips, Stanley and Maffitt, David and Pringle, Michael and others},
  journal={Journal of digital imaging},
  volume={26},
  number={6},
  pages={1045--1057},
  year={2013},
  publisher={Springer}
}

@article{reiss2019deep,
  title={Deep PPG: large-scale heart rate estimation with convolutional neural networks},
  author={Reiss, Attila and Indlekofer, Ina and Schmidt, Philip and Van Laerhoven, Kristof},
  journal={Sensors},
  volume={19},
  number={14},
  pages={3079},
  year={2019},
  publisher={Multidisciplinary Digital Publishing Institute}
}

@article{DBLP:journals/corr/abs-1808-04337,
  author    = {Samir Chowdhury and
               Facundo M{\'{e}}moli},
  title     = {The {G}romov-{W}asserstein distance between networks and stable network
               invariants},
  journal   = {CoRR},
  volume    = {abs/1808.04337},
  year      = {2018},
  url       = {http://arxiv.org/abs/1808.04337},
  archivePrefix = {arXiv},
  eprint    = {1808.04337},
  bibsource = {dblp computer science bibliography, https://dblp.org}
}

@article{xu2020learning,
  title={Learning autoencoders with relational regularization},
  author={Xu, Hongteng and Luo, Dixin and Henao, Ricardo and Shah, Svati and Carin, Lawrence},
  journal={arXiv preprint arXiv:2002.02913},
  year={2020}
}

@article{chapel2020partial,
  title={Partial Gromov-Wasserstein with Applications on Positive-Unlabeled Learning},
  author={Chapel, Laetitia and Alaya, Mokhtar Z and Gasso, Gilles},
  journal={arXiv preprint arXiv:2002.08276},
  year={2020}
}

@inproceedings{xu2019scalable,
  title={Scalable Gromov-Wasserstein learning for graph partitioning and matching},
  author={Xu, Hongteng and Luo, Dixin and Carin, Lawrence},
  booktitle={Advances in neural information processing systems},
  pages={3052--3062},
  year={2019}
}

@inproceedings{peyre2016gromov,
  title={Gromov-wasserstein averaging of kernel and distance matrices},
  author={Peyr{\'e}, Gabriel and Cuturi, Marco and Solomon, Justin},
  booktitle={International Conference on Machine Learning},
  pages={2664--2672},
  year={2016}
}

@inproceedings{alvarez2019towards,
  title={Towards optimal transport with global invariances},
  author={Alvarez-Melis, David and Jegelka, Stefanie and Jaakkola, Tommi S},
  booktitle={The 22nd International Conference on Artificial Intelligence and Statistics},
  pages={1870--1879},
  year={2019},
  organization={PMLR}
}

@article{dgh-sm,
  title={On the use of {G}romov-{H}ausdorff distances for shape comparison},
  author={M{\'e}moli, Facundo},
  year={2007},
  journal={The Eurographics Association}
}

@article{sturm2006geometry,
  title={On the geometry of metric measure spaces},
  author={Sturm, Karl-Theodor},
  journal={Acta mathematica},
  volume={196},
  number={1},
  pages={65--131},
  year={2006},
  publisher={Springer}
}

@article{sturm2012space,
  title={The space of spaces: curvature bounds and gradient flows on the space of metric measure spaces},
  author={Sturm, Karl-Theodor},
  journal={arXiv preprint arXiv:1208.0434},
  year={2012}
}

@article{vayer2020fused,
  title={Fused Gromov-Wasserstein distance for structured objects},
  author={Vayer, Titouan and Chapel, Laetitia and Flamary, R{\'e}mi and Tavenard, Romain and Courty, Nicolas},
  journal={Algorithms},
  volume={13},
  number={9},
  pages={212},
  year={2020},
  publisher={Multidisciplinary Digital Publishing Institute}
}

@article{memoli2011gromov,
  title={Gromov--Wasserstein distances and the metric approach to object matching},
  author={M{\'e}moli, Facundo},
  journal={Foundations of computational mathematics},
  volume={11},
  number={4},
  pages={417--487},
  year={2011},
  publisher={Springer}
}

@incollection{bauer2020persistence,
  title={Persistence diagrams as diagrams: A categorification of the stability theorem},
  author={Bauer, Ulrich and Lesnick, Michael},
  booktitle={Topological Data Analysis},
  pages={67--96},
  year={2020},
  publisher={Springer}
}

@article{crawley2015decomposition,
  title={Decomposition of pointwise finite-dimensional persistence modules},
  author={Crawley-Boevey, William},
  journal={Journal of Algebra and its Applications},
  volume={14},
  number={05},
  pages={1550066},
  year={2015},
  publisher={World Scientific}
}

@phdthesis{stefanou2018dynamics,
  title={Dynamics on Categories and Applications},
  author={Stefanou, Anastasios},
  year={2018},
  publisher={State University of New York at Albany}
}

@article{de2018theory,
  title={THEORY OF INTERLEAVINGS ON CATEGORIES WITH A FLOW},
  author={DE SILVA, V and MUNCH, E and STEFANOU, A},
  journal={Theory and Applications of Categories},
  volume={33},
  number={21},
  pages={583--607},
  year={2018}
}

@article{bubenik2015metrics,
  title={Metrics for generalized persistence modules},
  author={Bubenik, Peter and De Silva, Vin and Scott, Jonathan},
  journal={Foundations of Computational Mathematics},
  volume={15},
  number={6},
  pages={1501--1531},
  year={2015},
  publisher={Springer}
}

@article{bubenik2014categorification,
  title={Categorification of persistent homology},
  author={Bubenik, Peter and Scott, Jonathan A},
  journal={Discrete \& Computational Geometry},
  volume={51},
  number={3},
  pages={600--627},
  year={2014},
  publisher={Springer}
}

@article{curry2018fiber,
  title={The fiber of the persistence map for functions on the interval},
  author={Curry, Justin},
  journal={Journal of Applied and Computational Topology},
  volume={2},
  number={3-4},
  pages={301--321},
  year={2018},
  publisher={Springer}
}

@article{curry2015topological,
  title={Topological data analysis and cosheaves},
  author={Curry, Justin Michael},
  journal={Japan Journal of Industrial and Applied Mathematics},
  volume={32},
  number={2},
  pages={333--371},
  year={2015},
  publisher={Springer}
}

@article{yan2019structural,
  title={A Structural Average of Labeled Merge Trees for Uncertainty Visualization},
  author={Yan, Lin and Wang, Yusu and Munch, Elizabeth and Gasparovic, Ellen and Wang, Bei},
  journal={IEEE transactions on visualization and computer graphics},
  volume={26},
  number={1},
  pages={832--842},
  year={2019},
  publisher={IEEE}
}

@article{agarwal2018computing,
  title={Computing the Gromov-Hausdorff distance for metric trees},
  author={Agarwal, Pankaj K and Fox, Kyle and Nath, Abhinandan and Sidiropoulos, Anastasios and Wang, Yusu},
  journal={ACM Transactions on Algorithms (TALG)},
  volume={14},
  number={2},
  pages={1--20},
  year={2018},
  publisher={ACM New York, NY, USA}
}

@article{gasparovic2019intrinsic,
  title={Intrinsic Interleaving Distance for Merge Trees},
  author={Gasparovic, Ellen and Munch, Elizabeth and Oudot, Steve and Turner, Katharine and Wang, Bei and Wang, Yusu},
  journal={arXiv preprint arXiv:1908.00063},
  year={2019}
}

@inproceedings{farahbakhsh2019fpt,
  title={{FPT}-Algorithms for Computing Gromov-Hausdorff and Interleaving Distances Between Trees},
  author={Farahbakhsh Touli, Elena and Wang, Yusu},
  booktitle={27th Annual European Symposium on Algorithms (ESA 2019)},
  year={2019},
  organization={Schloss Dagstuhl-Leibniz-Zentrum fuer Informatik}
}

@article{botnan2020decomposition,
  title={Decomposition of persistence modules},
  author={Botnan, Magnus and Crawley-Boevey, William},
  journal={Proceedings of the American Mathematical Society},
  volume={148},
  number={11},
  pages={4581--4596},
  year={2020}
}

@article{mbw13,
  title={Interleaving distance between merge trees},
  author={Morozov, Dmitriy and Beketayev, Kenes and Weber, Gunther},
  journal={Discrete and Computational Geometry},
  volume={49},
  pages={22--45},
  year={2013}
}

@inproceedings{chazal2009proximity,
  title={Proximity of persistence modules and their diagrams},
  author={Chazal, Fr{\'e}d{\'e}ric and Cohen-Steiner, David and Glisse, Marc and Guibas, Leonidas J and Oudot, Steve Y},
  booktitle={Proceedings of the twenty-fifth annual symposium on Computational geometry},
  pages={237--246},
  year={2009}
}

@article{bl17,
  title={Universality of the homotopy interleaving distance},
  author={Blumberg, Andrew J and Lesnick, Michael},
  journal={arXiv preprint arXiv:1705.01690},
  year={2017}
}

@article{vds,
  title={Categorified reeb graphs},
  author={De Silva, Vin and Munch, Elizabeth and Patel, Amit},
  journal={Discrete \& Computational Geometry},
  volume={55},
  number={4},
  pages={854--906},
  year={2016},
  publisher={Springer}
}

@article{memoli2018metric,
  title={Metric graph approximations of geodesic spaces},
  author={M{\'e}moli, Facundo and Okutan, Osman Berat},
  journal={arXiv preprint arXiv:1809.05566},
  year={2018}
}

@incollection{bowditch1991notes,
  title={Notes on Gromov's hyperbolicity criterion for path-metric spaces},
  author={Bowditch, Brian H},
  booktitle={Group theory from a geometrical viewpoint},
  year={1991}
}

@incollection{gromov1987hyperbolic,
  title={Hyperbolic groups},
  author={Gromov, Mikhael},
  booktitle={Essays in group theory},
  pages={75--263},
  year={1987},
  publisher={Springer}
}

@article{hang2020correspondence,
      title={Correspondence Modules and Persistence Sheaves: A Unifying Framework for One-Parameter Persistent Homology},
      author={Haibin Hang and Washington Mio},
      journal={arXiv preprint arXiv:2006.08557},
      year={2020},
}

\appendix

\section{Interval Topology}\label{sec:interval-topology}
Let $X$ be a compact topological space and $f: X \to \R$ be a continuous function. Let $E_f$ be the epigraph of $f: X \to \R$, $\phi: E_f \to \R$ be the projection map, and  $\pi: E_f \to T_f X$ be the Reeb quotient map, which identifies two points if they are in the same connected component of a level set.

Define a partial order on $T_f X$ by letting $p \leq q$ if there exists $x$ in $X$ and $r,s$ in $\R$ such that $f(x) \leq r \leq s$ and $p=\pi(x,r)$ and $q=\pi(x,s)$. The fact that this defines a poset follows from the following lemma:

\begin{lemma}\label{lem:poset_characterization}
Let $(x,r)$ and $(y,s)$ be points in $E_f$ such that $r \leq s$. The following are equivalent:
\begin{enumerate}[i)]
    \item $\pi(x,r) \leq \pi(y,s)$.
    \item $\pi(x,s) = \pi(y,s)$.
    \item $\pi(x,t) = \pi(y,t)$ for all $t > s$.
\end{enumerate}
\end{lemma}
\begin{proof}
$i) \implies ii):$ Let $z$ be a point in $X$ such that $\pi(z,r)=\pi(x,r)$ and $\pi(z,s) = \pi(y,s)$.  Let $C$ be the connected component of $\phi^{-1}(r)$ containing $(x,r)$ and $D$ be the connected component in $\phi^{-1}(s)$ containing $(y,s)$. Let $C'$ be the subset of $E_f$ obtained by shifting up the second coordinates of points in $C$ to $s$. As $C$ contains $(z,r)$ and $D$ contains $(z,s)$, both $C'$ and $D$ contains $(z,s)$. Since $C'$ is connected, it is contained in $D$. Hence $\pi(x,s) = \pi(y,s)$.

$ii) \implies iii)$ Note that $\pi(x,s) = \pi(y,s) \leq \pi(y,t)$. Now, the result follows from the argument in the previous part.

$iii) \implies i)$ Let $(s_n)$ be a decreasing sequence converging to $s$. Let $C_n \times \{ s_n\}$ be the connected component of $\phi^{-1}(s_n)$ containing both $(x,s_n)$ and $(y,s_n)$. Note that $(C_n)$ is a decreasing sequence of closed connected sets in $X$. Let $C$ be the intersection of $C_n's$. Note that $C \times \{s\}$ is contained in $E_f$ and $x,y$ are elements of $C$. Let us show that $C$ is connected. Let $V$ and $W$ be open subsets in $X$ such that $C$ is contained in the disjoint union of $U$ of $V$ and $W$. Assume $x$ is in $V$. Let $U_n$ be the complement of $C_n$ in $X$. The collection consisting of $U$ and all $U_n$'s is an open cover of $X$, hence it has a finite subcover. Since $(U_n)$ is an increasing family of open sets, this implies that for $n$ large enough $X$ is the union of $U$ and $U_n$, hence $C_n$ is contained in $U$. Since $C_n$ is connected and has non-empty intersection with $V$ for all $n$, for $n$ large enough $C_n$ is contained in $V$. Therefore $C$ is contained in $V$. This shows $C$ is connected, hence $\pi(x,s)=\pi(y,s)$. By definition, $\pi(x,r) \leq \pi(y,s).$
\end{proof}

\begin{corollary}\label{cor:tree_poset}
\begin{enumerate}[i)]
    \item $(T_f X, \leq)$ is a tree poset in the sense that the upper set of any point is a chain (i.e., well ordered).
    \item Let $C$ be a chain in $(T_f X, \leq)$. Then $C$ has an infimum (i.e., the maximum of its lower bounds), and the closure of $C$ in the quotient topology is contained in the upper set of its infimum.
\end{enumerate}
\end{corollary}
\begin{proof}

$i):$ Assume that  $\pi(x,r) \leq \pi(y,s)$, $\pi(x,r) \leq \pi(z,t)$ and $s \leq t$. Then, by Lemma \ref{lem:poset_characterization}, $\pi(y,t) = \pi(x,t) = \pi(z,t)$, and therefore $\pi(y,s) \leq \pi(z,t)$. 

$ii):$ let $(x_n,r_n)$ be a sequence in $C$ such that $\lim r_n = r : = inf\{s : (x,s) \in C \} $. The sequence $(x_n)$ has a subsequence converging to a point $x$ in $X$. Let $(y_n,s_n)$ be a convergent sequence in $E_f$ such that $\pi(y_n,s_n)$ is in $C$ for all $n$. Let $lim(y_n, s_n) = (y,s)$. We need to show that $\pi(x,r) \leq \pi(y,s)$. Let $t>s$. Without loss of generality, we can assume that $r_n, s_n \leq t$ for all $n$ and $(x_n)$ converges to $x$. Since $C$ is a chain, by Lemma \ref{lem:poset_characterization}, $\pi(y_i,t)=\pi(x_j,t)$ for all $i,j$. Let $D$ be the connected component of $\phi^{-1}(t)$ containing $(y_n,t)$ and $(x_n,t)$ for all $n$. Since $D$ is closed, it contains $(x,t)$ and $(y,t)$, hence $\pi(x,t) = \pi(y,t)$. Since $t>s$ was arbitrary, by Lemma \ref{lem:poset_characterization}, $\pi(x,r) \leq \pi(y,s)$. This shows $\pi(x,r)$ is a lower bound for the closure of $C$. Let $\pi(x',r')$ be another lower bound for $C$. For any $t>r$, there exists $n$ such that $r_n \leq t$, so $\pi(x,t) = \pi(x_n, t) = \pi(x',t)$. Therefore, by Lemma \ref{lem:poset_characterization}, $\pi(x',r') \leq \pi(x,r)$.
\end{proof}

We introduce a new topology on $T_f X$, which we call the \textit{interval topology}, defined as follows. 

\begin{definition}\label{defn:interval-topology}
Let $T_f X$ be the merge tree, viewed as a poset.
A subset $U$ of $T_fX$ is open in the \define{interval topology} if for each $p$ in $U$ and each $q$ in $X \setminus \{ p \}$ comparable to $p$, there exists $w$ in $U$ strictly in between $p$ and $q$ such that the interval between $p$ and $w$ is contained in $U$.
\end{definition}

\begin{proposition}\label{prop:interval_topology}
If $(T_f X, \leq)$ has finitely many leaves (i.e., minimal elements in the poset), then the quotient topology coincides with the interval topology.
\end{proposition}

\begin{lemma}\label{lem:upper_set}
The upper set of any point is closed in the interval topology.
\end{lemma}
\begin{proof}
Let $C$ be the upper set of a point $\pi(z,u)$ in $T_f X$. Let $U$ be its complement. Let us show that $U$ is open in the interval topology. Let $\pi(x,r)$ be in $U$. Let $\pi(y,s)$ be comparable to $\pi(x,r)$. If $\pi(y,s) < \pi(x,r)$, then everything in between them is already in $U$ as $U$ is the complement of an upper set. Let us assume then $\pi(x,r) < \pi(y,s)$. By Lemma \ref{lem:poset_characterization}, there exists $t$ such that $r<t<s$ and $\pi(x,t)$ is not in the upper set of $\pi(z, u)$. Then $\pi(x,t)$ is strictly in between $\pi(x,s)$ and $\pi(y,s)$ and the interval in between $\pi(x,s)$ and $\pi(x,t)$ is in $U$.
\end{proof}

\begin{lemma}\label{lem:height}
The height map $h: T_f X \to \R$, $\pi(x,r) \mapsto r$ is continuous with respect to the interval topology.
\end{lemma}
\begin{proof}
Let $r$ be a real number and $\epsilon>0$. We need to show that the preimage of $(r-\epsilon, r+\epsilon)$ is open in the interval topology. Assume that $\pi(x,s)$ is in $T_f X$ and $|s-r| < \epsilon$. Let $\pi(y,t)$ be a point comparable to $\pi(x,s)$. If $\pi(y,t) < \pi(x,s)$, then $\pi(x,s) = \pi(y,s)$. Pick $s'>t$ such that $r-\epsilon < s' < s$. Then $\pi(y,s')$ is strictly in between $\pi(y,t)$ and $\pi(x,s)$ and the image of the interval between $\pi(y,s')$ and $\pi(x,s)$ is $[s',s]$, which is contained in $(r-\epsilon, r+ \epsilon)$. If $\pi(y,t) > \pi(x,s)$ , then $\pi(x,t)=\pi(y,t)$. Pick $s' < t$ such that $s < s' < r + \epsilon$. Then $\pi(x,s')$ is strictly in between $\pi(x,s)$ and $\pi(y,t)$, and the image of the interval between $\pi(x,s)$ and $\pi(x,s')$ is $[s. s']$, which is contained in $(r-\epsilon, r+\epsilon)$.
\end{proof}

\begin{lemma}\label{lem:upper_map}
For each $x$ in $X$, the map $\psi_x: [f(x), \infty) \to T_f X$, $r \mapsto \pi(x,r)$ is continuous with respect to the interval topology.
\end{lemma}
\begin{proof}
Let $U$ be an open set in the interval topology. Assume $\psi_x(r)=\pi(x,r)$ is in $U$. The result follows from the comparability of $\pi(x,r)$ with $\pi(x, f(x))$ and $\pi(x, r+1)$.
\end{proof}

\begin{proof}[Proof of Proposition \ref{prop:interval_topology}]
Let $C_1,\dots,C_n$ be the upper sets of leaves in $T_f X$. In both topologies, they give a finite closed cover by Corollary \ref{cor:tree_poset} and Lemma \ref{lem:upper_set}. This implies that a set $U$ is open in these topologies if and only if its intersection with each $C_i$'s is relatively open. Hence,  it is enough to show that the relative topologies on the $C_i$'s coincide for the quotient and the interval topologies. Both the height map $\pi(x,r) \mapsto r$  of Lemma \ref{lem:height} and $\psi_x$ of Lemma \ref{lem:upper_map} are continuous with respect to both topologies. If  we let $\pi(x,f(x))$ denote the leaf of $C_i$, and restrict $h$ to $C_i$, then these maps becomes inverses of each other. This completes the proof.
\end{proof}

We now return to the proof of Proposition \ref{prop:merge_tree_metrics} from the main text, which says that the topology induced by the $\ell^p$-metrics on a merge tree agrees with the quotient space topology.

\begin{proof}[Proof of Proposition \ref{prop:merge_tree_metrics}]
Using the logic from the proof of Proposition \ref{prop:interval_topology}, it suffices to show that the relative topologies on the $C_i$'s with respect to the interval and metric topologies coincide. Indeed, the intersection of an $\ell^p$-ball with each upper set $C_j$ of a leaf of the merge tree is relatively open with respect to the interval topology. This is clear, since this intersection is homeomorphic to an open interval via the height map. Similarly, relatively open sets in the interval topology are relatively open in the metric topology.
\end{proof}

\section{Comparison and Existence of Merge Trees}\label{sec:existence_of_merge_tree}

In this paper we used two primary perspectives on merge trees: the classical merge tree (Definition~\ref{defn:merge-tree-as-reeb-graph}), which defines the merge tree as the Reeb graph of the epigraph, and the generalized merge tree (Definition~\ref{defn:generalized-MT}), which is defined as the display poset of the persistent set.
Each approach has its own advantages, but in order to ensure that the Reeb graph of a function is actually a graph requires specifying conditions on $f:X\to\R$ such as the Morse condition or piecewise linearity~\cite[p.~857]{vds}. 
In this section, we describe a third way of constructing the merge tree following the metric tree construction in~\cite{memoli2018metric} and show that when $X$ is locally path connected, the classical construction coincides with the metric construction described here.

For this section we assume that $X$ is a compact path-connected topological space and $f:X \to \R$ is a continuous function. We are using the definition of metric tree given in \cite[Section 3.4]{bowditch1991notes}, which is also called real tree or $\mathbb{R}$-tree.

\begin{remark}\label{rem:mergeTreeEssence}
As the quotient of the epigraph of $f$, the merge tree is the wedge of the image of the graph of $f$ and the half real line. Let us denote the image of the graph of $f: X \to \R$ inside the merge tree by $M_fX$. 
\end{remark}

\begin{definition}[Metric Merge Tree]
Let $X$ be a compact path connected topological space and $f:X \to \R$ be a continuous function. Let us define $m_f: X \times X \to \R$ by
$$m_f(x,y):=\inf\{\max f \circ \gamma | \gamma:[a,b] \to X, \gamma(a)=x, \gamma(b)=y \}. $$
Define $t_f: X \times X \to \R$ by
$$t_f(x,y)=2m_f(x,y)-f(x)-f(y). $$
It is easy to see that $t_f$ is a pseudo-metric (symmetric, non-negative and satisfies triangle inequality). Let us denote the associated metric space by $(\mathcal{T}_fX,t_f)$ and call it \textit{metric merge tree}. The function $f$ is still well defined on $\mathcal{T}_f X$.
\end{definition}

\begin{proposition}\label{prop:realtree} Let $X$ be a compact topological space and $f:X \to \R$ be a continuous function. Then
$(\mathcal{T}_fX,t_f)$ is a metric tree.
\end{proposition}

\begin{proof}
It is enough to show that $t_f$ is a path metric and it  has hyperbolicity zero (see \cite[Proposition 3.4.2]{bowditch1991notes}). To show that $t_f$ is a path metric, it is enough to show that for each $x,y$ i $X$ and $\epsilon>0$, there exists $z$ in $X$ such that 
$$t_f(x,z) \leq \frac{t_f(x,y)}{2}+\epsilon, \qquad t_f(y,z) \leq \frac{t_f(x,y)}{2}+\epsilon$$ 
(see \cite[Theorem 2.4.16]{burago2001course}). Take $x,y$ in $X$ and $\epsilon>0$. Let $\gamma:[0,1] \to X$ be a continuous curve from $x$ to $y$ such that $$m_f(x,y) \geq \max f \circ \gamma \epsilon.$$
Let $M:=\max f\circ \gamma$. Define
$$t_0:=\min \{t: f(\gamma(t))=M \}, \qquad t_1:=\max \{t:f(\gamma(t))=M \}. $$
Let $d_0:=M-f(x)$ and $d_1:=M-f(y)$. Without loss of generality $d_0 >= d_1$. Let $r:=(d_0+d_1)/2$ and 
$$s:=\min \{t: f\circ \gamma(t)=f(x)+r \} .$$
Let $z=\gamma(s)$. Now, we have:
\begin{align*}
    t_f(x,y) &\geq 2(M-\epsilon) - f(x) -f(y)=d_0+d_1-2\epsilon \\
    t_f(x,z) &= r = \frac{d_0+d_1}{2}\\
    t_f(y,z) &\leq (M-f(x)-r)+d_1=\frac{d_0+d_1}{2}.
\end{align*}

Now let us show that $t_f$ has $0$-hyperbolicity. Let $p$ be a point where $f$ takes its maximum. The Gromov product $g_p: X \times X \to \R$ with respect to the pseudo-metric $t_f$ is defined by:
$$g_p(x,y)=\frac{1}{2}(t_f(p,x)+t_f(p,y)-t_f(x,y)). $$
The metric $t_f$ has zero hyperbolicty if $$g_p(x,z) \geq \min(g_p(x,y),g_p(y,z))$$ for all $x,y,z$ in $X$ (see \cite[Corollary 1.1.B]{gromov1987hyperbolic}).

\textbf{Claim 1}: $g_p(x,y)=f(p)-m_f(x,y)$ for all $x,y$ in $X$:

By maximality of $p$, $m_f(p,x)=f(p)$. Hence, 
\begin{align*}
    t_f(p,x) = 2m_f(p,x)-f(p)-f(x) = f(p)-f(x).
\end{align*}
So we have
\begin{align*}
    g_p(x,y) &=\frac{1}{2}(t_f(p,x)+t_f(p,y)-t_f(x,y)) \\
            &=\frac{1}{2}(f(p)-f(x)+f(p)-f(y)-2m_f(x,y)+f(x)+f(y) ) \\
            &=f(p)-m_f(x,y)
\end{align*}

\textbf{Claim 2}: $g_p(x,z) \geq \min(g_p(x,y),g_p(y,z))$ for all $x,y,z$ in $X$.

By its definition, $m_f(x,z) \leq \max(m_f(x,y),m_f(y,z))$. Claim 2 follows from this and Claim 1.

Therefore  the hyperbolicity of $\mathcal{T}_fX$ is $0$.
\end{proof}

\begin{theorem}\label{thm:mergeTreeEq}
If $X$ is locally path connected, then $M_fX$ is the underlying topological space of $\mathcal{T}_fX$.
\end{theorem}
\begin{proof}
It is enough to show that $t_f(x,y) = 0$ if and only if $(x,f(x))$ and $(y,f(y))$ are identified under the Reeb quotient map $\pi$.

Assume $t_f(x,y)=0$. Then, $m_f(x,y)=f(x)=f(y)$. So, for each $t>f(x)=f(y)$, there is a path in the sublevel set $X^{\leq t}$ connecting $x,y$. Therefore, $\pi(x,t) = \pi(y,t)$, and by Lemma \ref{lem:poset_characterization}, $\pi(x,r) = \pi(y,r)$. 

Assume $\pi(x,f(x))=\pi(y,f(y))$. Let $r=f(x)$ and $t > r$. Since $x$ and $y$ are in the same connected component of the (closed) sublevel set $X^{\leq r}$ and $X$ is locally path connected, $x,y$ are in the same path component of the (open) sublevel set $X^{<t}$. This implies that $m_f(x,y) < t$. Since $t>r$ arbitrary, $m_f(x,y) = r = f(x) = f(y)$, so $t_f(x,y) = 0$.

\end{proof}

\section{Technical Proofs}\label{sec:technical_proofs}

This section contains proofs from the main body of the paper which are technical, but essentially straightforward. These results are focused on barcodes and bottleneck distances. To keep the exposition clean, we deal with barcodes satisfying the simplifying assumptions:
\begin{itemize}
    \item the barcode is a  \emph{set} (all multiplicities of intervals are equal to one);
    \item the barcode is finite; 
    \item each interval in the barcode is a half open interval of the form $[b,d)$, where $0 \leq b \leq d \leq \infty$.
\end{itemize}

\subsection{Proof of Proposition \ref{prop:decorated_bottleneck_reformulation}}\label{sec:technical_proof_decorated_bottleneck_reformulation}

To prove the proposition, we will introduce the notion of truncating a barcode.

\begin{definition}\label{def:truncation}
Let $B$ be a barcode and let $I = [b,d) \in B$. The \define{truncation of $I$} at height $h$, denoted $\mathrm{trunc}_h(I)$ is
\begin{itemize}
    \item equal to $I$ if $h \leq b$;
    \item the interval $[h,d)$ if $b \leq h \leq d$;
    \item the empty interval if $h > d$.
\end{itemize}
The \define{truncation of $B$} at height $h$, $\mathrm{trunc}_h(B)$, is the barcode obtained by truncating all of the intervals of $B$ at height $h$. 
\end{definition}

\begin{lemma}\label{lem:truncation}
Let $B$ and $B'$ be barcodes whose intervals $[b,d)$ all satisfy $b \geq H$ for some constant $H \geq 0$. Suppose that there exists a $\delta$-matching between $B$ and $B'$. For $\epsilon \leq \delta$, there is a $\delta$-matching between $B$ and $\mathrm{trunc}_{H+\e}(B')$. 
\end{lemma}

\begin{proof}
Let $\xi$ be a $\delta$-matching of $B$ and $B'$. Define a matching $\widehat{\xi}$ of $B$ and $\mathrm{trunc}_{H+\epsilon}(B')$ by setting 
\begin{align*}
\mathrm{dom}(\widehat{\xi}) &:= \{I \in \mathrm{dom}(\xi) \mid \mathrm{trunc}_{H+\e}(\xi(I)) \neq \emptyset\}\\
\mathrm{ran}(\widehat{\xi}) &:= \{I' \in \mathrm{ran}(\xi) \mid \mathrm{trunc}_{H+\e}(I') \neq \emptyset\}, \mbox{ and } \\
\widehat{\xi}(I) &:= \mathrm{trunc}_{H+\e}(\xi(I)).
\end{align*}
Let $I = [b,d) \in \mathrm{dom}(\widehat{\xi})$ with $I' = [b',d') = \xi(I)$. Then $b' \leq H + \e \leq d'$ and $\mathrm{trunc}_{H+\e}(I') = [H+\e,d')$. We have $|d-d'| \leq \delta$, by the assumption that $\xi$ was a $\delta$-matching. Moreover,
\[
\delta \geq b-b' \geq b - (H + \e)
\]
and
\[
\delta \geq \e \geq \e + (H-b) = (H+\e) - b
\]
imply that $|b - (H+\e)| \leq \delta$, so that the cost of matching $I$ with $\mathrm{trunc}_{H+\e}(I')$ is less than $\delta$. On the other hand, if $I \not \in \mathrm{dom}(\widehat{\xi})$ then either $I \not \in \mathrm{dom}(\xi)$, in which case we are done, or $\mathrm{trunc}_{H+\e}(\xi(I)) = \emptyset$. In the latter case, let $I = [b,d)$ and $I' = [b',d') = \xi(I)$. We have 
\[
d-b \leq d' + \delta - b \leq d' + \delta - H \leq H + \e + \delta - H \leq 2\delta,
\]
hence $\|[b,d)\|_\Delta \leq \delta$. Similar arguments handle intervals $I' \not \in \mathrm{ran}(\widehat{\xi})$. 
\end{proof}

\begin{proof}[Proof of Proposition \ref{prop:decorated_bottleneck_reformulation}]
Let 
\[
\widehat{d}_B(\cB_F,\cB_G) := \inf \{\epsilon \geq 0 \mid \mbox{$\exists$ $\e$-matching of $\mathcal{B}_F$ and $\mathcal{B}_G$}\},
\]
so that our goal is to show $d_B = \widehat{d}_B$. Clearly $d_B \leq \widehat{d}_B$, since the former infimizes over a larger set of matchings than the latter. To see the reverse inequality, let $\Phi$, $\Psi$ define an $(\e,\delta)$-matching. If $\e \geq \delta$, then each of the $\delta$-matchings between $\cB_F(p)$ and $\cB_G(\Phi(p))$ and between $\cB_G(q)$ and $\cB_F(\Psi(q))$ is, in particular, an $\e$-matching. This implies the existence of an $\e$-matching between $\cB_F$ and $\cB_G$. Finally, suppose that $\delta \geq \e$. Let $\widehat{\Phi}$ and $\widehat{\Psi}$ denote the $\delta$-interleaving maps obtained by composing $\Phi$ and $\Psi$ with flows in their respective target merge trees. We claim that there exist $\delta$-matchings between all pairs $\cB_F(p)$ and $\cB_G(\widehat{\Phi}(p))$ and between $\cB_G(q)$ and $\cB_F(\widehat{\Psi}(q))$. Indeed, setting $H = \pi_F(p)$, with $\pi_F:\cM_F \to \R$ denoting projection, we have that
\[
\cB_G(\widehat{\Phi}(p)) = \mathrm{trunc}_{H + \delta - \e}(\cB_G(\Phi(p))).
\]
Lemma \ref{lem:truncation} therefore implies that there exists a $\delta$-matching of $\cB_F(p)$ and $\cB_G(\widehat{\Phi}(p))$. The existence of a $\delta$-matching of $\cB_G(q)$ and $\cB_F(\widehat{\Psi}(q))$ follows for the same reason.
\end{proof}

We now have the tools at hand to prove Proposition \ref{prop::determinedleaves}.

\begin{proof}[Proof of Proposition \ref{prop::determinedleaves}.]
Let $p\in \cM_F$ be an arbitrary point. Choose a leaf $v$ of $\cM_F$ with $v \cle p$ and a leaf $w \cle \Phi(p)$. Let $\pi_F:\cM_F \to \R$ be the projection map. Then the barcodes at these points can be expressed as 
\begin{align*}
\cB_F(p) &= \mathrm{trunc}_{\pi_F(p)}(\cB_F(v)), and \\
\cB_G(\Phi(p)) &= \mathrm{trunc}_{\pi_F(p) + \e}(\cB_G(w)).
\end{align*}
It is straightforward to check that a $\delta$-matching of $\cB_F(p)$ and $\cB_F(\Phi(p))$ induces a $\delta$-matching of the truncated barcodes. A similar argument applies to pairs $q \in \cM_G$ and $\Psi(q) \in \cM_F$.
\end{proof}

\subsection{Proof of Proposition \ref{prop:disjointness}}\label{sec:technical_proofs_disjointness}

    Let $B = \{I_j = [b_j,d_j)\}_{j=1}^N$ be the degree-$k$ barcode for $F$. To ease exposition, assume that the births in the barcode are distinct---this assumption is easily removed at the cost of necessitating more involved notation. Suppose that the indices have been chosen so that $b_1 < b_2 < \cdots < b_N$. For each bar $I_j$, we choose a representative cycle $c_j$ which generates the persistent homology class represented by $I_j$; in particular, choose $c_j$ so that all of its simplices are contained in the connected component of the birth simplex for $I_j$. Let $[c_j]$ denote the homology class of $c_j$. Let $p_j \in \cM_F$ denote the birth point of $I_j$.
    
    We can use these representative cycles to build bases for various homology vector spaces which agree with the barcode decomposition. In the following, let $\iota_{b,d}$ denote the inclusion map $F(b) \hookrightarrow F(d)$ for each $b < d$. By an abuse of notation, we also let $\iota_{b,d}$ denote the induced map on homology $H_k(F(b)) \to H_k(F(d))$. We now construct our bases:
    \begin{enumerate}
        \item $\{v_1 := [c_1]\}$ is a basis for $H_k(F(b_1))$,
        \item $\{[c_2],\iota_{b_1,b_2}(v_1)\}$ is a spanning set for $H_k(F(b_2))$ and is linearly independent if 
        \[
        \iota_{b_1,b_2}(v_1) \neq 0.
        \]
        Moreover, we can choose $\lambda_1^{(2)} \in \Bbbk$ so that
        \[
        v_2 := [c_2] + \lambda_1^{(2)} \iota_{b_1,b_2}(v_1)
        \]
        has the property that 
        \[
        \iota_{b_2,d} (v_2) \not \in \mathrm{span} \{\iota_{b_1,d} (v_1)\}
        \]
        for any $d < d_2$. In particular, $v_2 \not \in \mathrm{ker}(\iota_{b_2,d})$ for $d < d_2$.
        \item Suppose that we have defined $v_1,\ldots,v_{j-1}$. Then 
        \[
        \{[c_j],\iota_{b_{j-1},b_j}(v_{j-1}),\ldots,\iota_{b_1,b_j}(v_1)\}
        \]
        is a spanning set for $H_k(F(b_j))$ which is linearly independent once any zero vectors have been removed. We define
        \[
        v_j := [c_j] + \lambda^{(j)}_{j-1} \iota_{b_{j-1},b_j} (v_{j-1}) + \cdots + \lambda_1^{(j)} \iota_{b_1,b_{j-1}} (v_1),
        \]
        where coefficients $\lambda^{(j)}_i$ are chosen so that 
        \begin{equation}\label{eqn:inclusion_condition}
        \iota_{b_j,d} (v_j) \not \in \mathrm{span}\{\iota_{b_{j-1},d} (v_{j-1}),\ldots,\iota_{b_1,d} (v_1) \}
        \end{equation}
        for all $d < d_j$. In particular, $v_j \not \in \mathrm{ker}(\iota_{b_j,d})$ for $d < d_j$.
    \end{enumerate}
The lifespan of the vector $v_j$ defined via this process represents the persistent homology bar $I_j$. 

To prove that $\widehat{\cF} \approx \cF$, we will show that $\cF$ is real interval decomposable. To do so, it suffices to show that each of these $v_j$ can be chosen to be a linear combination of classes with cycle representatives that lie in the same connected component as the birth simplex of $I_j$. This will be achieved via induction. 

For the base case of the induction argument, we have $v_1 = [c_1]$, and the claim follows immediately. Suppose that $v_1,\ldots,v_{j-1}$ have the desired property and consider $v_j$. We sort the $v_1,\ldots,v_{j-1}$, into a collection $v_{r_1},\ldots,v_{r_m}$ of vectors corresponding to bars in $B$ which are born in the same component as $c_j$ (i.e., bars whose birth points satisfy $p_{r_1},\ldots,p_{r_m} \cle p_j$) and a collection $v_{s_1},\ldots,v_{s_n}$ of vectors which do not have this property, so that
\begin{align*}
v_j &= [c_j] + \lambda_{r_1}^{(j)} \iota_{b_{r_1},b_j} (v_{r_1}) + \cdots \\
&\qquad + \lambda_{r_m}^{(j)} \iota_{b_{r_m},b_j} (v_{r_m}) + \lambda_{s_1}^{(j)} \iota_{b_{s_1},b_j} (v_{s_1}) + \cdots + \lambda_{s_n}^{(j)} \iota_{b_{s_n},b_j} (v_{s_n}).
\end{align*}
We claim that taking $\lambda_{s_i}^{(j)} = 0$ for all $i = 1,\ldots,n$ results in a valid vector $v_j$; i.e., that this yields a $v_j$ satisfying \eqref{eqn:inclusion_condition}. Indeed, let $d < d_j$ and suppose that 
\begin{align}
&\iota_{b_j,d} \left([c_j] + \lambda_{r_1}^{(j)} \iota_{b_{r_1},b_j} (v_{r_1}) + \cdots + \lambda_{r_m}^{(j)} \iota_{b_{r_m},b_j} (v_{r_m}) \right) \nonumber \\
&\hspace{2in} = \mu_1 \iota_{b_j,d}(v_{s_1}) + \cdots + \mu_m \iota_{b_j,d}(v_{s_m}) \label{eqn:image_in_span}
\end{align}
for some coefficients $\mu_i$. We first note that there must be some $v_{s_i} \not \in \mathrm{ker}(\iota_{s_i,d})$---otherwise the right-hand side of \eqref{eqn:image_in_span} is zero, hence left-hand side is as well and we arrive at the contradiction that $\iota_{b_j,d}(v_j) = 0$. We therefore have a well defined number
\[
d' := \min \{\mathrm{merge}(p_j,p_{s_i}) \mid \iota_{b_{s_i},d}(v_{s_i}) \neq 0\}.
\]
By the $H_k$-disjointness assumption, either $d_j < d'$ or $d_{s_i} < d'$ for all $s_i$. The former case implies $d < d'$. By the induction hypothesis, there is a cycle representation of the left-hand side of \eqref{eqn:image_in_span} whose vertices all belong to the same connected component as those of $\iota_{b_j,d}(c_j)$ and a cycle representation of the right-hand side whose vertices belong to a different connected component than $c_j$. Moreover, these cycle representatives are homologous, and we have arrived at a contradiction. On the other hand, if $d_{s_i} < d'$ for all $s_i$ then it must be that $d < d_{s_i}$ for some $s_i$ (since there exists some $v_{s_i} \not \in \mathrm{ker}(\iota_{b_{s_i},d})$) and a similar argument yields a contradiction. This completes the proof. \qed

\end{document}